\RequirePackage{fix-cm}
\documentclass[envcountsame,final]{svjour3}
\pdfoutput=1
\smartqed
\usepackage[T1]{fontenc}
\usepackage{times} 
  
\usepackage{amsmath,amssymb,amsfonts}
\numberwithin{equation}{section}
\usepackage{graphicx,verbatim,mathtools}
\usepackage{listings,color}
\usepackage{tikz,setspace}
\usetikzlibrary{positioning}
\usepackage{pgfplots}
\usepackage[width=12.8cm,asymmetric]{geometry}
\usepackage[only,llbracket,rrbracket,llparenthesis,rrparenthesis]{stmaryrd}
\newcommand{\lVERT}{\lvert\kern-0.25ex\lvert\kern-0.25ex\lvert}
\newcommand{\rVERT}{\rvert\kern-0.25ex\rvert\kern-0.25ex\rvert}

\DeclareMathSymbol{\Gamma}{\mathalpha}{letters}{"00}
\DeclareMathSymbol{\Delta}{\mathalpha}{letters}{"01}
\DeclareMathSymbol{\Theta}{\mathalpha}{letters}{"02}
\DeclareMathSymbol{\Lambda}{\mathalpha}{letters}{"03}
\DeclareMathSymbol{\Xi}{\mathalpha}{letters}{"04}
\DeclareMathSymbol{\Pi}{\mathalpha}{letters}{"05}
\DeclareMathSymbol{\Sigma}{\mathalpha}{letters}{"06}
\DeclareMathSymbol{\Upsilon}{\mathalpha}{letters}{"07}
\DeclareMathSymbol{\Phi}{\mathalpha}{letters}{"08}
\DeclareMathSymbol{\Psi}{\mathalpha}{letters}{"09}
\DeclareMathSymbol{\Omega}{\mathalpha}{letters}{"0A}
\DeclareMathSymbol{\varGamma}{\mathalpha}{operators}{"00}
\DeclareMathSymbol{\varDelta}{\mathalpha}{operators}{"01}
\DeclareMathSymbol{\varTheta}{\mathalpha}{operators}{"02}
\DeclareMathSymbol{\varLambda}{\mathalpha}{operators}{"03}
\DeclareMathSymbol{\varXi}{\mathalpha}{operators}{"04}
\DeclareMathSymbol{\varPi}{\mathalpha}{operators}{"05}
\DeclareMathSymbol{\varSigma}{\mathalpha}{operators}{"06}
\DeclareMathSymbol{\varUpsilon}{\mathalpha}{operators}{"07}
\DeclareMathSymbol{\varPhi}{\mathalpha}{operators}{"08}
\DeclareMathSymbol{\varPsi}{\mathalpha}{operators}{"09}
\DeclareMathSymbol{\varOmega}{\mathalpha}{operators}{"0A}

\newcommand{\allmodesymb}[2]{\relax\ifmmode{\mathchoice
{\mbox{\fontsize{\tf@size}{\tf@size}#1{#2}}}
{\mbox{\fontsize{\tf@size}{\tf@size}#1{#2}}}
{\mbox{\fontsize{\sf@size}{\sf@size}#1{#2}}}
{\mbox{\fontsize{\ssf@size}{\ssf@size}#1{#2}}}}
\else
\mbox{#1{#2}}\fi}

\renewcommand{\Delta}{\varDelta}
\renewcommand{\Lambda}{\varLambda}
\renewcommand{\Omega}{\varOmega}
\renewcommand{\Pi}{\varPi}
\renewcommand{\Phi}{\varPhi}

\renewcommand{\a}{\alpha}
\newcommand{\bo}[1]{\mathbf{#1}} 
\renewcommand{\rm}[1]{\mathrm{#1}} 
\newcommand{\llb}{\llbracket}
\newcommand{\rrb}{\rrbracket}
\newcommand{\lla}{\lbrace}
\newcommand{\rra}{\rbrace}
\newcommand{\llc}{\left\langle}
\newcommand{\rrc}[1][n]{\right\rangle_{#1}}
\newcommand{\lld}{\llparenthesis}
\newcommand{\rrd}[1][n]{\rrparenthesis_{#1}}
\renewcommand{\d}{\rm{d}}
\newcommand{\e}{\rm{e}}
\newcommand{\im}{\rm{i}}
\newcommand{\eps}{\varepsilon}
\newcommand{\abs}[1]{\lvert#1\rvert} 
\newcommand{\tends}{\rightarrow}

\newcommand{\norm}[1]{\lVert#1\rVert} 
\newcommand{\p}{\partial}
\newcommand{\half}{\tfrac{1}{2}}

\DeclareMathOperator{\diam}{diam}

\DeclareMathOperator{\card}{card}
\DeclareMathOperator{\Div}{div}
\DeclareMathOperator{\Trace}{Tr}
\DeclareMathOperator{\Dim}{dim}
\DeclareMathOperator{\DivT}{\Div_{\rm T}}
\DeclareMathOperator{\nablaT}{\nabla_{\rm T}}

\DeclareMathOperator{\supp}{supp}
\newcommand{\eval}[2]{ #1\rvert_{#2}}
\newcommand{\pair}[2]{\langle #1,#2 \rangle}
\renewcommand{\dim}{d}
\newcommand{\Rd}{\mathbb{R}^{\dim}}
\newcommand{\R}{\mathbb{R}}
\newcommand{\N}{\mathbb{N}}
\newcommand{\Z}{\mathbb{Z}}

\newcommand{\C}{\mathbb{C}}
\newcommand{\calA}{\mathcal{A}}
\newcommand{\calF}{\mathcal{F}}
\newcommand{\calT}{\mathcal{T}}

\newcommand{\calP}{\mathcal{P}}

\newcommand{\calQ}{\mathcal{Q}}
\newcommand{\It}{\mathcal{J}_\tau}
\newcommand{\Ld}{\Lambda}
\newcommand{\Om}{\Omega}
\newcommand{\DO}{\partial\Om}

\newcommand{\Ob}{\overline{\Omega}}

\newcommand{\Jo}{J_h}
\newcommand{\Blambda}{B_{h,*}}
\newcommand{\Bdgr}{B_{h,\theta}}
\newcommand{\Bdg}{B_{h,1/2}}
\newcommand{\normdgr}[1]{\norm{#1}_{h,\theta}}
\newcommand{\normah}[1]{\norm{#1}_{a_h}}
\newcommand{\cstab}{c_{\rm{s}}}
\newcommand{\norme}[1]{\norm{#1}_{h}}
\newcommand{\normx}[1]{{\lVERT#1\rVERT}_h}
\newcommand{\absJ}[1]{\abs{#1}_{\rm{J}}}
\newcommand{\normH}[1]{\norm{#1}_{H}}
\newcommand{\ch}{c_{\calT}}
\newcommand{\cp}{c_{\calP}}

\newcommand{\Adg}{A_{h}}
\newcommand{\Vh}{V_{h,\bo p}}
\newcommand{\Vt}{V^{\tau,\bo q}}
\newcommand{\Vth}{V_{h,\bo p}^{\tau, \bo q}}
\newcommand{\HIO}{H(I;\Om)}

\newcommand{\calFhib}{\calF_h^{i,b}}
\newcommand{\Fg}{F_\gamma}
\newcommand{\Ll}{L_\lambda}
\newcommand{\Lw}{L_\omega}
\newcommand{\Cdg}{C_h}
\newcommand{\SIn}{\sum_{n=1}^N \int_{I_n}}
\newcommand{\SIN}{\sum_{n=1}^N\int\limits_{I_n}}
\newcommand{\SK}{\sum_{K\in\calT_h}}
\newcommand{\CFh}{C^{\calF}_h}
\newcommand{\pj}{\Pi} 
\newcommand{\pmp}{\pj^{{\bo m, \bo p}}_h}
\newcommand{\php}{\pj^{\bo p}_h}
\newcommand{\prq}{\pj^{{\bo r, \bo q}}_{\tau}}
\newcommand{\ptau}{\pj^{r,q}_\tau}
\newcommand{\pq}{\pj^{\bo q}_{\tau}}
\newcommand{\hf}{\phi} 
\newcommand{\Xell}{X_\ell}
\newcommand{\dofx}{\mathrm{DoF}_x}
\newcommand{\doft}{\mathrm{DoF}_\tau}

\title{Discontinuous Galerkin finite element methods for time-dependent Hamilton--Jacobi--Bellman equations with Cordes coefficients}
\author{Iain~Smears \and Endre~S\"{u}li}
\institute{Mathematical Institute, University of Oxford, \email{smears@maths.ox.ac.uk}, \and  \email{suli@maths.ox.ac.uk}}
\date{\today}
\titlerunning{DGFEM FOR PARABOLIC HJB EQUATIONS}
\authorrunning{I.~SMEARS \& E.~S\"ULI}

\begin{document}

\maketitle 

\begin{abstract}
We propose and analyse a fully-discrete discontinuous Galerkin time-stepping method for parabolic Hamilton--Jacobi--Bellman equations with Cordes coefficients. The method is consistent and unconditionally stable on rather general unstructured meshes and time-partitions. Error bounds are obtained for both rough and regular solutions, and it is shown that for sufficiently smooth solutions, the method is arbitrarily high-order with optimal convergence rates with respect to the mesh size, time-interval length and temporal polynomial degree, and possibly suboptimal by an order and a half in the spatial polynomial degree. Numerical experiments on problems with strongly anisotropic diffusion coefficients and early-time singularities demonstrate the accuracy and computational efficiency of the method, with exponential convergence rates under combined $hp$- and $\tau q$-refinement.
\end{abstract}

\section{Introduction}\label{sec:introduction}

We consider the numerical analysis of the Cauchy--Dirichlet problem for Hamilton--Jacobi--Bellman (HJB) equations of the form
\begin{equation}\label{eq:HJB_intro}
\p_t w - \sup_{\a\in\Ld}[L^\a u - f^\a]=0\qquad \text{in }\Om\times I,
\end{equation}
where $\Om\subset \R^\dim$ is a bounded convex domain, $I=(0,T)$, $\Ld$ is a compact metric space, and where the $L^\a$ are nondivergence form elliptic operators given by
\begin{equation}\label{eq:Lalpha_intro}
L^\a v \coloneqq  a^\a:D^2 v + b^\a\cdot \nabla v - c^\a v,\qquad \a\in\Ld.
\end{equation}
HJB equations of the form \eqref{eq:HJB_intro} arise from problems of optimal control of stochastic processes over a finite-time horizon \cite{Fleming2006}. Note that the specific form of the HJB equation in \eqref{eq:HJB_intro} is obtained after reversing the time variable of the control problem, and thus it will be considered along with an initial-time Cauchy condition and a lateral Dirichlet boundary condition. 
\hyphenation{multi-dimensional}
We are interested in consistent, stable and high-order methods for multidimensional HJB equations with uniformly elliptic but possibly strongly anisotropic diffusion coefficients. Moreover, the results of this work are applicable to other forms of HJB equations, such as the case where the supremum is replaced by an infimum in \eqref{eq:HJB_intro}, and also to equations of Bellman--Isaacs type from stochastic differential games.

Monotone schemes, which conserve the maximum principle in the discrete setting, represent a significant class of numerical methods for \eqref{eq:HJB_intro} and are supported by a general convergence theory by Barles and Souganidis \cite{Barles1991}. Since the history and early literature of these methods is discussed for example in \cite{Fleming2006,Kushner1990} or in the introduction of \cite{Jensen2013}, we mention here only some recent developments. Building on earlier works such as \cite{Camilli1995,Crandall1996}, Debrabant and Jakobsen developed in \cite{Debrabant2013} a semi-Lagrangian framework for constructing wide-stencil monotone finite difference schemes for HJB and Bellman--Isaacs equations. Uniform convergence to the viscosity solution of monotone finite element methods was shown by Jensen and the first author in \cite{Jensen2013} through an extension of the Barles--Souganidis framework, along with strong convergence results in $L^2(H^1)$ under nondegeneracy assumptions.
 
An alternative approach to the numerical solution of HJB equations was proposed in \cite{Smears2013,Smears2014}, based on the Cordes condition which comes from the study of nondivergence form elliptic and parabolic equations with discontinuous coefficients \cite{Cordes1956,Maugeri2000}. The Cordes condition is an algebraic assumption on the coefficients of the operators $L^\a$; it is well-suited for numerical analysis since the techniques of analysis of the continuous problem can be extended to the discrete setting.
Moreover, HJB equations are connected to the Cordes condition through the fact that linearisations of the nonlinear operator are nondivergence form operators with discontinuous coefficients.
Unlike their divergence form counterparts, linear nondivergence form equations with discontinuous coefficients are generally ill-posed, even under uniform ellipticity or parabolicity conditions \cite{Gilbarg2001,Maugeri2000}; however, well-posedness is recovered under the Cordes condition \cite{Maugeri2000}.
In fact, as first shown in \cite{Smears2014}, the Cordes condition permits a straightforward proof of existence and uniqueness in $H^2$ of the solution of a fully nonlinear elliptic HJB equation.

The discretisation of linear nondivergence form elliptic equations by $hp$-version discontinuous Galerkin finite element methods (DGFEM) was first considered in \cite{Smears2013}.
There, the stability of the numerical method was achieved through the Cordes condition and the key ideas of testing the equation with $\Delta v_h$, where $v_h$ is a test function from the finite element space, and of weakly enforcing an important integration by parts identity connected to the Miranda--Talenti Inequality.
An $hp$-version DGFEM for elliptic HJB equations was then proposed in \cite{Smears2014} along with a full theoretical analysis in terms of consistency, stability and convergence. The accuracy and efficiency of the method was demonstrated through numerical experiments for a range of challenging problems, including boundary layers, corner singularities and strongly anisotropic diffusion coefficients.

This work extends our previous results to parabolic HJB equations by combining the spatial discretisation of \cite{Smears2014} with a discontinuous Galerkin (DG) time-stepping scheme \cite{Thomee2006}. The resulting method is consistent, unconditionally stable and arbitrarily high-order, whilst permitting rather general unstructured meshes and time partitions.
Although other time-stepping schemes could be considered, Sch\"otzau and Schwab showed in \cite{Schotzau2000} that a key feature of DG time-stepping methods is the potential for exponential convergence rates, even for solutions with limited regularity; our numerical experiments below show that our method retains this quality.

In order to treat the nonlinearity of the HJB operator, the time-stepping scheme is nonstandard and leads to strong control of a discrete $H^1(L^2)\cap L^2(H^2)$-type norm.
The consistency and good stability properties of the resulting method lead to optimal convergence rates in terms of the mesh size $h$, time-interval length $\tau$, and temporal polynomial degrees $q$. The rates in the spatial polynomial degrees $p$ are possibly suboptimal by an order and a half, as is common for DGFEM that are stable in discrete $H^2$-norms, such as DGFEM for biharmonic equations \cite{Mozolevski2007}. In addition to error bounds for regular solutions, we use Cl\'ement-type projection operators to obtain bounds under very weak regularity assumptions that are in particular applicable to problems with early-time singularities induced by the initial datum.

The contributions of this paper are as follows. In section~\ref{sec:problem_analysis}, we define the problem under consideration and show its well-posedness. Then, in section~\ref{sec:temporal}, we introduce the essential ideas of the time-stepping scheme in a semidiscrete context and show its stability. Full discretisation in space and time is considered in sections~\ref{sec:fem} and \ref{sec:numerical_scheme}, where we show the method's consistency. Stability and well-posedness of the scheme are then obtained in section~\ref{sec:stability} and error bounds are derived in section~\ref{sec:error_analysis}. The results of numerical experiments are reported in section~\ref{sec:numexp}.

\section{Analysis of the problem}\label{sec:problem_analysis}
Let $\Om$ be a bounded convex polytopal open set in $\Rd$, $\dim\geq 2$, let $\Ld$ be a compact metric space, and let $I\coloneqq (0,T)$, with $T>0$. It is assumed that $\Om$ and $\Ld$ are non-empty. Convexity of $\Om$ implies that the boundary $\DO$ of $\Om$ is Lipschitz \cite{Grisvard2011}. Let the real-valued functions $a_{ij}$, $b_i$, $c$ and $f$ belong to $C\left(\Ob\times\overline{I}\times\Ld\right)$ for each $i,\,j \in \{1,\dots,\dim\}$. For each $\a\in\Ld$, define the functions $a_{ij}^\a\colon (x,t)\mapsto a_{ij}(x,t,\a)$, where $(x,t)\in\Ob\times\overline{I}$ and $i,\,j \in \{1,\dots,\dim\}$; the functions $b_i^\a$, $c^\a$ and $f^\a$ are similarly defined. We introduce the matrix functions $a^\a\coloneqq (a_{ij}^\a)$ and the vector functions $b^\a\coloneqq (b^\a_i)$ for notational convenience. The operators $L^\a \colon L^2(I;H^2(\Om))\tends L^2(I;L^2(\Om))$ are given by
\begin{equation}
L^\a v \coloneqq  a^\a:D^2 v + b^\a\cdot \nabla v - c^\a v, \quad v\in L^2(I;H^2(\Om)),\;\a\in\Ld,
\end{equation}
where $D^2 v$ denotes the Hessian matrix of $v$.
Compactness of $\Ld$ and continuity of the functions $a$, $b$, $c$ and $f$ imply that the fully nonlinear operator $F$, given by
\begin{equation}
F\colon v\mapsto F[v]\coloneqq \p_t v - \sup_{\a\in\Ld}\left[ L^\a v - f^\a \right] = \inf_{\a\in\Ld}\left[\p_t v - L^\a v + f^\a \right],
\end{equation}
is well-defined as a mapping from $\HIO\coloneqq L^2(I;H^2(\Om)\cap H^1_0(\Om))\cap H^1(I;L^2(\Om))$ into $L^2(I;L^2(\Om))$.
The problem considered is to find a function $u\in \HIO$ that is a strong solution of the parabolic HJB equation subject to  Cauchy--Dirichlet boundary conditions:
\begin{equation}\label{eq:HJB}
\begin{aligned}
F[u] &= 0 & &\text{in }\Om\times I,
\\ u &=0 & & \text{on }\DO\times I,
\\ u &=u_0 & &\text{on }\Om\times \{0\},
\end{aligned}
\end{equation}
where $u_0\in H^1_0(\Om)$. Note that the lateral condition $u=0$ on $\DO\times I$ is incorporated in the function space $\HIO$.
Well-posedness of \eqref{eq:HJB} is established in section~\ref{sec:well_posedness} under the following hypotheses. The function $c$ is nonnegative and there exist positive constants $\nu\leq \overline{\nu}$ such that
\begin{equation}\label{eq:uniform_ellipticity}
\nu \abs{\xi}^2 \leq \xi^\top  a^\a(x,t)\; \xi \leq \bar{\nu} \abs{\xi}^2 \quad\forall\,\xi\in\Rd,\;\forall\,(x,t)\in\Om\times I,\;\forall\,\a\in\Ld.
\end{equation}
We assume the Cordes condition \cite{Smears2013,Smears2014}: there exist $\eps\in (0,1]$, $\lambda>0$ and $\omega>0$ such that
\begin{equation}\label{eq:cordes_condition}
\frac{\abs{a^\a}^2 + 1/\lambda^2 + 1/\omega^2}{(\Trace a^\a + 1/\lambda + 1/\omega)^2} \leq \frac{1}{\dim+1+\eps} \quad\text{in }\Ob\times\overline{I},\,\;\forall\,\a\in\Ld,
\end{equation}
where $\abs{a^\a}$ denotes the Frobenius norm of the matrix $a^\a$.
In the special case where $b\equiv 0$ and $c \equiv 0$, we set $\lambda = 0 $ and assume that there exist $\eps\in (0,1]$ and $\omega >0$ such that
\begin{equation}\label{eq:cordes_condition2}
\frac{\abs{a^\a}^2 + 1/\omega^2}{\left( \Trace a^\a + 1/\omega \right)^2} \leq \frac{1}{\dim+\eps} \quad\text{in }\Ob\times\overline{I},\;\forall \,\a\in\Ld.
\end{equation}
As explained in \cite{Smears2014}, $\lambda$ and $\omega$ serve to make the Cordes condition invariant under rescaling of the spatial and temporal domains. In the case of elliptic equations in two dimensions without lower order terms, the Cordes condition is equivalent to uniform ellipticity.
Given \eqref{eq:cordes_condition}, by considering transformations of the unknown of the type $u = \mathrm{e}^{\mu t}\tilde{u}$, we can assume without loss of generality that
\begin{equation}\label{eq:cordes_condition3}
\frac{\abs{a^\a}^2 + \abs{b^\a}^2/2\lambda +  (c^\a/\lambda)^2 + 1/\omega^2}{(\Trace a^\a + c^\a/\lambda + 1/\omega)^2} \leq \frac{1}{\dim+1+\eps} \quad\text{in }\Ob\times\overline{I}\,\;\forall\,\a\in\Ld.
\end{equation}
The relevance of \eqref{eq:cordes_condition} is to show that the Cordes condition is essentially independent of the lower order terms $b^\a$ and $c^\a$, although it will be simpler to work with \eqref{eq:cordes_condition3}.
Define the strictly positive function $\gamma\colon \Om\times I\times \Ld\tends \R_{>0}$ by
\begin{equation}
\gamma(x,t,\a)\coloneqq \frac{\Trace a^\a(x,t) + c^\a/\lambda + 1/\omega}{\abs{a^\a(x,t)}^2 + \abs{b^\a}^2/2\lambda + (c^\a/\lambda)^2 +1/\omega^2}.
\end{equation}
In the case of $b\equiv 0$ and $c\equiv 0$, the function $\gamma$ is defined by
\begin{equation}
\gamma(x,t,\a) \coloneqq \frac{\Trace a^\a(x,t) + 1/\omega}{\abs{a^\a(x,t)}^2 +1/\omega^2}.
\end{equation}
Continuity of the data implies that $\gamma \in C(\Ob\times\overline{I}\times\Ld)$, and it follows from \eqref{eq:uniform_ellipticity} that there exists a positive constant $\gamma_0 >0$ such that $\gamma\geq \gamma_0$ on $\Ob\times\overline{I}\times\Ld$. For each $\a\in\Ld$, define $\gamma^\a\colon (x,t)\mapsto \gamma(x,t,\a)$, and define the operator $\Fg\colon \HIO\tends L^2(I;L^2(\Om))$ by
\begin{equation}
\Fg[v]\coloneqq \inf_{\a\in\Ld}\left[ \gamma^\a \left(\p_t v - L^\a v + f^\a\right)\right].
\end{equation}
For $\omega$ and $\lambda$ as in \eqref{eq:cordes_condition3}, we introduce the operators $\Ll$ and $\Lw$ defined by
\begin{align}\label{eq:Ll_Lw_operators}
\Ll v &\coloneqq \Delta v - \lambda v &  \Lw v &\coloneqq \omega \,\p_t v - \Ll v.
\end{align}
The following result is similar to \cite[Lemma~1]{Smears2014}, so the proof is omitted here.
\begin{lemma}\label{lem:cordes_bound}
Let $\Om$ be a bounded open subset of $\Rd$, let $I=(0,T)$, and suppose that \eqref{eq:cordes_condition3} holds, or that \eqref{eq:cordes_condition2} holds if $b\equiv 0$ and $c\equiv 0$. Let $U\subset \Om$ be an open set, let $J\subset I$ be an open interval, and let the functions $u$, $v\in L^2(J;H^2(U))\cap H^1(J;L^2(U))$, and set $w\coloneqq u-v$. Then, the following inequality holds a.e.\ in $U$, for a.e.\ $t\in J$:
\begin{equation}
\abs{\Fg[u]-\Fg[v]-\Lw w} \leq \sqrt{1-\eps}\left(\omega^2 \abs{\p_t w}^2+\abs{D^2 w}^2 + 2\lambda \abs{\nabla w}^2 + \lambda^2 \abs{w}^2\right)^{1/2},
\end{equation}
with $\lambda =0$ if $b\equiv 0$ and $c\equiv 0$.
\end{lemma}

In the following analysis, we shall write $a \lesssim b$ for $a,b \in \R$ to signify that there exists a constant $C$ such that $a \leq C\,b$, where $C$ is independent of discretisation parameters such as the element sizes of the meshes and the polynomial degrees of the finite element spaces used below, but otherwise possibly dependent on other fixed quantities, such as, for example, the constants in \eqref{eq:uniform_ellipticity} and \eqref{eq:cordes_condition} or the shape-regularity parameters of the mesh.
\subsection{Well-posedness}\label{sec:well_posedness}
For a bounded convex domain $\Om\subset\R^\dim$, the Miranda--Talenti Inequality \cite{Grisvard2011,Maugeri2000} states that $\abs{v}_{H^2(\Om)}\leq \norm{\Delta v}_{L^2(\Om)}$ for all $v\in H^2(\Om)\cap H^1_0(\Om)$. Along with the Poincar\'{e} Inequality, it implies that $H\coloneqq H^2(\Om)\cap H^1_0(\Om)$ is a Hilbert space when equipped with the inner-product
$ \pair{u}{v}_{\Delta}\coloneqq \pair{\Ll u}{\Ll v}_{L^2(\Om)}$, where $\Ll$ is from \eqref{eq:Ll_Lw_operators} and $\lambda\geq 0$ is from \eqref{eq:cordes_condition3}. It is possible to identify $H^*$, the dual space of $H$, with $L^2(\Om)$ through the duality pairing 
\begin{equation}\label{eq:h2_l2_duality_pairing}
\begin{aligned}
&\pair{f}{v}_{L^2\times H}\coloneqq \int_\Om  f (-\Ll v)\, \d x, & &f\in L^2(\Om),\; v\in H.
\end{aligned}
\end{equation}
Indeed, we clearly have $L^2(\Om) \hookrightarrow H^*$, and $H^2$-regularity of solutions of Poisson's equation in convex domains \cite{Grisvard2011} shows that this embedding is an isometry: for any $f\in L^2(\Om)$, we have $\norm{f}_{L^2(\Om)} = \norm{f}_{H^*}$. If $\varphi\in H^*$, then the Riesz Representation Theorem implies that there is a unique $w\in H$ such that $\pair{w}{v}_{\Delta}=\varphi(v)$ for all $v\in H$. Then $f=-\Ll w\in L^2(\Om)$ satisfies $\pair{f}{v}_{L^2\times H}=\varphi(v)$ for all $v\in H$.

The space $H^1_0(\Om)$ may be equipped with the inner-product $\pair{u}{v}_{H^1_0}\coloneqq\int_\Om\nabla u\cdot \nabla v + \lambda uv\, \d x$ with associated norm $\norm{\cdot}_{H^1_0}$; we note that the Poincar\'e Inequality implies positive definiteness of $\pair{\cdot}{\cdot}_{H^1_0}$ in the case of $\lambda=0$.

The relevance of these choices of duality pairing and inner-products is that the spaces $H$, $H^1_0(\Om)$ and $L^2(\Om)$ form a Gelfand triple as a result of the following integration by parts identity: for any $w\in H^1_0(\Om)$ and $v\in H$, we have
\begin{equation}
\pair{w}{v}_{L^2\times H} =\int_{\Om} w (-\Ll v)\,\d x = \int_{\Om} \nabla w\cdot \nabla v + \lambda wv \, \d x = \pair{w}{v}_{H^1_0}.
\end{equation}
Recall that $\HIO\coloneqq L^2(I;H^2(\Om)\cap H^1_0(\Om))\cap H^1(I;L^2(\Om))$. The general theory of Bochner spaces, see for instance \cite{Wloka1987}, yields the following result.
\begin{lemma}\label{lem:gelfand_triple}
Let $\Om\subset \Rd$ be a bounded convex domain and let $I=(0,T)$. Then,
\begin{equation*}
H=H^2(\Om)\cap H^1_0(\Om) \hookrightarrow H^1_0(\Om) \hookrightarrow L^2(\Om)
\end{equation*}
form a Gelfand triple \emph{\cite{Wloka1987}} under the inner product $\pair{\cdot}{\cdot}_{H^1_0}$ and the duality pairing $\pair{\cdot}{\cdot}_{L^2\times H}$.
The space $\HIO$ is continuously embedded in $C(\overline{I};H^1_0(\Om))$, and for every $u$, $v\in \HIO$ and any $t\in \overline{I}$, we have
\begin{equation}\label{eq:bochner}
\pair{u(t)}{v(t)}_{H^1_0} = 
\pair{u(0)}{v(0)}_{H^1_0}
+ \int_0^t \pair{\p_t u}{v}_{L^2\times H} + \pair{\p_t v}{u}_{L^2\times H} \,\d s.
\end{equation}
\end{lemma}
Define the norms $\normH{\cdot}$ on $H$ and $\norm{\cdot}_{\HIO}$ on $\HIO$ by
\begin{align}
&\normH{v}^2 \coloneqq \abs{v}^2_{H^2(\Om)} + 2\lambda \abs{v}_{H^1(\Om)}^2 + \lambda^2 \norm{v}_{L^2(\Om)}^2, & & v \in H,\\
&\norm{v}_{\HIO}^2 \coloneqq \int_0^T \omega^2 \norm{\p_t v}_{L^2(\Om)}^2 + \normH{v}^2\,\d t,& & v\in \HIO.\label{eq:HIO_norm}
\end{align}
We will make use of the following solvability result for the Cauchy--Dirichlet problem associated to the linear operator $\Lw$ from \eqref{eq:Ll_Lw_operators}.
\begin{theorem}\label{thm:heat_equation}
Let $\Om\subset \Rd$ be a bounded convex domain and let $I=(0,T)$. For each $g\in L^2(I;L^2(\Om))$ and $v_0\in H^1_0(\Om)$, there exists a unique $v\in \HIO$ such that
\begin{equation}\label{eq:heat_equation}
\begin{aligned}
\Lw v &= g & &\text{\emph{ a.e.\ in }} \Om,\text{\emph{ for a.e.\ }} t\in I,
\\ v(0)&=v_0 & &\text{\emph{ in }} \Om.
\end{aligned}
\end{equation}
Moreover, the function $v$ satisfies
\begin{equation}\label{eq:heat_bound}
\norm{v}_{\HIO}^2  + \omega \norm{v(T)}_{H^1_0}^2 \leq  \norm{g}_{L^2(I;L^2(\Om))}^2 + \omega \norm{v_0}_{H^1_0}^2.
\end{equation}
\end{theorem}
In Theorem~\ref{thm:heat_equation}, well-posedness of \eqref{eq:heat_equation} is simply a special case of the general theory of Galerkin's method for parabolic equations, see \cite{Wloka1987}. The bound \eqref{eq:heat_bound} is obtained by combining \eqref{eq:bochner}, integration by parts and the Miranda--Talenti Inequality.
\begin{theorem}\label{thm:hjb_well_posed}
Let $\Om\subset \Rd$ be a bounded convex domain, let $I=(0,T)$, and let $\Ld$ be a compact metric space. Let the data $a$, $b$, $c$ and $f$ be continuous on $\Ob\times \overline{I}\times \Ld$ and satisfy \eqref{eq:uniform_ellipticity} and \eqref{eq:cordes_condition3}, or alternatively \eqref{eq:cordes_condition2} in the case where $b\equiv 0$ and $c\equiv 0$. Then, there exists a unique strong solution $u\in\HIO$ of the HJB equation \eqref{eq:HJB}. Moreover, $u$ is also the unique solution of $\Fg[u]=0$ in $\Om\times I$, $u=0$ on $\DO \times I$ and $u=u_0$ on $\Om\times \left\{0\right\}$.
\end{theorem}
\begin{proof}
The proof consists of establishing the equivalence of \eqref{eq:HJB} with the problem of solving the equation $\Fg[u]=0$ and $u(0)=u_0$, which can be analysed with the Browder--Minty Theorem. Let the operator $\calA\colon \HIO\tends \HIO^{*}$ be defined by
\begin{equation}\label{eq:calA_definition}
\pair{\calA(u)}{v}\coloneqq\int_I\int_\Om \Fg[u]\,\Lw v\,\d x\,\d t + \omega \pair{u(0)-u_0}{v(0)}_{H^1_0}.
\end{equation}
Compactness of $\Ld$ and continuity of the data imply that $\calA$ is Lipschitz continuous. Indeed, letting $u$, $v$ and $z\in \HIO$, we find that
\begin{multline}\label{eq:calA_Lipschitz_continuity}
\abs{\pair{\calA(u)-\calA(v)}{z}} \leq\norm{\Fg[u]-\Fg[v]} _{L^2(I;L^2(\Om))}\norm{\Lw z}_{L^2(I;L^2(\Om))} \\ + \omega\norm{u(0)-v(0)}_{H^1_0}\norm{z(0)}_{H^1_0}
\leq C\norm{u-v}_{\HIO}\norm{z}_{\HIO},
\end{multline}
where the constant $C$ depends only on the dimension $d$, $\omega$, $T$, and on the supremum norms of $a$, $b$, $c$ and $f$ and $\gamma$ over $\Ob\times \overline{I}\times\Ld$.
We also claim that $\calA$ is strongly monotone. Define $w\coloneqq u-v$. Addition and substraction of $\int_{I_n}\pair{\Lw w}{\Lw w}_{L^2}\,\d t$ shows that
\begin{multline*}
\pair{\calA(u)-\calA(v)}{w} = \norm{\Lw w}^2_{L^2(I;L^2(\Om))}+ \omega \norm{w(0)}_{H^1_0}^2 \\ +\int_I\int_\Om \left( \Fg[u]-\Fg[v]-\Lw w\right)\Lw w\,\d x\, \d t .
\end{multline*}
Lemma~\ref{lem:cordes_bound}, the bound \eqref{eq:heat_bound} and the Cauchy--Schwarz Inequality show that
\begin{equation}\label{eq:calA_strong_monotonicity}
\begin{aligned}
\pair{\calA(u)-\calA(v)}{w} &\geq \frac{1}{2} \norm{\Lw w}_{L^2(I;L^2(\Om))}^2 + \omega\norm{w(0)}_{H^1_0}^2  - \frac{1-\eps}{2}\norm{w}_{\HIO}^2 \\
& \geq \frac{\eps}{2} \norm{w}_{\HIO}^2 + \frac{\omega}{2}\norm{w(T)}_{H^1_0}^2 + \frac{\omega}{2}\norm{w(0)}_{H^1_0}^2.
\end{aligned}
\end{equation}
The inequalities \eqref{eq:calA_Lipschitz_continuity} and \eqref{eq:calA_strong_monotonicity} imply that $\calA$ is a bounded, continuous, coercive and strongly monotone operator, so the Browder--Minty~Theorem \cite{Renardy2004} shows that there exists a unique $u\in\HIO$ such that $\calA(u)=0$.

Theorem~\ref{thm:heat_equation} shows that for each $g\in L^2(I;L^2(\Om))$, there exists a $v\in\HIO$ such that $\Lw v = g$ and $v(0)=0$. So, $\calA(u)=0$ implies that $\int_I\int_\Om \Fg[u]\,g\,\d x \, \d t =0$ for all $g\in L^2(I;L^2(\Om))$, and since $\Fg[u] \in L^2(I;L^2(\Om))$, we obtain $\Fg[u]=0$. Theorem~\ref{thm:heat_equation} also shows that $\pair{u(0)}{v}_{H^1_0} = \pair{u_0}{v}_{H^1_0}$ for all $v\in H^1_0(\Om)$, hence $u(0)=u_0$.

We claim that $u\in\HIO$ solves $\Fg[u]=0$ if and only if $u$ solves \eqref{eq:HJB}. Since $\gamma^\a$ is positive, $\gamma^\a(\p_t u-L^\a u+ f^\a)\geq 0$ for all $\a\in\Ld$ is equivalent to $\p_t u - L^\a u+f^\a \geq 0$ for all $\a\in \Ld$, so $\Fg[u]\geq 0$ is equivalent to $F[u]\geq 0$. Compactness of $\Ld$ and continuity of the data imply that for a.e.\ $t\in I$, for a.e.\ point of $\Om$, the extrema in the definitions of $\Fg[u]$ and $F[u]$ are attained by some elements of $\Ld$, thereby giving $\Fg[u]\leq 0$ if and only if $F[u]\leq 0$. Therefore, existence and uniqueness in $\HIO$ of a solution of $\Fg[u]=0$ is equivalent to existence and uniqueness of a solution of \eqref{eq:HJB}.\qed
\end{proof}

\section{Temporal semi-discretisation}\label{sec:temporal}%
In this section, we explore some of the general principles underlying the numerical scheme for the parabolic problem \eqref{eq:HJB}.
Before presenting the fully-discrete scheme in section~\ref{sec:numerical_scheme}, we briefly consider in this section the temporal semi-discretisation of parabolic HJB equations, so as to highlight some key ideas in the derivation and analysis of a stable method. The fully-discrete scheme will then combine these ideas with the methods from \cite{Smears2014} used to discretise space.

The proof of Theorem~\ref{thm:hjb_well_posed} indicates that we should discretise the operator appearing in \eqref{eq:calA_definition}, and find stability in a norm that is analogous to $\norm{\cdot}_{\HIO}$ from \eqref{eq:HIO_norm}.
Although \eqref{eq:calA_definition} expresses the global space-time problem, we will employ a temporal discontinuous Galerkin method, thus leading to a time-stepping scheme.

Let $\left\{\It\right\}_\tau$ be a sequence of partitions of $(0,T)$ into half-intervals $I_n\coloneqq (t_{n-1},t_n] \in \It$, with $1\leq n \leq N=N(\tau)$. We say that $\It$ is regular provided that
\begin{equation}
[0,T]=\bigcup_{I_n\in\It}\overline{I_n},\quad 0=t_0 \leq t_{n-1} < t_n \leq t_N = T,\quad\forall\, n\leq N,\; \forall \,\tau.
\end{equation}
For each interval $I_n \in \It$, let $\tau_n \coloneqq \abs{t_n-t_{n-1}}$. It is assumed that $\tau = \max_{1\leq n \leq N} \tau_n$.
For each $\tau$, let $\bo q=\left(q_1,\dots,q_N\right)$ be a vector of positive integers, so $q_n\geq 1$ for all $I_n\in\It$. For a vector space $V$ and $I_n\in\It$, let $\calQ_{q_n}\left(V\right)$ denote the space of $V$-valued univariate polynomials of degree at most $q_n$. Recalling that $H\coloneqq H^2(\Om)\cap H^1_0(\Om)$, we define the semi-discrete DG finite element space $\Vt$ by
\begin{equation}\label{eq:vt_definition}
\Vt \coloneqq \left\{v\in L^2(I;H),\; \eval{v}{I_n}\in \calQ_{q_n}(H)\;\;\forall\,I_n\in\It \right\}.
\end{equation}
Functions from $\Vt$ are taken to be left-continuous, but are generally discontinuous at the partition points $\{t_n\}_{n=1}^{N-1}$.
We denote the right-limit of $v\in\Vt$ at $t_n$ by $v(t_n^+)$, where $0\leq n < N$.
The jump operators $\lld \cdot \rrd$ and average operators $\llc \cdot \rrc$, $0\leq n \leq N$, are defined by
\begin{equation}
\begin{aligned}
\lld v \rrd &\coloneqq - v(0^+), & \llc v \rrc &\coloneqq v(0^+), & &\text{if }n=0,\\
\lld v \rrd &\coloneqq v(t_n)-v(t_n^+), & \llc v \rrc &\coloneqq\half v(t_n)+\half v(t_n^+), &
&\text{if }1\leq n < N,\\
\lld v \rrd &\coloneqq  v(T), &  \llc v \rrc &\coloneqq v(T), &
&\text{if } n=N.
\end{aligned}
\end{equation}
Define the nonlinear form $A_\tau\colon\Vt\times \Vt \tends \R$  by
\begin{multline}
A_\tau(u_\tau;v_\tau) \coloneqq \SIn \pair{\Fg[u_\tau]}{\Lw v_\tau}_{L^2(\Om)} \, \d t \\ - \omega \sum_{n=0}^{N-1} \pair{\lld u_\tau \rrd}{\llc v_\tau \rrc}_{H^1_0}
+\frac{\omega}{2} \sum_{n=1}^{N-1} \pair{\lld u_\tau \rrd}{\lld v_\tau \rrd}_{H^1_0}.
\end{multline}
We note that $ \half \lld v \rrd -\llc v \rrc  = v(t_n^+)$ for $1\leq n < N$.
The semi-discrete scheme consists of finding a $u_\tau\in\Vt$ such that
\begin{equation}\label{eq:semitime_scheme}
A_\tau(u_\tau;v_\tau) = \omega \pair{u_0}{v_\tau(0^+)}_{H^1_0}
\qquad\forall\,v_\tau\in\Vt.
\end{equation}
Since the solution $u \in \HIO$ of \eqref{eq:HJB} belongs to $C(\overline{I};H^1_0(\Om))$, it is clear that $A_\tau(u;v_\tau)=\omega\pair{u_0}{v_\tau(0^+)}_{H^1_0}$ for all $v_\tau\in\Vt$, so the scheme is consistent. By considering test functions $v_\tau$ that have support on successive intervals $\overline{I_n}\in\It$, it is easily seen that $\eval{u_\tau}{I_n}$ is determined only by the data and by $u(t_{n-1})$, thus \eqref{eq:semitime_scheme} is a time-stepping scheme.
The main ingredients required to show that the above scheme is stable are as follows.
We introduce the bilinear form $C_\tau\colon \Vt\times\Vt\tends \R$ defined by
\begin{multline}\label{eq:c_tau_definition_2}
C_\tau(u_\tau,v_\tau) \coloneqq \SIn \pair{ \Lw u_\tau}{\Lw v_\tau}_{L^2(\Om)}\, \d t
\\ - \omega \sum_{n=0}^{N-1} \pair{\lld u_\tau \rrd}{\llc v_\tau \rrc}_{H^1_0}
 +\frac{\omega}{2} \sum_{n=1}^{N-1} \pair{\lld u_\tau \rrd}{\lld v_\tau \rrd}_{H^1_0}.
\end{multline}
Integration by parts shows that for any $u_\tau$, $v_\tau\in \Vt$, we have
\begin{multline}\label{eq:c_tau_definition_1}
C_\tau(u_\tau,v_\tau) 
=
\SIn \omega^2 \pair{\p_t u_\tau}{\p_t v_\tau}_{L^2(\Om)} + \pair{\Ll u_\tau}{\Ll v_\tau}_{L^2(\Om)}\, \d t
\\
+\omega \sum_{n=1}^{N} \pair{\llc u_\tau \rrc}{\lld v_\tau \rrd}_{H^1_0} +\frac{\omega}{2} \sum_{n=1}^{N-1} \pair{\lld u_\tau \rrd}{\lld v_\tau \rrd}_{H^1_0}.
\end{multline}
Combining \eqref{eq:c_tau_definition_2} and \eqref{eq:c_tau_definition_1} reveals the stability properties of $C_\tau$ when re-written as
\begin{multline}\label{eq:c_tau_definition_3}
C_\tau(u_\tau,v_\tau)= \frac{1}{2} \SIn \omega^2 \pair{\p_t u_\tau}{\p_t v_\tau}_{L^2} + \pair{\Ll u_\tau}{\Ll v_\tau}_{L^2} + \pair{\Lw u_\tau}{\Lw v_\tau}_{L^2}\, \d t
\\ +\frac{\omega}{2} \sum_{n=1}^{N} \pair{\llc u_\tau \rrc}{\lld v_\tau \rrd}_{H^1_0} - \frac{\omega}{2}\sum_{n=0}^{N-1} \pair{\lld u_\tau \rrd}{\llc v_\tau \rrc}_{H^1_0}+\frac{\omega}{2} \sum_{n=1}^{N-1} \pair{\lld u_\tau \rrd}{\lld v_\tau \rrd}_{H^1_0}.
\end{multline}
Indeed, it follows from \eqref{eq:c_tau_definition_3} and the Miranda--Talenti Inequality that, for any $u_\tau \in \Vt$,
\begin{multline}\label{eq:c_tau_stability}
C_\tau(u_\tau,u_\tau)\geq
\frac{1}{2} \SIn \omega^2  \norm{\p_t u}_{L^2(\Om)}^2+\norm{u_\tau}_{H}^2 + \norm{\Lw u_\tau }_{L^2(\Om)}^2  \d t
\\  + \frac{\omega}{2} \norm{u_\tau(T)}_{H^1_0}^2 +  \frac{\omega}{2}\norm{u_\tau(0^+)}_{H^1_0}^2 +\frac{\omega}{2} \sum_{n=1}^{N-1} \norm{\lld u_\tau \rrd }_{H^1_0}^2.
\end{multline}
The key observation here is that the antisymmetric terms in \eqref{eq:c_tau_definition_3} cancel in $C_\tau(u_\tau,u_\tau)$, and this technique will be used again in section~\ref{sec:stability} for the analysis of stability of the fully-discrete scheme.
The above considerations imply stability of the scheme as follows: \eqref{eq:c_tau_definition_2} implies that 
\begin{equation*}
A_\tau(u_\tau;v_\tau) = \SIn \pair{\Fg[u_\tau]-\Lw u_\tau}{\Lw v_\tau}_{L^2(\Om)} \,\d t + C_\tau(u_\tau,v_\tau) \quad\forall\,u_\tau,\,v_\tau\in\Vt;
\end{equation*}
which mirrors the addition-substraction step of the proof of Theorem~\ref{thm:hjb_well_posed}. Then, we use \eqref{eq:c_tau_stability} to show that $A_\tau$ is strongly monotone: for any $u_\tau$, $v_\tau\in\Vt$, $w_\tau \coloneqq u_\tau - v_\tau$, we have
\[
A_\tau(u_\tau;w_\tau)-A_\tau(v_\tau;w_\tau)
\geq \frac{\eps}{2}\SIn \omega^2 \norm{\p_t w_\tau}_{L^2(\Om)}^2 + \norm{w_\tau}_{H}^2  \d t +\frac{\omega}{2} \sum_{n=0}^{N} \norm{\lld w_\tau \rrd }_{H^1_0}^2.
\]
Therefore, the well-posedness of the semi-discrete scheme can be shown by an induction argument, based on the Browder--Minty Theorem, that is similar to the one given in the proof of Theorem~\ref{thm:dg_stability} below, concerning the well-posedness of the fully-discrete scheme.
Instead of pursuing the analysis of the semi-discrete scheme further, we now turn towards the fully-discrete method.

\section{Finite element spaces}
\label{sec:fem}%
Let $\{\calT_h\}_h$ be a sequence of shape-regular meshes on $\Om$, such that each element $K\in\calT_h$ is a simplex or a parallelepiped. Let $h_K \coloneqq \diam K$ for each $K\in\calT_h$. It is assumed that $h = \max_{K\in\calT_h} h_K$ for each mesh $\calT_h$.
Let $\calF^i_h$ denote the set of interior faces of the mesh $\calT_h$ and let $\calF_h^b$ denote the set of boundary faces.
The set of all faces of $\calT_h$ is denoted by $\calFhib \coloneqq \calF_h^i\cup\calF_h^b$. Since each element has piecewise flat boundary, the faces may be chosen to be flat.

\paragraph{Mesh conditions}
The meshes are allowed to be irregular, i.e.\ there may be hanging nodes. We assume that there is a uniform upper bound on the number of faces composing the boundary of any given element; in other words, there is a $c_{\calF}>0$, independent of $h$, such that
\begin{equation}\label{eq:card_F_bound}
\max_{K \in \calT_h} \card \{F \in \calFhib \colon F \subset \p K \} \leq c_{\calF}.
\end{equation}
It is also assumed that any two elements sharing a face have commensurate diameters, i.e.\ there is a $\ch \geq 1$, independent of $h$, such that, for any $K$, $K^{\prime}$ that share a face,
\begin{equation}\label{eq:c_h_bound}
 \max ( h_K, h_{K^{\prime}} ) \leq \ch \min( h_K, h_{K^{\prime}} ).
\end{equation}
For each $h$, let $\bo p = \left( p_K ;\; K\in \calT_h\right)$ be a vector of positive integers, such that there is a $\cp \geq 1$, independent of $h$, such that, for any $K$, $K^{\prime}$ that share a face,
\begin{equation}\label{eq:c_p_bound}
 \max( p_K, p_{K^\prime} ) \leq \cp \min( p_K, p_{K^\prime} ).
\end{equation}

\paragraph{Function spaces}
For each $K \in \calT_h$, let $\calP_{p_K}$ be the space of all real-valued polynomials in $\R^\dim$ with either total or partial degree at most $p_K$. In particular, we allow the combination of spaces of polynomials of fixed total degree on some parts of the mesh with spaces of polynomials of fixed partial degree on the remainder. We also allow the use of the space of polynomials of total degree at most $p_K$ even when $K$ is a parallelepiped.
The spatial discontinuous Galerkin finite element space $V_{h,\bo p}$ is defined by
\begin{equation}\label{eq:cordes_fem_space}
 V_{h,\bo p} \coloneqq \left\{v \in L^2(\Om),\; \eval{v}{K} \in \calP_{p_K}\;\;\forall\, K \in \calT_h \right\}.
\end{equation}
For $\It$ a regular partition of $I$, the space-time discontinuous Galerkin finite element space $\Vth$ is defined by
\begin{equation}\label{eq:spacetime_fem_space}
\Vth \coloneqq \left\{ v\in L^2\left(I;\Vh\right),\; \eval{v}{I_n}\in \calQ_{q_n}\left(\Vh\right)\;\;\forall \,I_n\in \It\right\}.
\end{equation}
As in section~\ref{sec:temporal}, we take functions from $\Vth$ to be left-continuous.
The support of a function $v_h\in\Vth$, denoted by $\supp v_h$, is a subset of $\overline{I}$, and is understood to be the support of $v_h\colon I\tends \Vh$, i.e.\ when viewing $v_h$ as a mapping from $I$ into $\Vh$. 

For $\bo s  \coloneqq \left( s_K \colon K\in \calT_h\right)$ a vector of nonnegative real numbers, and $r\in[1,\infty]$, define the broken Sobolev space $W^{\bo s}_r(\Om;\calT_h) \coloneqq \left\{ v \in L^r(\Om),\;\eval{v}{K} \in W^{s_K}_r(K)\;\forall K\in\calT_h \right\}$.
For shorthand, define $H^{\bo s}(\Om;\calT_h)\coloneqq W^{\bo s}_2(\Om;\calT_h)$, and, for $s\geq 0$, set $W^{s}_r(\Om;\calT_h) \coloneqq W^{\bo s}_r(\Om;\calT_h)$, where $s_K = s$ for all $K\in\calT_h$.
Define the norm $\norm{\cdot}_{W^{\bo s}_r(\Om;\calT_h)}$ on $W^{\bo s}_r(\Om;\calT_h)$ by $\norm{v}_{W^{\bo s}_r(\Om;\calT_h)}^r \coloneqq \sum_{K\in\calT_h}\norm{v}_{W^{s_K}_r(K)}^r$, with the usual modification when $r=\infty$.

\paragraph{Spatial jump, average, and tangential operators}
For each face $F$, let $n_F \in \Rd$ denote a fixed choice of a unit normal vector to $F$. Since each face $F$ is flat, the normal $n_F$ is constant. For an element $K\in\calT_h$ and a face $F\subset \p K$, let $\tau_F\colon H^{s}(K)\tends H^{s-1/2}(F)$, $s>1/2$, denote the trace operator from $K$ to $F$. The trace operator $\tau_F$ is extended componentwise to vector-valued functions.
Define the jump operator $\llb \cdot \rrb$ and the average operator $\lla \cdot \rra$ by
\begin{align*}
  \llb \phi \rrb  &\coloneqq \tau_F \left( \eval{\phi}{K_{\rm{ext}}} \right) - \tau_F \left( \eval{\phi}{K_{\rm{int}}} \right),
& \lla \phi \rra &\coloneqq  \half\tau_F \left( \eval{\phi}{K_{\rm{ext}}} \right) + \half \tau_F \left( \eval{\phi}{K_{\rm{int}}} \right),
&\text{if }F\in\calF^i_h,
\\
\llb \phi \rrb  &\coloneqq \tau_F \left(\eval{\phi}{K_{\rm{ext}}}\right),
&  \lla \phi \rra &\coloneqq  \tau_F \left(\eval{\phi}{K_{\rm{ext}}} \right),
&\text{if }F\in\calF_h^b,
\end{align*}
where $\phi$ is a sufficiently regular scalar or vector-valued function, and $K_{\rm{ext}}$ and $K_{\rm{int}}$ are the elements to which $F$ is a face, i.e.\ $F= \p K_{\rm{ext}} \cap \p K_{\rm{int}}$.
Here, the labelling is chosen so that $n_F$ is outward pointing for $K_{\rm{ext}}$ and inward pointing for $K_{\rm{int}}$.
Using this notation, the jump and average of scalar-valued functions, resp.\ vector-valued, are scalar-valued, resp.\ vector-valued.
For a face $F$, let $\nabla_{\rm T}$ and $\Div_{\rm T}$ denote respectively the tangential gradient and tangential divergence operators on $F$; see \cite{Grisvard2011,Smears2013} for further details.

\section{Numerical Scheme}\label{sec:numerical_scheme}
The definition of the numerical scheme requires the following bilinear forms, which were first introduced in the analysis of elliptic HJB equations in \cite{Smears2014}.
First, for $\lambda\geq0$ as in section~\ref{sec:problem_analysis}, the symmetric bilinear form $\Blambda\colon \Vh \times \Vh \tends \R$ is defined by
\begin{equation*}
\begin{split}
\Blambda (u_h, v_h) &\coloneqq \sum_{K\in\calT_h}\left[\pair{D^2u_h}{D^2v_h}_K + 2\lambda \pair{\nabla u_h}{\nabla v_h}_K + \lambda^2 \pair{u_h}{v_h}_K\right]\\
&+ \sum_{F \in \calF^i_h} \bigl[\pair{\DivT\nablaT \lla u_h \rra}{ \llb \nabla v_h \cdot n_F \rrb }_F + \pair{\DivT\nablaT \lla v_h \rra}{ \llb \nabla u_h \cdot n_F \rrb }_F \bigr] \\
&-\sum_{F \in \calFhib} \bigl[ \pair{\nablaT \lla \nabla u_h \cdot n_F \rra}{\llb \nablaT v_h \rrb}_F + \pair{\nablaT \lla \nabla v_h \cdot n_F \rra}{\llb \nablaT u_h \rrb}_F\bigr]
\\ &  - \lambda \sum_{F\in\calFhib} \left[ \pair{\lla \nabla u_h \cdot n_F\rra}{\llb v_h \rrb}_F+\pair{\lla \nabla v_h \cdot n_F\rra}{\llb u_h \rrb}_F \right]
\\ & - \lambda \sum_{F\in\calF^i_h} \left[ \pair{\lla u_h \rra}{\llb \nabla v_h\cdot n_F \rrb}_F+\pair{\lla v_h\rra}{\llb \nabla u_h\cdot n_F\rrb}_F \right],
\end{split}
\end{equation*}
Then, for face-dependent quantities $\mu_F>0$ and $\eta_F>0$, to be specified later, let the jump stabilisation term $\Jo\colon\Vh\times\Vh\tends\R$ be defined by
\begin{multline}
 \Jo(u_h,v_h) \coloneqq \sum_{F\in\calF^i_h} \mu_F \pair{\llb \nabla u_h\cdot n_F \rrb}{\llb \nabla v_h\cdot n_F \rrb}_F
\\+ \sum_{F \in\calFhib}\bigl[ \mu_F \pair{\llb \nablaT u_h \rrb}{\llb \nablaT v_h \rrb}_F
+  \eta_F \pair{\llb u_h \rrb}{\llb v_h \rrb}_F\bigr].
\end{multline}
Recalling that $\Ll v \coloneqq \Delta v - \lambda v$, we introduce the one-parameter family of bilinear forms $\Bdgr \colon \Vh \times \Vh \tends \R$, where $\theta\in[0,1]$, defined by
\begin{equation}\label{eq:cordes_bdgr}
 \Bdgr(u_h,v_h) \coloneqq \theta B_{h,*}(u_h,v_h) + (1-\theta) \sum_{K\in\calT_h}\pair{\Ll u_h}{\Ll v_h}_K + \Jo(u_h,v_h).
\end{equation}
Define the bilinear form $a_h\colon \Vh\times\Vh\tends\R$ by
\begin{multline}
a_h(u_h,v_h) \coloneqq \sum_{K\in\calT_h}\pair{\nabla u_h}{\nabla v_h}_K + \lambda \pair{u_h}{v_h}_K - \sum_{F\in\calFhib}\pair{\lla \nabla u_h\cdot n_F \rra}{\llb v_h \rrb}_F 
\\ -\sum_{F\in\calFhib}\pair{\lla \nabla v_h\cdot n_F \rra}{\llb u_h \rrb}_F + \sum_{F\in\calFhib} \mu_F \pair{\llb u_h\rrb}{\llb v_h \rrb}_F.
\end{multline}
Observe that the bilinear form $a_h$ corresponds precisely to the standard symmetric interior penalty discretisation of the operator $-\Ll$, and its symmetry plays an imporant role in the subsequent analysis.

Define the bilinear forms $\CFh$ and $\Cdg \colon \Vth \times \Vth \tends \R$ by
\begin{align}
\CFh(u_h,v_h) &\coloneqq \omega \SIn \sum_{F\in\calF_h^i}\pair{\llb \nabla u_h\cdot n_F\rrb}{\lla \p_t v_h \rra }_F\, \d t
\\& +\omega\SIn \sum_{F\in\calFhib}\left[ \mu_F \pair{\llb u_h\rrb}{\llb \p_t v_h \rrb}_F - \pair{\llb u_h\rrb}{\lla \nabla \p_t v_h \cdot n_F\rra}_F\right]  \d t ,\nonumber
\\
\Cdg(u_h,v_h) &\coloneqq \SIn \SK \pair{\Lw u_h}{\Lw v_h}_K \, \d t + \CFh(u_h,v_h) \label{eq:cdg_definition}
\\ &\qquad+ \SIn \Bdg(u_h,v_h)-\sum_{K\in\calT_h}\pair{\Ll u_h}{\Ll v_h}_K\,\d t  \nonumber
\\&\qquad- \omega \sum_{n=0}^{N-1} a_h(\lld u_h \rrd,\llc v_h \rrc) + \frac{\omega}{2} \sum_{n=1}^{N-1} a_h(\lld u_h \rrd,\lld v_h \rrd).\nonumber
\end{align}
Define the nonlinear form $\Adg\colon\Vth\times\Vth\tends\R$ by
\begin{equation}\label{eq:adg_definition}
\Adg(u_h;v_h)\coloneqq\sum_{n=1}^N\int_{I_n}\sum_{K\in\calT_h}\left[\pair{\Fg[u_h]}{\Lw v_h}_K - \pair{\Lw u_h}{\Lw v_h}_K\right] \d t+ \Cdg(u_h,v_h).
\end{equation}
The form $\Adg$ is linear in its second argument, but it is nonlinear in its first argument.
Supposing that $u_0$ is sufficiently regular, such as $u_0\in H^s(\Om;\calT_h)$, with $s>3/2$, the numerical scheme is to find $u_h\in\Vth$ such that
\begin{equation}\label{eq:numerical_scheme}
\Adg(u_h;v_h)= \omega\, a_h(u_0,v_h(0^+)) \qquad\forall\, v_h\in\Vth.
\end{equation}
If $u_0$ fails to be sufficiently regular, we can replace $u_0$ in the right-hand side of \eqref{eq:numerical_scheme} with a suitable projection into $\Vh$, at the expense of introducing a consistency error that vanishes in the limit.
By testing with functions $v_h\in\Vth$ that are supported on $\overline{I_n}$, it is found that \eqref{eq:numerical_scheme} is equivalent to finding $u_h\in\Vth$ such that
\begin{multline}\label{eq:numerical_scheme_time-stepping}
\int_{I_n} \SK \pair{\Fg[u_h]}{\Lw v_h}_K + \Bdg(u_h,v_h)-\SK \pair{\Ll u_h}{\Ll v_h}_K\,\d t
\\ + \omega \int_{I_n} \sum_{F\in\calF_h^i}\pair{\llb \nabla u_h\cdot n_F\rrb}{\lla \p_t v_h \rra }_F+ \sum_{F\in\calFhib} \mu_F \pair{\llb u_h\rrb}{\llb \p_t v_h \rrb}_F \,\d t
 \\ - \omega \int_{I_n} \sum_{F\in\calFhib}\pair{\llb u_h\rrb}{\lla \nabla \p_t v_h \cdot n_F\rra}_F \, \d t
+ \omega \,a_h(u_h(t_{n-1}^+),v_h(t_{n-1}^+)) 
\\ = \omega\, a_h(u_h(t_{n-1}),v_h(t_{n-1}^+)) ,
\end{multline}
for all $v_h\in \calQ_{q_n}(\Vh)$, with the convention $u_h(t_0)\coloneqq u_0$.
Therefore, \eqref{eq:numerical_scheme} defines a time-stepping scheme, and in practice it is \eqref{eq:numerical_scheme_time-stepping} that is used for computations.

\paragraph{Consistency}
The following result is shown in \cite{Smears2013,Smears2014}.
\begin{lemma}\label{lem:broken_identity}
Let $\Om$ be a bounded Lipschitz polytopal domain and let $\calT_h$ be a simplicial or parallelepipedal mesh on $\Om$. Let $w\in H^s(\Om;\calT_h) \cap H^2(\Om)\cap H^1_0(\Om)$, with $s>5/2$. Then, for every $v_h \in \Vh$, we have the identities
\begin{equation}\label{eq:consistency_B_lambda}
 \Blambda(w,v_h)=\sum_{K\in\calT_h}\pair{\Ll w}{\Ll v_h}_K \quad\text{and}\quad \Jo(w,v_h)=0.
\end{equation}
\end{lemma}
\begin{lemma}\label{lem:cdg_consistency}
Let $\Om$ be a bounded Lipschitz polytopal domain, let $\calT_h$ be a simplicial or parallelepipedal mesh on $\Om$. Let $I=(0,T)$ and let $\It=\{I_n\}_{n=1}^N$ be a regular partition of $I$. Suppose that $u_0 \in H^1_0(\Om) \cap H^{r}(\Om;\calT_h)$ with $r>3/2$. Then, for any $w\in \HIO \cap L^2(I;H^s(\Om;\calT_h))$, with $s>5/2$, such that $w(0)=u_0$, we have
\begin{equation}\label{eq:cdg_consistency}
\Cdg(w,v_h) = \SIn \SK \pair{\Lw w }{\Lw v_h}_K\, \d t + \omega\,a_h(u_0,v_h(0^+)) \qquad\forall\,v_h\in\Vth.
\end{equation}
\end{lemma}
\begin{proof}
Let the function $w$ be as above, so that $w(t) \in H^2(\Om)\cap H^1_0(\Om)\cap H^s(\Om;\calT_h)$ for a.e. $t\in I$. Lemma~\ref{lem:broken_identity} shows that
$\int_{I_n} \Bdg(w,v_h) \, \d t = \int_{I_n} \SK \pair{\Ll w}{\Ll v_h}_K \, \d t$ for all $I_n\in\It$ and all $ v_h\in\Vth$.
The spatial regularity of $w$ also implies that $\llb \nabla w(t) \cdot n_F \rrb$ vanishes for all $F\in\calF_h^i$ and a.e.\ $t\in I$, whilst $\llb w(t) \rrb$ and $\llb \nablaT w(t) \rrb$ vanish for all $F\in\calFhib$ and a.e.\ $t\in I$. Therefore we have
$\CFh(w,v) = 0$ for all $v_h\in\Vth$.
Finally, since $\HIO \hookrightarrow C(\overline{I};H^1_0(\Om))$ by Lemma~\ref{lem:gelfand_triple}, the jump $\lld w \rrd = 0$ for each $0< n < N$, and thus $a_h(\lld w \rrd , v_h ) = 0$ for all $v_h\in \Vh$, $0 < n < N$.
The above identities and the definition of $\Cdg$ in \eqref{eq:cdg_definition} imply \eqref{eq:cdg_consistency}.
\qed
\end{proof}

Lemma~\ref{lem:cdg_consistency} and the definition of the nonlinear form $\Adg$ in \eqref{eq:adg_definition} immediately imply the following consistency result for the numerical scheme.
\begin{corollary}\label{cor:scheme_consistency}
Under the hypotheses of Lemma~\ref{lem:cdg_consistency}, suppose that the solution $u\in \HIO$ of \eqref{eq:HJB} belongs to $L^2(I;H^s(\Om;\calT_h))$, with $s>5/2$. Then, $u$ satisfies
\begin{equation}
\Adg(u;v_h) = \omega\,a_h(u_0,v_h(0^+)) \qquad\forall\,v_h\in\Vth.
\end{equation}
\end{corollary}
\section{Stability}\label{sec:stability}
It will be seen below that, for $\mu_F$ appropriately chosen, the symmetric bilinear form $a_h$ is coercive on $\Vh$, and thus defines an inner-product on $\Vh$, with associated norm $\normah{v_h}^2 \coloneqq a_h(v_h,v_h)$ for $v_h\in\Vh$. Define the functionals 
\begin{align}
&\abs{v_h}_{H^2(K),\lambda}^2 \coloneqq \abs{v_h}_{H^2(K)}^2 + 2\lambda \abs{v_h}_{H^1(K)}^2 + \lambda^2\norm{v_h}_{L^2(K)}^2, & &v_h\in \Vh,\;K\in\calT_h,\\
&\absJ{v_h}^2 \coloneqq \Jo(v_h,v_h), & & v_h\in\Vh.
\end{align}
For each $\theta \in [0,1]$, we introduce the functional $\normdgr{\cdot}\colon\Vth\tends\R$ defined by
\begin{multline}
\normdgr{v_h}^2 \coloneqq \SIn \SK \theta  \left[\omega^2\norm{\p_t v_h}_{L^2(K)}^2 + \abs{v_h}_{H^2(K),\lambda}^2 \right] + \absJ{v_h}^2\, \d t 
\\ +  \SIn \SK \left(1-\theta\right) \norm{\Lw v_h}_{L^2(K)}^2  \d t
+ \omega \sum_{n=0}^{N} \normah{\lld v_h \rrd}^2.
\end{multline}
It is shown below that, for an appropriate choice of $\mu_F$, $\normdgr{\cdot}$ defines a norm on $\Vth$ for each $\theta\in [0,1]$.
For each face $F \in \calFhib$, define
\begin{equation}\label{eq:def_tilde_ph}
\tilde{h}_F \coloneqq
  \begin{cases}
  \min(h_K,h_{K^\prime}), &\text{if }F \in \calF^i_h, \\
  h_K, &\text{if }F \in \calF^b_h,
  \end{cases}
  \qquad
   \tilde{p}_F \coloneqq
  \begin{cases}
  \max(p_K,p_{K^\prime}), &\text{if }F \in \calF^i_h, \\
  p_K, &\text{if }F \in \calF^b_h,
  \end{cases}
\end{equation}
where $K$ and $K^\prime$ are such that $F = \p K \cap \p K^\prime$ if $F\in\calF^i_h$ or $F \subset \p K \cap \DO$ if $F\in\calF^b_h$.
The following result is from \cite[Lemma 6]{Smears2014}.
\begin{lemma}\label{lem:space_dg_coercivity}
Let $\Om$ be a bounded convex polytopal domain and let $\{\calT_h\}_h$ be a shape-regular sequence of simplicial or parallelepipedal meshes satisfying \eqref{eq:card_F_bound}. Then, for each constant $\kappa >1$, there exists a positive constant $\cstab$, independent of $h$, $\bo p$ and $\theta$, such that, for any $v_h\in\Vh$ and any $\theta\in[0,1]$, we have
\begin{equation}\label{eq:dg_coercivity}
\Bdgr(v_h,v_h) \geq \SK \left[ \frac{\theta}{\kappa}\abs{v_h}_{H^2(K),\lambda}^2+\left(1-\theta\right)\norm{\Ll v_h}_{L^2(K)}^2\right]+\frac{1}{2}\absJ{v_h}^2,
\end{equation}
whenever, for any fixed constant $\sigma\geq 1$,
\begin{equation}\label{eq:dg_stabilisation}
 \mu_F = \sigma \cstab \frac{\tilde{p}_F^2}{\tilde{h}_F} \quad\text{and}\quad
\eta_F > \sigma \lambda\, \cstab \frac{\tilde{p}_F^2}{\tilde{h}_F}.
\end{equation}
\end{lemma}
We note that $\mu_F$ may be chosen as in Lemma~\ref{lem:space_dg_coercivity} whilst also guaranteeing the standard discrete Poincar\'{e} Inequality:
\begin{equation}\label{eq:a_h_coercivity}
\sum_{K\in\calT_h} \norm{v_h}_{H^1(K)}^2 + \sum_{F\in\calFhib}\mu_F\norm{\llb v_h\rrb}_{L^2(F)}^2 \lesssim a_h(v_h,v_h) = \normah{v_h}^2 \qquad \forall\, v_h\in\Vh.
\end{equation}
In the subsequent analysis, we shall choose $\mu_F$ and $\eta_F$ to be given by
\begin{align}\label{eq:mu_eta}
\mu_F &\coloneqq \sigma \, \cstab \frac{\tilde{p}_F^2}{\tilde{h}_F}, 
& \eta_F &\coloneqq \sigma \max(1,\lambda)\, \cstab \frac{\tilde{p}_F^6}{\tilde{h}_F^3},
\end{align}
where $\cstab$ is chosen so that Lemma~\ref{lem:space_dg_coercivity} holds for $\kappa < (1-\eps)^{-1}$, and where $\sigma\geq 1$ is a fixed constant chosen such that \eqref{eq:a_h_coercivity} also holds. Note that these orders of penalisation are the strongest that remain consistent with the discrete $H^2$-type norm appearing in the analysis of this work; see \cite{Mozolevski2007} for an example of a scheme for the biharmonic equation using the same penalisation orders.

To verify that the functional $\normdgr{\cdot}$ defines a norm on $\Vth$, suppose that $\normdgr{v_h}=0$ for some $v_h\in\Vth$. Then, the jumps of $v_h$ vanish across the mesh faces and across time intervals and, therefore, $v_h\in\HIO$ with $v_h(0)=0$. The fact that the volume terms in $\normdgr{v_h}$ also vanish shows that $\Lw v_h=0$, so it follows from \eqref{eq:heat_bound} that $v_h \equiv 0$. Hence, the functional $\normdgr{\cdot}$ defines a norm on $\Vth$.
\begin{lemma}\label{lem:cdg_coercivity}
Under the hypotheses of Lemma~\ref{lem:space_dg_coercivity}, let $I=(0,T)$ and $\{\It\}_\tau$ be a sequence of regular partitions of $I$. Let $\mu_F$ and $\eta_F$ satisfy \eqref{eq:mu_eta} for each face $F$, so that Lemma~\ref{lem:space_dg_coercivity} holds for a given $\kappa >1$. Then, for every $v_h\in\Vth$, we have
\begin{multline}\label{eq:cdg_coercivity}
\Cdg(v_h,v_h) \geq \frac{1}{2} \SIn \SK \omega^2 \norm{\p_t v_h}_{L^2(K)}^2 + \kappa^{-1} \abs{v_h}_{H^2(K),\lambda}^2 + \absJ{v_h}^2\, \d t
\\  + \frac{1}{2}\SIn \SK \norm{\Lw v_h}_{L^2(K)}^2  \d t
+ \frac{\omega}{2} \sum_{n=0}^{N} \normah{\lld v_h \rrd}^2.
\end{multline}
\end{lemma}
\begin{proof}
We begin by showing that, for any $u_h$, $v_h\in\Vth$, the bilinear form $\Cdg$ satisfies the following identity:
\begin{multline}\label{eq:cdg_symmetric}
\Cdg(u_h,v_h) = \SIn \SK \omega^2 \pair{\p_t u_h}{\p_t v_h}_K + \Bdg(u_h,v_h) \, \d t - \CFh(v_h, u_h)
\\ + \omega \sum_{n=1}^N a_h(\llc u_h \rrc, \lld v_h \rrd) + \frac{\omega}{2} \sum_{n=1}^{N-1} a_h(\lld u_h \rrd, \lld v_h \rrd).
\end{multline}
The first step in deriving \eqref{eq:cdg_symmetric} is to show  that for any $u_h$, $v_h\in \Vth$, we have
\begin{multline}\label{eq:discrete_bochner}
\SIn \SK \pair{ \omega\, \p_t u_h}{-\Ll v_h}_K + \pair{\omega \,\p_t v_h}{-\Ll u_h}_K \,\d t
\\= \omega\sum_{n=1}^N a_h(\llc u_h \rrc,\lld v_h \rrd) + \omega\sum_{n=0}^{N-1} a_h(\lld u_h \rrd, \llc v_h \rrc) - \CFh(u_h,v_h) - \CFh(v_h,u_h).
\end{multline}
Indeed, integration by parts over $\calT_h$ shows that, for any $I_n\in\It$ and a.e.\ $t\in I_n$,
\begin{multline}
\SK \pair{\omega\, \p_t u_h}{-\Ll v_h}_K
= \omega \SK \pair{\nabla \p_t u_h}{\nabla v_h}_K + \lambda \pair{\p_t u_h}{v_h}_K
\\ - \omega \sum_{F\in\calF_h^i}\pair{\lla \p_t u_h \rra}{\llb \nabla v_h\cdot n_F\rrb}_F  - \omega \sum_{F\in\calFhib} \pair{\llb \p_t u_h \rrb}{\lla \nabla v_h \cdot n_F \rra}_F.
\end{multline}
Therefore, it is found that, for any $I_n\in\It$ and a.e.\ $t\in I_n$,
\begin{multline}\label{eq:cdg_symmetric_1}
\SK   \pair{ \omega\, \p_t u_h}{-\Ll v_h}_K  + \pair{\omega \,\p_t v_h}{-\Ll u_h}_K 
\\  = \omega \frac{\d }{\d t}\, a_h(u_h,v_h) - \omega\sum_{F\in\calFhib}\mu_F \left[\pair{\llb \p_t u_h\rrb}{\llb v_h\rrb}_F  + \pair{\llb u_h\rrb}{\llb \p_t v_h\rrb}_F\right]
\\  - \omega \sum_{F\in\calF_h^i}\left[\pair{\lla \p_t u_h \rra}{\llb \nabla v_h\cdot n_F\rrb}_F 
+ \pair{\lla \p_t v_h \rra}{\llb \nabla u_h\cdot n_F\rrb}_F \right]
\\  + \omega \sum_{F\in\calFhib} \left[
\pair{\llb v_h \rrb}{\lla \nabla \p_t u_h \cdot n_F \rra}_F+
\pair{\llb u_h \rrb}{\lla \nabla \p_t v_h \cdot n_F \rra}_F 
\right].
\end{multline}
We obtain \eqref{eq:discrete_bochner} upon integration and summation of \eqref{eq:cdg_symmetric_1} over all time intervals.
So, we have
\begin{multline*}
\SIn \SK \pair{\Lw u_h}{\Lw v_h}_K \d t
= \SIn \SK \omega^2 \pair{\p_t u_h}{\p_t v_h}_K + \pair{\Ll u_h}{\Ll v_h}_K  \d t
\\
+ \omega \sum_{n=1}^N a_h(\llc u_h \rrc,\lld v_h \rrd) + \omega\sum_{n=0}^{N-1} a_h(\lld u_h \rrd, \llc v_h \rrc) - \CFh(u_h,v_h) - \CFh(v_h,u_h).
\end{multline*}
The proof of \eqref{eq:cdg_symmetric} is then completed by substituting the above identity in the definition of $\Cdg$ from \eqref{eq:cdg_definition}.
Expanding $\Cdg$ with both \eqref{eq:cdg_definition} and \eqref{eq:cdg_symmetric} shows that
\begin{multline}\label{eq:cdg_stability}
\Cdg(u_h,v_h) = \frac{1}{2} \SIn \SK \omega^2 \pair{\p_t u_h}{\p_t v_h}_K + B_{h,1}(u_h,v_h) + \Jo(u_h,v_h) \, \d t
\\ + \frac{1}{2} \SIn \SK \pair{\Lw u_h}{\Lw v_h}_K \, \d t + \frac{1}{2} \,\CFh(u_h,v_h)- \frac{1}{2} \,\CFh(v_h,u_h)
\\ + \frac{\omega}{2} \sum_{n=1}^N a_h(\llc u_h \rrc, \lld v_h \rrd)
- \frac{\omega}{2} \sum_{n=0}^{N-1} a_h(\lld u_h \rrd,\llc v_h \rrc)
 + \frac{\omega}{2} \sum_{n=1}^{N-1} a_h(\lld u_h \rrd, \lld v_h \rrd).
\end{multline}
Note that to get \eqref{eq:cdg_stability}, we have used the identity 
\[\Bdg(u_h,v_h)- \frac{1}{2} \SK \pair{\Ll u_h}{\Ll v_h}_K  = \frac{1}{2} B_{h,1}(u_h,v_h) + \frac{1}{2} \Jo(u_h,v_h).
\]
To show \eqref{eq:cdg_coercivity}, we substitute $u_h = v_h$ in \eqref{eq:cdg_stability} and first observe that the flux terms involving $\CFh$ cancel. Furthermore, the symmetry of the bilinear form $a_h$ implies that
\begin{multline*}
 \sum_{n=1}^N a_h(\llc v_h \rrc, \lld v_h \rrd)
- \sum_{n=0}^{N-1} a_h(\lld v_h \rrd,\llc v_h \rrc)
 +  \sum_{n=1}^{N-1} \normah{\lld v_h \rrd}^2
\\=  a_h(v_h(T),v_h(T))+ a_h(v_h(0^+),v_h(0^+))+ \sum_{n=1}^{N-1} \normah{\lld v_h \rrd}^2
 = \sum_{n=0}^{N} \normah{\lld v_h \rrd}^2.
\end{multline*}
Then, we apply Lemma~\ref{lem:space_dg_coercivity} for $\theta = 1$ to get $B_{h,1}(v_h,v_h)\geq \kappa^{-1} \SK \abs{v_h}_{H^2(K),\lambda}^2$, thereby yielding \eqref{eq:cdg_coercivity}.\qed
\end{proof}

Recall that for a function $v_h\in\Vth$, the support of $v_h$ is a subset of $\overline{I}$, since $v_h$ is viewed as a mapping from $I$ into $\Vh$.
\begin{theorem}\label{thm:dg_stability}
Let $\Om$ be a bounded convex polytopal domain and let $\{\calT_h\}_h$ be a shape-regular sequence of meshes satisfying \eqref{eq:card_F_bound}. Let $I=(0,T)$ and let $\{\It\}_{\tau}$ be a sequence of regular partitions of $I$. Let $\Ld$ be a compact metric space and let the data $a$, $b$, $c$ and $f$ be continuous on $\Ob\times \overline{I}\times \Ld$ and satisfy \eqref{eq:uniform_ellipticity} and \eqref{eq:cordes_condition3}, or alternatively \eqref{eq:cordes_condition2} in the case where $b\equiv 0$ and $c\equiv 0$. Assume that the initial data $u_0 \in H^1_0(\Om)\cap H^s(\Om;\calT_h)$ with $s>3/2$. Let $\mu_F$ and $\eta_F$ satisfy \eqref{eq:mu_eta}, with $\cstab$ chosen so that Lemmas~\textnormal{\ref{lem:space_dg_coercivity}}~and~\textnormal{\ref{lem:cdg_coercivity}} hold with $\kappa<(1-\eps)^{-1}$. Then, for every $z_h$, $v_h\in\Vth$, we have
\begin{equation}\label{eq:adg_monotonicity}
\norm{z_h-v_h}_{h,1}^2 \leq C \left( \Adg(z_h;z_h-v_h)-\Adg(v_h;z_h-v_h)\right),
\end{equation}
where the constant $C\coloneqq 2\kappa/(1-\kappa \left(1-\eps\right))$.
Moreover, $\Adg$ is interval-wise Lipschitz continuous, in the sense that there exists a constant $C$, independent of the discretisation parameters, such that, for any $I_n\in\It$ and any $u_h$, $v_h$ and $z_h\in \Vth$ with support contained in $\overline{I_n}$, we have
\begin{equation}\label{eq:adg_lipschitz}
\abs{ \Adg(u_h;z_h)-\Adg(v_h;z_h)} \leq C \norm{u_h-v_h}_{h,1} \norm{z_h}_{h,1}.
\end{equation}
Therefore, there exists a unique solution $u_h\in\Vth$ of the numerical scheme \eqref{eq:numerical_scheme}. 
\end{theorem}
\begin{proof}
We begin by showing strong monotonicity of the nonlinear form $\Adg$. Let $z_h$, $v_h\in\Vth$ and set $w_h\coloneqq z_h-v_h$.
Then, by \eqref{eq:adg_definition} and Lemma~\ref{lem:cdg_coercivity}, we have
\begin{multline*}
\Adg(z_h;w_h)-\Adg(v_h;w_h) =  \Cdg(w_h,w_h) \\+\SIn \SK \pair{\Fg[z_h]-\Fg[v_h]-\Lw w_h}{\Lw w_h}_K\, \d t.
\end{multline*}
Lemma~\ref{lem:cordes_bound} and Young's Inequality show that
\begin{multline*}
\SIn \SK \abs{\pair{\Fg[z_h]-\Fg[v_h]-\Lw w_h}{\Lw w_h}_K} \d t \leq \frac{1}{2} \SIn \SK  \norm{\Lw w_h}_{L^2(K)}^2  \d t
\\ 
+ \frac{1-\eps}{2} \SIn \SK \omega^2 \norm{\p_t w_h}_{L^2(K)}^2+\abs{w_h}_{H^2(K),\lambda}^2  \d t.
\end{multline*}
Since $1<\kappa<(1-\eps)^{-1}$, Lemma~\ref{lem:cdg_coercivity} implies that
\begin{multline}\label{eq:adg_monotonicity_1}
\Adg(z_h;w_h)-\Adg(v_h;w_h) 
\geq \frac{1}{C} \SIn \SK \omega^2 \norm{\p_t w_h}_{L^2(K)}^2+\abs{w_h}_{H^2(K)}^2  \d t
\\ + \frac{1}{2}\SIn  \absJ{w_h}^2 \, \d t
+ \frac{\omega}{2} \sum_{n=0}^{N} \normah{\lld w_h \rrd}^2,
\end{multline}
where $C= 2\kappa /(1-\kappa\left(1-\eps\right))\geq 2$, thus showing \eqref{eq:adg_monotonicity}.

To show \eqref{eq:adg_lipschitz}, consider $u_h$, $v_h$ and $z_h \in \Vth$ that all have support in $\overline{I_n}$, and set $w_h\coloneqq u_h-v_h$. It then follows from $\supp v_h\subset\overline{I_n}$ that
\begin{multline*}
\norm{v_h}_{h,1}^2 = \int_{I_n} \SK\left[ \omega^2 \norm{\p_t v_h}_{L^2(K)}^2 + \abs{v_h}_{H^2(K),\lambda}^2\right] + \absJ{v_h}^2\,\d t \\+ \omega\normah{v_h(t_n)}^2+\omega\normah{v_h(t_{n-1}^+)}^2,
\end{multline*}
and similarly for $u_h$ and $z_h$.
We also have
\begin{multline*}
\Adg(u_h;z_h)-\Adg(v_h;z_h) = \int_{I_n} \SK \pair{\Fg[u_h]-\Fg[v_h]}{\Lw z_h}_K \, \d t
  + \CFh(w_h,z_h) \\
+ \int_{I_n} \Bdg(w_h,z_h)-\SK\pair{\Ll w_h}{\Ll z_h}_K\,\d t 
+ \omega \,a_h(w_h(t_{n-1}^+),z_h(t_{n-1}^+)).
\end{multline*}
Lipschitz continuity of $\Fg$ implies that
\[
 \int_{I_n} \SK \abs{ \pair{\Fg[u_h]-\Fg[v_h]}{\Lw z_h}_K} \, \d t\lesssim \norm{w_h}_{h,1}\norm{z_h}_{h,1}.
\]
Furthermore, we have $\abs{\CFh(w_h,z_h)} \leq E_1+E_2 $, where 
\begin{align*}
 E_1 & \coloneqq \omega \int_{I_n} \sum_{F\in\calF_h^i} \abs{\pair{\llb \nabla w_h \cdot n_F\rrb}{\lla\p_t z_h\rra}_F } \, \d t,
\\ E_2 & \coloneqq \omega \int_{I_n} \sum_{F\in\calFhib}\mu_F \abs{\pair{ \llb w_h\rrb}{\llb \p_t z_h\rrb}_F}+\abs{\pair{\llb w_h \rrb}{ \lla \nabla \p_t z_h \cdot n_F \rra}_F}\,\d t.
\end{align*}
The shape-regularity of the meshes $\{\calT\}_h$, the mesh assumption \eqref{eq:card_F_bound} and the trace and inverse inequalities show that
\begin{equation*}
\begin{split}
E_1& \lesssim \left(\int_{I_n} \SK \omega^2\,\norm{\p_t z_h}_{L^2(K)}^2  \d t\right)^{1/2} \left(\int_{I_n} \sum_{F\in\calF_h^i} \frac{\tilde{p}^2_F}{\tilde{h}_F} \norm{\llb \nabla w_h\cdot n_F \rrb}_{L^2(F)}^2  \d t\right)^{1/2},
\\ E_2 & \lesssim \left(\int_{I_n}  \SK \omega^2\,\norm{\p_t z_h}_{L^2(K)}^2  \d t \right)^{1/2}
\left(\int_{I_n} \sum_{F\in\calFhib} \frac{\tilde{p}^6_F}{\tilde{h}^3_F} \norm{\llb w_h\rrb}_{L^2(F)}^2 \d t\right)^{1/2}.
\end{split}
\end{equation*}
Since $\mu_F$ and $\eta_F$ satisfy \eqref{eq:mu_eta}, we conclude that $\abs{\CFh(w_h,z_h)}\lesssim \norm{w_h}_{h,1}\,\norm{z_h}_{h,1}$.
By applying trace and inverse inequalities on the flux terms of the bilinear form $B_{h,*}$, it is found that
\[
\abs{B_{h,*}(w_h,z_h)}\lesssim \left( \SK \abs{w_h}_{H^2(K),\lambda}^2+\absJ{w_h}^2\right)^{1/2} \left( \SK \abs{z_h}_{H^2(K),\lambda}^2+\absJ{z_h}^2\right)^{1/2}.
\]
Therefore, $\int_{I_n}\abs{ \Bdg(w_h,z_h)}+\SK\abs{\pair{\Ll w_h}{\Ll z_h}_K}\,\d t \lesssim \norm{u_h-v_h}_{h,1}\,\norm{z_h}_{h,1}$, thus completing the proof of \eqref{eq:adg_lipschitz}.

Since the numerical scheme \eqref{eq:numerical_scheme} is equivalent to solving \eqref{eq:numerical_scheme_time-stepping} for each $I_n\in\It$, and since $\Adg$ is strongly monotone and Lipschitz continuous on the subspace of $\Vth$ of functions with support in $\overline{I_n}$, for each $I_n\in\It$, repeated applications of the Browder--Minty Theorem show that there exists a unique $u_h\in\Vth$ that solves \eqref{eq:numerical_scheme}.
\qed
\end{proof}

\section{Error analysis}\label{sec:error_analysis}
The techniques of error analysis in the literature on discontinuous Galerkin time discretizations of parabolic equations often require sufficient temporal regularity of the exact solution \cite{Akrivis2004,Schotzau2000}, which, in the present setting, would correspond to the case where $u$ is at least in $H^1(I_n;H)$ for each $I_n\in\It$.
In the first part of this section, we present error bounds for regular solutions, where it is found that the method has convergence orders that are optimal with respect to $h$, $\tau$ and $\bo q$, and that are possibly suboptimal with respect to $\bo p$ by an order and a half. In a second part, we use Cl\'{e}ment quasi-interpolants in Bochner spaces to extend the analysis under weaker regularity assumptions, in order to cover the case where $u\notin H^1(I_n;H)$.

Our reasons for presenting the error analysis in two parts are twofold. First, the error analysis for regular solutions is simpler and permits the use of known approximation theory from \cite{Schotzau2000}, whereas the case of rough solutions requires the additional construction of a Cl\'{e}ment quasi-interpolation operator. Second, the Cl\'{e}ment operator is generally suboptimal by one order in $\tau $ when applied to solutions with higher temporal regularity. Thus, the results given here for regular and rough solutions are complementary to each other.

We will present error bounds in the norm $\norme{\cdot}$ defined by
\begin{equation}
\norme{v}^2 \coloneqq \SIn\SK \left[\omega^2\norm{\p_t v}_{L^2(K)}^2 + \abs{v}_{H^2(K),\lambda}^2\right] + \absJ{v}^2\,\d t + \omega \sum_{n=0}^{N-1}\normah{\lld v \rrd}^2.
\end{equation}
We remark that for $v_h\in\Vth$, we have $\norm{v_h}_{h,1}^2=\norme{v_h}^2 + \omega \normah{\lld v_h \rrd[N]}^2$. Error bounds in the norm $\norm{\cdot}_{h,1}$ can be shown under additional regularity assumptions for the solution at time $T$.
To simplify the notation in this section, let
\begin{equation}
\begin{aligned}
X_0&\coloneqq L^2(\Om),
& X_1&\coloneqq H^1_0(\Om),
& X_2 &\coloneqq H=H^2(\Om)\cap H^1_0(\Om).
\end{aligned}
\end{equation}
Similarly to the definition of the broken Sobolev spaces $H^s(\Om;\calT_h)$, for a Hilbert space $X$, we define the broken Bochner space $H^\sigma(I;X;\It)$ to be the space of functions $u\in L^2(I;X)$ with restrictions $\eval{u}{I_n}\in H^\sigma(I_n;X)$ for each $I_n\in\It$. We equip $H^\sigma(I;X;\It)$ with the obvious norm.

Since the error bounds presented below are given in a very general and flexible form, it can be helpful to momentarily consider their implications for the case of smooth solutions approximated on quasi-uniform meshes and time-partitions with uniform polynomial degrees. In this setting, it can be seen that Theorem~\ref{thm:error_bound_regular} below implies that
\begin{multline}\label{eq:error_bound_smooth}
\norme{u-u_h} \lesssim h^{p-1}\sum_{\ell=0}^1 \norm{u}_{H^\ell(I;H^{p+1-2\ell}(\Om;\calT_h))}\\  + h^{p}\norm{u_0}_{H^{p+1}(\Om;\calT_h)}
   + \tau^q \sum_{\ell\in\{0,2\}} \norm{u}_{H^{q+1-\ell/2}(I;X_\ell;\It)}. 
  \end{multline}
The bound \eqref{eq:error_bound_smooth} suggests combinations of the mesh sizes and polynomial degrees that are optimal in terms of balancing the approximation orders. For example, if $p=2q+1$, then the error bound is of order $(h^2+\tau)^q=(h+\sqrt{t})^{p-1}$, so an optimal method is found by choosing $\tau\simeq h^2$. Alternatively, choosing $p=q+1$ and $\tau\simeq h$ leads to an optimal method of order $h^{p-1}\simeq\tau^q$.

\subsection{Regular solutions}\label{sec:error_analysis_regular}
If the solution $u$ of \eqref{eq:HJB} belongs to $H^1(I;H,\It)$, then the error analysis may be based on the following approximation result, found for instance in \cite{Schotzau2000}, albeit presented here in a form amenable to our purposes.
\begin{theorem}\label{thm:schotzau_schwab_projector}
Let $\Om\subset\R^\dim$ be a bounded convex domain, and let $\{\It\}_{\tau}$ be a sequence of regular partitions of $I=(0,T)$. For each $\tau$, let $\bo q=(q_1,\dots,q_N)$ be a vector of positive integers. Then, for each $\tau$, there exists a linear operator $\pq \colon \HIO\cap H^1(I;H;\It)\tends \Vt$ such that the following holds. The operator $\pq $ is an interpolant at the interval endpoints, i.e.\ for any $u\in \HIO\cap H^1(I;H;\It)$, we have $\pq u(t_n)=\pq u(t_n^+)=u(t_n)$ for each $0\leq n \leq N$. For any $I_n\in\It$, any $\ell\in\{0,1,2\}$, any real number $\sigma_{n,\ell}\geq 1$ and any $j\in\{0,1\}$, we have
\begin{equation}\label{eq:schotzau_schwab_approximation}
\norm{u-\pq u}_{H^{j}(I_n;\Xell)}\lesssim \frac{\tau_{n}^{\varrho_{n,\ell}-j}}{q_n^{\sigma_{n,\ell}-j}}\norm{u}_{H^{\sigma_{n,\ell}}(I_n;\Xell)}
\qquad \forall \,u\in H^{\sigma_{n,\ell}}(I_n;\Xell),
\end{equation}
where $\varrho_{n,\ell}\coloneqq \min(\sigma_{n,\ell},q_n+1)$, and where the constant depends only on $\sigma_{n,\ell}$ and $\max\tau$.
\end{theorem}
The construction of $\pq$ in the proof of Theorem~\ref{thm:schotzau_schwab_projector} involves the truncated Legendre series of $\p_t u$ and the values of $u$ at the partition points. Therefore, the requirement of $H^1(I;H;\It)$ regularity is used to ensure that $\eval{\pq }{I_n}$ maps into $\calQ_{q_n}(H)$. A different approximation operator is used in section~\ref{sec:error_analysis_limited} to perform an analysis under weaker regularity assumptions.
\begin{theorem}\label{thm:error_bound_regular}
Let $\Om\subset \R^\dim$ be a bounded convex polytopal domain and let $\{\calT_h\}_h$ be a shape-regular sequence of simplicial or parallelepipedal meshes satisfying \eqref{eq:card_F_bound}, \eqref{eq:c_h_bound}, and let $\bo p=\left(p_K;\;K\in\calT_h\right)$ be a vector of positive integers such that \eqref{eq:c_p_bound} holds for each $h$, and such that $p_K\geq 2$ for all $K\in\calT_h$. Let $I=(0,T)$ and let $\{\It\}_{\tau}$ be a sequence of regular partitions of $I$, and, for each $\tau$, let $\bo q=(q_1,\dots,q_N)$ be a vector of positive integers.
Let $\Ld$ be a compact metric space and let the data $a$, $b$, $c$ and $f$ be continuous on $\Ob\times \overline{I}\times \Ld$ and satisfy \eqref{eq:uniform_ellipticity} and \eqref{eq:cordes_condition3}, or alternatively \eqref{eq:cordes_condition2} in the case where $b\equiv 0$ and $c\equiv 0$. Let $\mu_F$ and $\eta_F$ satisfy \eqref{eq:mu_eta}, with $\cstab$ chosen so that Lemmas~\ref{lem:space_dg_coercivity}~and~\ref{lem:cdg_coercivity} hold with $\kappa<(1-\eps)^{-1}$. 

Let $u\in \HIO$ be the unique solution of the HJB equation \eqref{eq:HJB}, and assume that $u \in L^2(I;H^{\bo s}(\Om;\calT_h))$ and $\p_t u \in L^2(I;H^{\overline{\bo s}}(\Om,\calT_h))$ for each $h$, with $s_K>5/2$ and $\overline{s}_K> 0$ for each $K\in\calT_h$.
Suppose also that, for each $\tau$, each $\ell\in\{0,2\}$ and each $I_n\in\It$, the function $\eval{u}{I_n}\in H^{\sigma_{n,\ell}}(I_n;\Xell)$ for some $\sigma_{n,\ell}\geq 1$. Assume that $u_0 \in H^1_0(\Om)\cap H^{\tilde{\bo s}}(\Om;\calT_h)$ with $\tilde{s}_K > 3/2$ for each $K\in\calT_h$.
Then, we have
\begin{multline}\label{eq:error_bound_regular}
\norme{u-u_h}^2  \lesssim  \SIn\SK \frac{h_K^{2t_K -4}}{p_K^{2 s_K-7}}\norm{u}_{H^{s_K}(K)}^2+\frac{h_K^{2\overline{t}_K}}{p_K^{2{\overline{s}_K}}}\norm{\p_t u}_{H^{\overline{s}_K}(K)}^2\d t
\\ +\max_{K\in\calT_h}p_K^3 \sum_{n=1}^N \sum_{\ell\in\{0,2\}}\frac{\tau_n^{2\varrho_{n,\ell}-2+\ell}}{q_n^{2\sigma_{n,\ell}-2+\ell}}\norm{u}_{H^{\sigma_{n,\ell}}(I_n;\Xell)}^2
+\SK \frac{h_K^{2\tilde{t}_K-2}}{p_K^{2\tilde{s}_K-3}}\norm{u_0}_{H^{\tilde{s}_K}(K)}^2,
\end{multline}
with a constant independent of $u$, $h$, $\bo p$, $\tau$ and $\bo q$, and where $t_K\coloneqq\min(s_K,p_K+1)$, $\overline{t}_K\coloneqq\min(\overline{s}_K,p_K+1)$ and $\tilde{t}_K\coloneqq \min(\tilde{s}_K,p_K+1)$ for each $K\in\calT_h$, and where $\varrho_{n,\ell}\coloneqq\min(\sigma_{n,\ell},q_n+1)$ for each $1\leq n\leq N$ and each $\ell\in\{0,2\}$.
\end{theorem}

Since the norm $\norme{\cdot}$ comprises the broken $H^2$-seminorm in space and a broken $H^1$-norm in time, it is seen that the error bound is optimal with respect to $h$, $\tau$ and $\bo q$, but is suboptimal with respect to $\bo p$ by an order and a half. We remark that since Theorem~\ref{thm:error_bound_regular} assumes $u\in H^1(I_1;H)$, the initial data satisfies $u_0 \in H$, so we may take $\tilde{s}_K\geq 2$ for each $K\in\calT_h$.
\begin{proof}
The approximation theory for $hp$-version discontinuous Galerkin finite element spaces (see Appendix~\ref{sec:approximation}) shows that there exists a sequence of linear projection operators $\{\php\}_h$, with $\php\colon L^2(\Om)\tends \Vh$ and such that for each $K\in\calT_h$, for each nonnegative real number $r_K\leq \max(s_K,\overline{s}_K,\tilde{s}_K)$ and for each nonnegative integer $j\leq r_K$, and if $r_K>1/2$, for each multi-index $\beta$ such that $\abs{\beta}<r_K-1/2$, we have
\begin{align}
\norm{u-\php u}_{H^j(K)} &\lesssim \frac{h_K^{\min(r_K,p_K+1)-j}}{(p_K+1)^{r_K-j}}\norm{u}_{H^{r_K}(K)} & &\forall\, u\in H^{r_K}(K),
\label{eq:pmp_approximation_bound_2}
\\ \norm{D^\beta(u-\php u)}_{L^2(\p K)} &\lesssim \frac{h_K^{\min(r_K,p_K+1)-\abs{\beta}-1/2}}{(p_K+1)^{r_K-\abs{\beta}-1/2}} \norm{u}_{H^{r_K}(K)} & &\forall \,u\in H^{r_K}(K),\label{eq:pmp_trace_bound_2}
\end{align}
where the constant is independent of $r_K$, $h_K$, $p_K$ but possibly dependent on $s_K$, $\overline{s}_K$ and $\tilde{s}_K$. The technical form of this approximation result expresses the optimality and stability of $\php$ for functions in $H^{r_K}(K)$, $0\leq r_K \leq \max(s_K,\overline{s}_K,\tilde{s}_K)$. In particular, we will use the fact that $\php$ is elementwise $L^2$-stable, $H^1$-stable and $H^2$-stable in the analysis below.

For each $h$ and $\tau$, let $z_\tau \coloneqq \pq u \in \Vt$, and let $z_h\coloneqq \php z_\tau \in \Vth$.
Continuity of $z_\tau$ implies continuity of $z_h$, so that $\lld z_h \rrd = 0$ for each $1\leq n < N$. Furthermore, we have $z_\tau(0^+) = u_0$, so $z_h(0^+)=\php u_0$.
Let $\xi_h\coloneqq u-z_h$ and let $\psi_h \coloneqq u_h-z_h$, so that $u-u_h = \xi_h - \psi_h$. Recall that $\norme{\psi_h}\leq \norm{\psi_h}_{h,1}$. Theorem~\ref{thm:dg_stability}, the scheme \eqref{eq:numerical_scheme} and Corollary~\ref{cor:scheme_consistency} show that
\begin{multline}\label{eq:error_bound_regular_0}
\norm{\psi_h}_{h,1}^2 \lesssim A_h(u_h;\psi_h)-A_h(z_h;\psi_h)=A_h(u;\psi_h)-A_h(z_h;\psi_h)
\\  = \SIn \SK \pair{\Fg[u]-\Fg[z_h]}{\Lw \psi_h}_K
+\Bdg(\xi_h,\psi_h)\,\d t
\\ - \SIn\SK \pair{\Ll \xi_h}{\Ll \psi_h}_{K}\,\d t
  + \CFh(\xi_h,\psi_h) + \omega \,a_h(\xi_h(t_0^+),\psi_h(t_0^+)).
\end{multline}
Therefore $\norme{\psi_h}^2\leq \norm{\psi_h}_{h,1}^2\leq \sum_{i=1}^4 D_i$, where the quantities $D_i$, $1\leq i \leq 4$, are defined by
\begin{gather*}
D_1 \coloneqq \SIn\SK\abs{ \pair{\Fg[u]-\Fg[z_h]}{\Lw \psi_h}_K} + \abs{\pair{\Ll \xi_h}{\Ll \psi_h}_{K}}\d t,
\\
D_2 \coloneqq \SIn\abs{\Bdg(\xi_h,\psi_h)} \d t,\quad
 D_3 \coloneqq  \abs{\CFh(\xi_h,\psi_h)}, \quad 
D_4  \coloneqq \omega \abs{a_h(\xi_h(0^+),\psi_h(0^+))}.
\end{gather*}
Lipschitz continuity of $\Fg$ implies that $D_1 \lesssim \sqrt{E_1+E_2}\; \norm{\psi_h}_{h,1}$, where $E_1$ and $E_2$ are defined by
\begin{equation*}\label{eq:E1E2_def}
\begin{aligned}
E_1&\coloneqq \SIn\SK \norm{\p_t \xi_h}_{L^2(K)}^2\d t,
& E_2&\coloneqq \SIn\SK \norm{\xi_h}_{H^2(K)}^2 \d t.
\end{aligned}
\end{equation*}
Since the sequence of meshes $\{\calT_h\}_h$ is shape-regular and since $\eval{\psi_h}{I_n} \in \calQ_{q_n}(\Vh)$ for each $I_n\in\It$, the use of trace and inverse inequalities on the flux terms appearing in $\Bdg(\xi_h,\psi_h)$ yields $D_2 \lesssim \sqrt{\sum\nolimits_{i=2}^6 E_i}\;\norm{\psi_h}_{h,1}$, where the quantities $E_i$, $3\leq i \leq 5$, are defined by
\begin{align*}
E_3&\coloneqq \SIn\sum_{F\in\calF_h^i}\mu_F^{-1}\norm{\DivT\nablaT\lla\xi_h\rra}_{L^2(F)}^2+ \sum_{F\in\calFhib}\mu_F^{-1} \norm{\nablaT\lla \nabla \xi_h\cdot n_F\rra}_{L^2(F)}^2\d t,
\\ E_4&\coloneqq \SIn\sum_{F\in\calFhib}\eta_F^{-1}\norm{\lla\nabla \xi_h\cdot n_F\rra}_{L^2(F)}^2+\sum_{F\in\calF_h^i}\mu_F^{-1}\norm{\lla \xi_h \rra}_{L^2(F)}^2\d t,
\\ E_5&\coloneqq \SIn\sum_{F\in\calF_h^i} \mu_F \norm{\llb \nabla \xi_h\cdot n_F\rrb}_{L^2(F)}^2+\sum_{F\in \calFhib} \mu_F \norm{\llb \nablaT \xi_h\rrb}_{L^2(F)}^2\d t,
\\ E_6 &\coloneqq \SIn\sum_{F\in\calFhib} \eta_F \norm{\llb \xi_h\rrb}_{L^2(F)}^2\d t.
\end{align*}
Note that $\eval{\p_t \psi_h}{I_n}\in \calQ_{q_n-1}(\Vh)$ for each $I_n\in\It$. Thus, similarly to the proof of Theorem~\ref{thm:dg_stability}, the use of trace and inverse inequalities leads to $D_3\lesssim \sqrt{E_4+E_5}\,\norm{\psi_h}_{h,1}$.
It follows from \eqref{eq:a_h_coercivity} that we have $D_4 \lesssim \sqrt{E_6+E_7+E_8}\,\norm{\psi_h}_{h,1}$, where the quantities $E_i$, $7\leq i \leq 9$, are defined by
\begin{align*}
E_7 &\coloneqq \SK \norm{u_0-\php u_0}_{H^1(K)}^2, & 
 E_8 &\coloneqq \sum_{F\in\calFhib} \mu_F\norm{u_0-\php u_0}_{L^2(F)}^2, \\
 E_9 &\coloneqq \sum_{F\in\calFhib} \mu_F^{-1} \norm{\lla \nabla(u_0 - \php u_0)\cdot n_F\rra}_{L^2(F)}^2. & &
\end{align*}
Therefore, \eqref{eq:error_bound_regular_0} implies that $\norme{\psi_h}^2\lesssim \sum_{i=1}^9 E_i$.
The properties of the operator $\php$, namely its linearity, $L^2$-stability and approximation properties \eqref{eq:pmp_approximation_bound_2}, together with \eqref{eq:schotzau_schwab_approximation}, imply that
\begin{equation}
\begin{split}
E_1 & \lesssim \SIn \SK  \norm{\p_t u-\php \p_t u}_{L^2(K)}^2+\norm{\php(\p_t u-\p_t z_\tau)}_{L^2(K)}^2 \d t
\\ &\lesssim \SIn\SK \norm{\p_t u-\php \p_t u}_{L^2(K)}^2\d t
+ \sum_{n=1}^N\norm{u-z_\tau}_{H^1(I_n;X_0)}^2
\\ &\lesssim \SIn \SK \frac{h_K^{2\overline{t}_K}}{p_K^{2 \overline{s}_K}}\norm{\p_t u}_{H^{\overline{s}_K}(K)}^2\d t
+ \sum_{n=1}^N \frac{\tau_n^{2\varrho_{n,0}-2}}{q_n^{2\sigma_{n,0}-2}} \norm{u}_{H^{\sigma_{n,0}}(I_n;X_0)}^2.
\end{split}
\end{equation}
Since the operator $\php$ is elementwise $H^2$-stable, it is found that
\begin{equation}
\begin{split}
E_2& \lesssim \SIn \SK \norm{u-\php u}_{H^2(K)}^2+\norm{\php(u-z_\tau)}_{H^2(K)}^2 \d t
\\ &\lesssim \SIn\SK \norm{u-\php u}_{H^2(K)}^2\d t  + \sum_{n=1}^N\norm{u-z_\tau}_{L^2(I_n;X_2)}^2
\\ &\lesssim\SIn\SK \frac{h_K^{2t_K-4}}{p_K^{2s_K-4}}\norm{u}_{H^{s_K}(K)}^2\d t + \sum_{n=1}^N\frac{\tau_n^{2\varrho_{n,2}}}{q_n^{2\sigma_{n,2}}}\norm{u}_{H^{\sigma_{n,2}}(I_n;X_2)}^2.
\end{split}
\end{equation}
The mesh assumptions \eqref{eq:card_F_bound}, \eqref{eq:c_h_bound} and \eqref{eq:c_p_bound}, the bound \eqref{eq:pmp_trace_bound_2}, and the application of trace and inverse inequalities on $\eval{\php(u-z_\tau)}{I_n}\in \calQ_{q_n}(\Vh)$, imply that
\begin{equation}
\begin{split}
E_3&\lesssim\SIn\SK \frac{h_K}{p_K^2} \norm{ D^2(u-\php z_\tau)}_{L^2(\p K)}^2\,\d t
\\ &\lesssim \SIn \SK \frac{h_K}{p_K^2}\left[ \norm{D^2(u-\php u) + D^2\php(u-z_\tau)}_{L^2(\p K)}^2\right]\d t 
\\ &\lesssim \SIn\SK \frac{h_K^{2t_K-4}}{p_K^{2s_K-3}}\norm{u}_{H^{s_K}(K)}^2+ \SK\norm{u-z_\tau}_{H^2(K)}^2\d t
\\ &\lesssim \SIn\SK \frac{h_K^{2t_K-4}}{p_K^{2s_K-3}}\norm{u}_{H^{s_K}(K)}^2\d t + \sum_{n=1}^N  \frac{\tau_n^{2\varrho_{n,2}}}{q_n^{2\sigma_{n,2}}}\norm{u}_{H^{\sigma_{n,2}}(I_n;X_2)}^2.
\end{split}
\end{equation}
Similarly to $E_3$, we find that
\begin{equation}
E_4 \lesssim  \SIn\SK \frac{h_K^{2t_K}}{p_K^{2s_K+1}}\norm{u}_{H^{s_K}(K)}^2\,\d t + \sum_{n=1}^N \frac{\tau_n^{2\varrho_{n,0}}}{q_n^{2\sigma_{n,0}}}\norm{u}_{H^{\sigma_{n,0}}(I_n;X_0)}^2.
\end{equation}
The spatial regularity of $u$ and $z_\tau$ imply that
\begin{multline*}
E_5=\SIn\sum_{F\in\calF_h^i}
\mu_F \norm{\llb \nabla \left[ u-\php u + \php(u-z_\tau)-(u-z_\tau)\right]\cdot n_F\rrb}_{L^2(F)}^2\d t
\\ +\SIn\sum_{F\in\calFhib} \mu_F \norm{\llb \nablaT \left[ u-\php u + \php(u-z_\tau)-(u-z_\tau)\right] \rrb}_{L^2(F)}^2\d t.
\end{multline*}
Therefore, the mesh assumptions \eqref{eq:card_F_bound}, \eqref{eq:c_h_bound} and \eqref{eq:c_p_bound} and the approximation bound \eqref{eq:pmp_trace_bound_2} yield
\begin{equation}
\begin{split}
E_5 &\lesssim \sum_{n=1}^N\int\limits_{I_n}\SK \frac{p_K^2}{h_K}\norm{\nabla(u-\php u)+\nabla \left[u-z_\tau-\php(u-z_\tau)\right]}_{L^2(\p K)}^2\d t
\\ &\lesssim \SIn\SK\frac{h_K^{2t_K-4}}{p_K^{2s_K-5}}\norm{u}_{H^{s_K}(K)}^2 + \SK p_K\norm{u-z_\tau}_{H^2(K)}^2\d t
\\ &\lesssim \SIn\SK\frac{h_K^{2t_K-4}}{p_K^{2s_K-5}}\norm{u}_{H^{s_K}(K)}^2  \d t
 + \max_{K\in\calT_h}p_K \sum_{n=1}^N\frac{\tau_n^{2\varrho_{n,2}}}{q_n^{2\sigma_{n,2}}}\norm{u}_{H^{\sigma_{n,2}}(I_n;X_2)}^2.
\end{split}
\end{equation}
Likewise, it follows from the spatial regularity of $z_\tau$, the mesh assumptions, and the approximation bound \eqref{eq:pmp_trace_bound_2} that
\begin{equation}
\begin{split}
E_6 &\lesssim \SIn\SK \frac{p_K^6}{h_K^3} \norm{u-\php u + \php(u-z_\tau)-(u-z_\tau)}_{L^2(\p K)}^2\d t
\\ & \lesssim \SIn\SK \frac{h_K^{2t_K-4}}{p_K^{2 s_K-7}}\norm{u}_{H^{s_K}(K)}^2 + \SK p_K^3 \norm{u-z_\tau}_{H^2(K)}^2\d t
\\ & \lesssim \SIn\SK \frac{h_K^{2t_K-4}}{p_K^{2 s_K-7}}\norm{u}_{H^{s_K}(K)}^2 \d t
+ \max_{K\in\calT_h} p_K^3 \sum_{n=1}^N \frac{\tau_n^{2\varrho_{n,2}}}{q_n^{2\sigma_{n,2}}}\norm{u}_{H^{\sigma_{n,2}}(I_n;X_2)}^2.
\end{split}
\end{equation}
Finally, it is readily shown that
\begin{equation}
\sum_{i=7}^9 E_i \lesssim \SK \frac{h_K^{2\tilde{t}_K-2}}{p_K^{2\tilde{s}_K-3}}\norm{u_0}_{H^{\tilde{s}_K}(K)}^2.
\end{equation}
Since $\norme{\xi_h}^2\leq \sum_{i=1}^9 E_i$, the above bounds and the triangle inequality $\norme{u-u_h}\leq \norme{\xi_h}+\norme{\psi_h}$ complete the proof of \eqref{eq:error_bound_regular}.
\qed
\end{proof}
\subsection{Rough solutions}\label{sec:error_analysis_limited}
The proof of Theorem~\ref{thm:error_bound_regular} depends on the approximation result from Theorem~\ref{thm:schotzau_schwab_projector}, which requires that the solution $u$ belongs to $H^1(I;H;\It)$. In this section, we relax this condition by using a Cl\'{e}ment quasi-interpolation result instead of Theorem~\ref{thm:schotzau_schwab_projector}.

For $\It$ a regular partition of $(0,T)$, let $\{\hf_m\}_{m=0}^N$ denote the set of hat functions of $\It$, i.e.\ $\hf_m$ is the unique piecewise-affine function on $\It$ such that $\hf_m(t_n)=\delta_{nm}$ for $0\leq n,m\leq N$. For $0\leq m\leq N$, let $J_m \coloneqq \supp \hf_m$, and note that $J_m=\overline{I_m}\cup\overline{I_{m+1}}$ for $1\leq n < N$, whilst $J_0=\overline{I_1}$ and $J_N=\overline{I_N}$.
\begin{theorem}\label{thm:clement_bochner_approximation}
Let $\Om\subset\R^\dim$ be a bounded convex domain, and let $\{\It\}_{\tau}$ be a sequence of regular partitions of $I=(0,T)$. For each $\tau$, let $\bo q = (q_1,\dots,q_N)$ be a vector of positive integers. Suppose that there exist positive constants $c_\tau$ and $c_{q}$ such that, for each $\tau$, we have\begin{equation}\label{eq:time_partition_regularity}
\begin{aligned}
&\frac{1}{c_\tau} \leq \frac{\tau_{n-1}}{\tau_{n}}\leq c_\tau,
& \frac{1}{c_q} \leq \frac{q_{n-1}}{q_{n}}\leq c_q, & &2\leq n\leq N.
\end{aligned}
\end{equation}
Let $u\in L^2(I;H)$ and suppose that $\eval{u}{J_m} \in H^{\sigma_{m,\ell}}(J_m;\Xell)$ for some $\sigma_{m,\ell}\in\R_{\geq 0}$ for each $\ell\in\{0,1,2\}$ and each $0\leq m\leq N$.
Then, there exists a sequence of functions $\{z_\tau\}_\tau$, such that $z_\tau \in \Vt$ for each $\tau$, and such that the following properties hold.
The functions $z_\tau$ are continuous on $I$, i.e.\ $\lld z_\tau \rrd =0$ for each $1\leq n < N$. For each $\ell\in\{0,1,2\}$ and each $I_n\in\It$, we have
\begin{equation}\label{eq:clement_bochner_l2_stability}
\norm{z_\tau}_{L^2(I_n;\Xell)}\lesssim \sum_{J_m\supset I_n} \norm{u}_{L^2(J_m;\Xell)},
\end{equation}
where the constant is independent of all other quantities.
For each $\ell\in\{0,1,2\}$, each $I_n\in\It$ and each nonnegative integer $j\leq \min_{J_m\supset I_n} \sigma_{m,\ell}$, we have
\begin{equation}\label{eq:clement_bochner_approximation}
\norm{u- z_\tau}_{H^j(I_n;\Xell)}\lesssim 
\sum_{J_m\supset I_n}\frac{\tau_n^{\varrho_{m,\ell}-j}}{q_n^{\sigma_{m,\ell}-j}}\norm{u}_{H^{\sigma_{m,\ell}}(J_m;\Xell)},
\end{equation}
where $\varrho_{m,\ell}\coloneqq \min(\sigma_{m,\ell},\min_{I_n\subset J_m} q_n)$, and the constant depends only on $\max\sigma_{m,\ell}$, $\max \tau$, $c_\tau$ and $c_q$.
\end{theorem}
\begin{proof}
For $0\leq m\leq N$, define $\bar{q}_m\coloneqq \min_{I_n\subset J_m} q_n$, and note $\bar{q}_m\geq 1$ for all $m$ since $q_n\geq 1$ for all $n$.
Since $u\in L^2(J_m;X_2)$ for each $m$, standard approximation theory for Bochner spaces (see Appendix~\ref{sec:approximation}) implies that there exist functions $v_m\in \calQ_{\bar{q}_m-1}(H)$, $0\leq m\leq N$, with the following properties.
For each $\ell\in\{0,1,2\}$, we have $\norm{v_m}_{L^2(J_m;\Xell)}\lesssim \norm{u}_{L^2(J_m;\Xell)}$, with a constant independent of all other quantities.
For each $\ell\in\{0,1,2\}$ and each nonnegative integer $j\leq \sigma_{m,\ell}$, we have
\begin{equation}\label{eq:clement_bochner_2}
\norm{u-v_m}_{H^{j}(J_m;\Xell)}\lesssim \frac{\abs{J_m}^{\varrho_{m,\ell}-j}}{\bar{q}_m^{\sigma_{m,\ell}-j}}\norm{u}_{H^{\sigma_{m,\ell}}(J_m;\Xell)},
\end{equation}
where $\varrho_{m,\ell}\coloneqq \min(\sigma_{m,\ell},\bar{q}_m)$, where $\abs{J_m}$ is the length of the interval $J_m$, and where the constant depends only on $\max\sigma_{m,\ell}$ and $\max\tau$.

The hypothesis~\eqref{eq:time_partition_regularity} and the bound~\eqref{eq:clement_bochner_2} imply that, for each $I_n\subset J_m$, each $\ell\in\{0,1,2\}$ and each nonnegative integer $j\leq\sigma_{m,\ell}$,
\begin{equation}\label{eq:clement_bochner_3}
\norm{u-v_m}_{H^j(I_n;\Xell)}\lesssim \frac{\tau_n^{\varrho_{m,\ell}-j}}{q_n^{\sigma_{m,\ell}-j}}\norm{u}_{H^{\sigma_{m,\ell}}(J_m;\Xell)},\end{equation}
where the constant depends only on $\max\sigma_{m,\ell}$, $\max\tau$, $c_\tau$ and $c_q$.

Define $z_\tau \coloneqq \sum_{m=0}^N \hf_m v_m$, where $\hf_m$ is the hat function over the interval $J_m$. Note that we have $\eval{v_m}{I_n}\in \calQ_{q_n-1}(H)$ for each $I_n\in\It$ since $\bar{q}_m\leq q_n$ for each $I_n\subset J_m$. Since $\hf_m$ is piecewise affine, it follows that $\eval{z_\tau}{I_n}\in \calQ_{q_n}(H)$ for each $I_n \in \It$, thereby showing that $z_\tau \in \Vt$. Furthermore, it is clear that $z_\tau$ is continuous on $I$, i.e.\ $\lld z_\tau \rrd =0$ for each $1\leq n \leq N-1$.
The bound \eqref{eq:clement_bochner_l2_stability} follows from $\norm{v_m}_{L^2(J_m;\Xell)}\lesssim \norm{u}_{L^2(J_m;\Xell)}$ and from the fact that $\norm{\hf_m}_{L^\infty(I)}=1$ for each $0\leq m \leq N$.
Since $\{\hf_m\}_{m=0}^N$ forms a partition of unity, the bound \eqref{eq:clement_bochner_3} implies that, for each $I_n\in \It$ and each $\ell\in\{0,1,2\}$, 
\begin{multline*}
\norm{u-z_\tau}_{L^2(I_n;\Xell)}\leq\sum_{J_m\supset I_n}\norm{\hf_m (u-v_m)}_{L^2(I_n;\Xell)} 
\\ \lesssim \sum_{J_m\supset I_n} \norm{u-v_m}_{L^2(I_n;\Xell)}
\lesssim \sum_{J_m\supset I_n} \frac{\tau_n^{\varrho_{m,\ell}}}{q_n^{\sigma_{m,\ell}}} \norm{u}_{H^{\sigma_{m,\ell}}(J_m;\Xell)},
\end{multline*}
and, for each integer $1\leq j\leq \min_{J_m\supset I_n}\sigma_{m,\ell}$,
\begin{multline*}
\abs{u-z_\tau}_{H^j(I_n;\Xell)}  \leq \sum_{J_m\supset I_n}\abs{\hf_m (u-v_m)}_{H^{j}(I_n;\Xell)}
\\  \lesssim \sum_{J_m\supset I_n} \abs{u-v_m}_{H^j(I_n;\Xell)}+\frac{1}{\tau_n}\abs{u-v_m}_{H^{j-1}(I_n;\Xell)}
 \lesssim \sum_{J_m\supset I_n}\frac{\tau_n^{\varrho_{m,\ell}-j}}{q_n^{\sigma_{m,\ell}-j}}\norm{u}_{H^{\sigma_{m,\ell}}(J_m;\Xell)}.\end{multline*}
This completes the proof of \eqref{eq:clement_bochner_approximation}.\qed
\end{proof}

\begin{theorem}\label{thm:error_bound_limited}
Let $\Om\subset\R^\dim$ be a bounded convex polytopal domain and let $\{\calT_h\}_h$ be a shape-regular sequence of simplicial or parallelepipedal meshes satisfying \eqref{eq:card_F_bound}, \eqref{eq:c_h_bound}, and let $\bo p=\left(p_K;\;K\in\calT_h\right)$ be a vector of positive integers satisfying \eqref{eq:c_p_bound} for each $h$ and such that $p_K\geq 2$ for each $K\in\calT_h$. Let $I=(0,T)$ and let $\{\It\}_{\tau}$ be a sequence of regular partitions of $I$, and, for each $\tau$, let $\bo q$ be a vector of positive integers such that \eqref{eq:time_partition_regularity} holds.
Let $\Ld$ be a compact metric space and let the data $a$, $b$, $c$ and $f$ be continuous on $\Ob\times \overline{I}\times \Ld$ and satisfy \eqref{eq:uniform_ellipticity} and \eqref{eq:cordes_condition3}, or alternatively \eqref{eq:cordes_condition2} in the case where $b\equiv 0$ and $c\equiv 0$. Let $\mu_F$ and $\eta_F$ satisfy \eqref{eq:mu_eta}, with $\cstab$ chosen so that Lemmas~\ref{lem:space_dg_coercivity}~and~\ref{lem:cdg_coercivity} hold with $\kappa<(1-\eps)^{-1}$. 

Let $u\in \HIO$ be the unique solution of the HJB equation \eqref{eq:HJB}, and assume that $u \in L^2(I;H^{\bo s}(\Om;\calT_h))$ and $\p_t u \in L^2(I;H^{\overline{\bo s}}(\Om,\calT_h))$ for each $h$, with $s_K>5/2$ and $\overline{s}_K>0$ for each $K\in\calT_h$.
Suppose also that, for each $\tau$, $\ell\in\{0,1,2\}$, and each $0\leq m\leq N$, the function $\eval{u}{J_m} \in H^{\sigma_{m,\ell}}(J_m;\Xell)$ for some real $\sigma_{m,\ell}\geq 0$, with $\sigma_{m,0}\geq 1$ for all $m$.  Assume that $u_0 \in H^1_0(\Om)\cap H^{\tilde{\bo s}}(\Om;\calT_h)$ with $\tilde{s}_K > 3/2$ for each $K\in\calT_h$.
Then, we have 
\begin{multline}\label{eq:nonsmooth_error_bound}
\norme{u-u_h}^2\lesssim \SIn\SK \frac{h_K^{2t_K -4}}{p_K^{2 s_K-7}}\norm{u}_{H^{s_K}(K)}^2 + \frac{h_K^{2\overline{t}_K}}{p_K^{2{\overline{s}_K}}}\norm{\p_t u}_{H^{\overline{s}_K}(K)}^2 \d t
\\ + \max_{K\in\calT_h}p_K^3\sum_{n=1}^N \sum_{\ell=0}^2\sum_{J_m\supset I_n}\frac{\tau_n^{2\varrho_{m,\ell}-2+\ell}}{q_n^{2\sigma_{m,\ell}-2+\ell}}\norm{u}_{H^{\sigma_{m,\ell}}(J_m;\Xell)}^2 + \SK \frac{h_K^{2\tilde{t}_K-2}}{p_K^{2\tilde{s}_K-3}}\norm{u_0}_{H^{\tilde{s}_K}(K)}^2,
\end{multline}
with a constant independent of $h$, $\bo p$, $\tau$, $\bo q$, and $u$, and where $t_K\coloneqq\min(s_K,p_K+1)$, $\overline{t}_K\coloneqq\min(\overline{s}_K,p_K+1)$, and $\tilde{t}_K\coloneqq \min(\tilde{s}_K,p_K+1)$ for each $K\in\calT_h$, and where $\varrho_{m,\ell}\coloneqq\min(\sigma_{m,\ell},\min_{I_n\subset J_m} q_n)$ for each $0\leq m\leq N$ and each $\ell\in\{0,1,2\}$.
\end{theorem}
\begin{proof}
For each $h$, let $\php\colon L^2(\Om)\tends \Vh$ denote the approximation operator of the proof of Theorem~\ref{thm:error_bound_regular};
for each $\tau$, let $z_\tau \in \Vt$ denote the approximation of $u$ given by Theorem~\ref{thm:clement_bochner_approximation}; then define $z_h\coloneqq \php z_\tau \in \Vth$. The fact that $z_\tau$ is continuous on $(0,T)$ implies that $z_h$ is also continuous on $(0,T)$, so $\lld z_h \rrd=0$ for $1\leq n < N$. Let $\xi_h\coloneqq u-z_h$ and $\psi_h\coloneqq u_h-z_h$, so that $u-u_h = \xi_h - \psi_h$.
As in the proof of Theorem~\ref{thm:error_bound_regular}, it is found that $\norme{\psi_h}^2\leq\norm{\psi_h}_{h,1}^2\lesssim \sum_{i=1}^9 E_i$, where the quantities $E_i$, $1\leq i\leq 9$, are defined as before.
Note that since $\sigma_{m,0}\geq 1$ for all $m$, the bound~\eqref{eq:clement_bochner_approximation} is applicable for $j=1$ and $\ell=0$.
Therefore, the arguments from the proof of Theorem~\ref{thm:error_bound_regular} and the approximation properties of $z_\tau$ from Theorem~\ref{thm:clement_bochner_approximation} imply that
\begin{align*}
E_1& \lesssim \SIn \SK \frac{h_K^{2\overline{t}_K}}{p_K^{2 \overline{s}_K}}\norm{\p_t u}_{H^{\overline{s}_K}(K)}^2\d t
+ \sum_{n=1}^N\sum_{J_m\supset I_n} \frac{t_n^{2\varrho_{m,0}-2}}{q_n^{2\sigma_{m,0}-2}} \norm{u}_{H^{\sigma_{m,0}}(J_m;X_0)}^2,
\\ E_2&\lesssim \SIn\SK \frac{h_K^{2t_K-4}}{p_K^{2s_K-4}}\norm{u}_{H^{s_K}(K)}^2\d t + \sum_{n=1}^N\sum_{J_m\supset I_n} \frac{\tau_n^{2\varrho_{m,2}}}{q_n^{2\sigma_{m,2}}}\norm{u}_{H^{\sigma_{m,2}}(J_m;X_2)}^2,
\\ E_3&\lesssim \SIn\SK \frac{h_K^{2t_K-4}}{p_K^{2s_K-3}}\norm{u}_{H^{s_K}(K)}^2\d t + \sum_{n=1}^N \sum_{J_m\supset I_n} \frac{\tau_n^{2\varrho_{m,2}}}{q_n^{2\sigma_{m,2}}}\norm{u}_{H^{\sigma_{m,2}}(J_m;X_2)}^2,
\\ E_4& \lesssim \SIN\SK \frac{h_K^{2t_K}}{p_K^{2s_K+1}}\norm{u}_{H^{s_K}(K)}^2\d t + \sum_{n=1}^N \sum_{J_m\supset I_n}\frac{\tau_n^{2\varrho_{m,0}}}{q_n^{2\sigma_{m,0}}}\norm{u}_{H^{\sigma_{m,0}}(J_m;X_0)}^2,
\\ E_5 &\lesssim \SIn\SK\frac{h_K^{2t_K-4}}{p_K^{2s_K-5}}\norm{u}_{H^{s_K}(K)}^2  \d t
+ \max_{K\in\calT_h}p_K \sum_{n=1}^N\sum_{J_m\supset I_n}\frac{\tau_n^{2\varrho_{m,2}}}{q_n^{2\sigma_{m,2}}}\norm{u}_{H^{\sigma_{m,2}}(J_m;X_2)}^2,
\\ E_6&\lesssim \SIn\SK \frac{h_K^{2t_K-4}}{p_K^{2 s_K-7}}\norm{u}_{H^{s_K}(K)}^2 \d t
+ \max_{K\in\calT_h} p_K^3 \sum_{n=1}^N \sum_{J_m\supset I_n}\frac{\tau_n^{2\varrho_{m,2}}}{q_n^{2\sigma_{m,2}}}\norm{u}_{H^{\sigma_{m,2}}(J_m;X_2)}^2.
\end{align*}
Using inverse inequalities and $H^1$-stability of $\php$, we find that
\begin{equation*}
\begin{split}
E_7 + E_8  & = \SK \norm{u_0-\php z_\tau (0^+)}_{H^1(K)}^2 
+ \sum_{F\in\calFhib}\mu_F^{-1}\norm{\{ \nabla (u_0-\php z_\tau(0^+))\cdot  n_F \}}_{L^2(F)}^2 \\
 & \lesssim \SK\norm{u_0-\php u_0}_{H^1(K)}^2 + \norm{u_0-z_\tau(0^+)}_{H^1(\Om)}^2  \\
& \lesssim \SK \frac{h_K^{2\tilde{t}_K-2}}{p_K^{2\tilde{s}_K-2}}\norm{u_0}_{H^{\tilde{s}_K}(K)}^2 + \norm{u_0-z_\tau(0^+)}_{H^1(\Om)}^2,
\end{split}
\end{equation*}
Since $\eval{z_\tau}{I_1}\in \calQ_{q_n}(H)$, we have $z_\tau(0^+)\in H^1_0(\Om)$, so
\begin{equation}
\begin{split}
E_9  &=\sum_{F\in\calFhib}\mu_F\norm{\llb u_0-\php z_\tau(0^+)\rrb}_{L^2(F)}^2 \\
&= \sum_{F\in\calFhib}\mu_F\norm{\llb u_0-\php u_0 + \php(u_0-z_\tau(0^+))-(u_0-z_\tau(0^+))\rrb}_{L^2(F)}^2
\\ &\lesssim \SK \frac{h_K^{2\tilde{t}_K-2}}{p_K^{2\tilde{s}_K-3}}\norm{u_0}_{H^{\tilde{s}_K}(K)}^2 + \max_{K\in\calT_h}   p_K \norm{u_0 - z_\tau(0^+)}_{H^1(\Om)}^2.
\end{split}
\end{equation}
Poincar\'{e}'s Inequality and \eqref{eq:clement_bochner_approximation} then show that
\begin{multline*}
\begin{aligned}
\norm{u_0-z_\tau(0^+)}_{H^1(\Om)}^2 &\lesssim \norm{u-z_\tau}_{L^2(I_1;X_2)}\norm{u-z_\tau}_{H^1(I_1;X_0)}+\frac{1}{\tau_1}\norm{u-z_\tau}_{L^2(I_1;X_1)}^2 
\\ &\lesssim  \sum_{J_m\supset I_1} \frac{\tau_1^{2\varrho_{m,2}}}{q_1^{2\sigma_{m,2}}}\norm{u}_{H^{\sigma_{m,2}}(J_m;X_2)}^2
+\frac{\tau_1^{2\varrho_{m,0}-2}}{q_1^{2\sigma_{m,0}-2}}\norm{u}_{H^{\sigma_{m,0}}(J_m;X_0)}^2 
\end{aligned}
\\ +  \sum_{J_m\supset I_1} \frac{\tau_1^{2\varrho_{m,1}-1}}{q_1^{2\sigma_{m,1}}}\norm{u}_{H^{\sigma_{m,1}}(J_m;X_1)}^2 .
\end{multline*}
Since $\norme{\xi_h}^2\lesssim \sum_{i=1}^9 E_i$, the combination of the above bounds with the triangle inequality $\norme{u-u_h}\leq \norme{\xi_h}+\norme{\psi_h}$ completes the proof of \eqref{eq:nonsmooth_error_bound}.
\qed
\end{proof}

\section{Numerical experiments}\label{sec:numexp} 
In the first experiment, we study the performance of the method on a fully nonlinear problem with strongly anisotropic diffusion coefficients, and observe optimal convergence rates for smooth solutions. In the second experiment, we show that the scheme gives exponential convergence rates when combining $hp$-refinement and $\tau q$-refinement, even for problems with rough solutions.

\subsection{First experiment}\label{sec:first_numexp}
We examine the orders of convergence of the method for a problem with strongly anisotropic diffusion coefficients and a smooth solution.
Let $\Om=(0,1)^2$, $I=(0,1)$, let $b^\a\equiv 0$, $c^\a\equiv 0$ and let the $a^\a$ be defined by
\begin{equation}
a^\a \coloneqq \a \begin{pmatrix}
 1 & 1/40 \\ 1/40 & 1/800
 \end{pmatrix}
 \a^\top,
 \quad \a\in \Ld\coloneqq\mathrm{SO}(2),
\end{equation}
where $\mathrm{SO}(2)$ is the special orthogonal group of $2\times 2$ matrices.
For $\omega=1$, $\lambda=0$, it is found that the Cordes condition \eqref{eq:cordes_condition2} holds with $\eps\approx 1.25\times 10^{-3}$. We choose $f^\a$ so that the exact solution is $u = \left(1-\mathrm{e}^{-t}\right)\mathrm{e}^{xy}\sin(\pi x)\sin(\pi y)$.
The strong anisotropy of the diffusion coefficient in this problem implies that monotone finite difference discretisations would require very large stencils in order to achieve consistency~\cite{Bonnans2003}.

The numerical scheme \eqref{eq:numerical_scheme} is applied on a sequence of uniform meshes obtained by regular subdivision of $\Om$ into quadrilateral elements of width $h=2^{-k}$, $1\leq k \leq 5$. The corresponding time partitions $\It$ are obtained by regular subdivision of the time interval $(0,1)$ into intervals of length $\tau=2^{-k+1}$, $1\leq k \leq 5$. The finite element spaces $\Vth$ are defined using polynomials of total degree $p$ in space and degree $q=p-1$ in time, for $p\in\{2,3,4\}$.
We set the penalty parameter $\cstab=5/2$ and $\sigma=1$ in \eqref{eq:mu_eta}. The semismooth Newton method analysed in \cite{Smears2013} is used to compute the numerical solution at each timestep.

In order to study the accuracy of the method, we measure the global error in the norm $\normx{\cdot}$ defined by
\begin{equation}
\normx{v}^2 \coloneqq \SIn\SK \left[\omega^2\norm{\p_t v}_{L^2(K)}^2 + \norm{v}_{H^2(K)}^2\right]\,\d t.
\end{equation}
Figure~\ref{fig:smooth_sol_convergence} presents the global relative errors achieved by the method, where it is seen that the optimal orders of convergence $\normx{u-u_h} \simeq h^{p-1} + \tau^q$ are achieved. The relative end-time errors, naturally measured in the broken $H^1$-norm, are also presented in Figure~\ref{fig:smooth_sol_convergence}, which shows the optimal convergence rates $\norm{u(T)-u_h(T)}_{H^1\left(\Om;\calT_h\right)}\simeq h^p$.
These results show that the method can deliver high accuracy despite the strong anisotropy of the problem and the very small value of the constant $\eps$ appearing in the Cordes condition.

\newlength\figureheight
\newlength\figurewidth
\setlength\figureheight{4cm}
\setlength\figurewidth{4.5cm}
\begin{figure}[t]
\begin{center}
\includegraphics{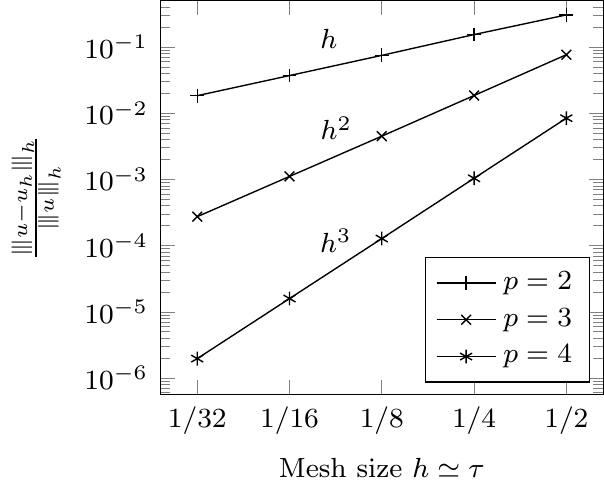}
\hfill
\includegraphics{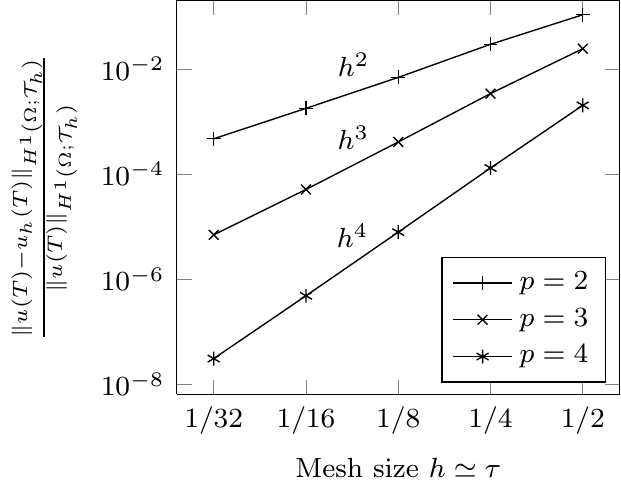}
\caption{Relative errors in approximating the solution of the problem of section~\ref{sec:first_numexp} using uniform meshes and time partitions with $\tau\simeq h$ and $p=q+1$. It is seen that the optimal convergence rates $\normx{u-u_h}\simeq h^{p-1}+\tau^q $ are achieved. The final time error, as measured in the broken $H^1$-norm, also converges with the optimal rate $\norm{u(T)-u_h(T)}_{H^1\left(\Om;\calT_h\right)}\simeq h^p$.}
\label{fig:smooth_sol_convergence}
\end{center}
\end{figure}

\subsection{Second experiment}\label{sec:numexp_2}
In section~\ref{sec:error_analysis_limited}, we considered error bounds for solutions with limited regularity. The significance of these results stems from the fact that the solutions of many parabolic HJB equations possess limited regularity as a result of early-time singularities induced by the initial datum.
This difficulty appears even in the simplest special case of the HJB equation \eqref{eq:HJB}, namely the heat equation: indeed, consider $\p_t u = \Delta u$ in $\Om \times (0,T)$, $\Om=(0,1)^2$, with homogeneous lateral boundary condition $u=0$ on $\DO\times(0,T)$ and initial datum $u_0(x,y) \coloneqq x\left(1-x\right)\sin(\pi y)$. Then, the solution is
\begin{equation}\label{eq:numexp_2}
u(x,y,t) = \frac{4}{\pi^3}\sum_{k=1}^\infty \frac{1-(-1)^k}{k^3}\exp(-(k^2+1)\,\pi^2 t)\sin(k\,\pi x)\sin(\pi y).
\end{equation}
It can be shown that for sufficiently small $t>0$ and nonnegative integers $\sigma$ and $\ell$ such that $2\sigma +\ell \geq 3$, we have $\norm{ \p_t^{\sigma} u }_{X_\ell}^2 \simeq t^{-(2\sigma+\ell-5/2)}$, with the constants of these lower and upper bounds both depending on $\sigma$ and $\ell$, but not on $t$. Therefore, $u \notin H^1(I;H)$, rather $u\in H^{7/4-\delta}(I;L^2(\Om))\cap H^{5/4-\delta}(I;H^1_0(\Om))\cap H^{3/4-\delta}(I;H)$ for arbitrarily small $\delta>0$.
It is noted that a linear problem is chosen here so that the solution may be found explicity through \eqref{eq:numexp_2}. Nevertheless, this example exhibits many features that are typical of more general parabolic problems, so that the following results remain relevant to more general HJB equations.

\begin{figure}[tbh]
\begin{center}
\begin{tabular}{c c c}
\includegraphics[width=2.5cm]{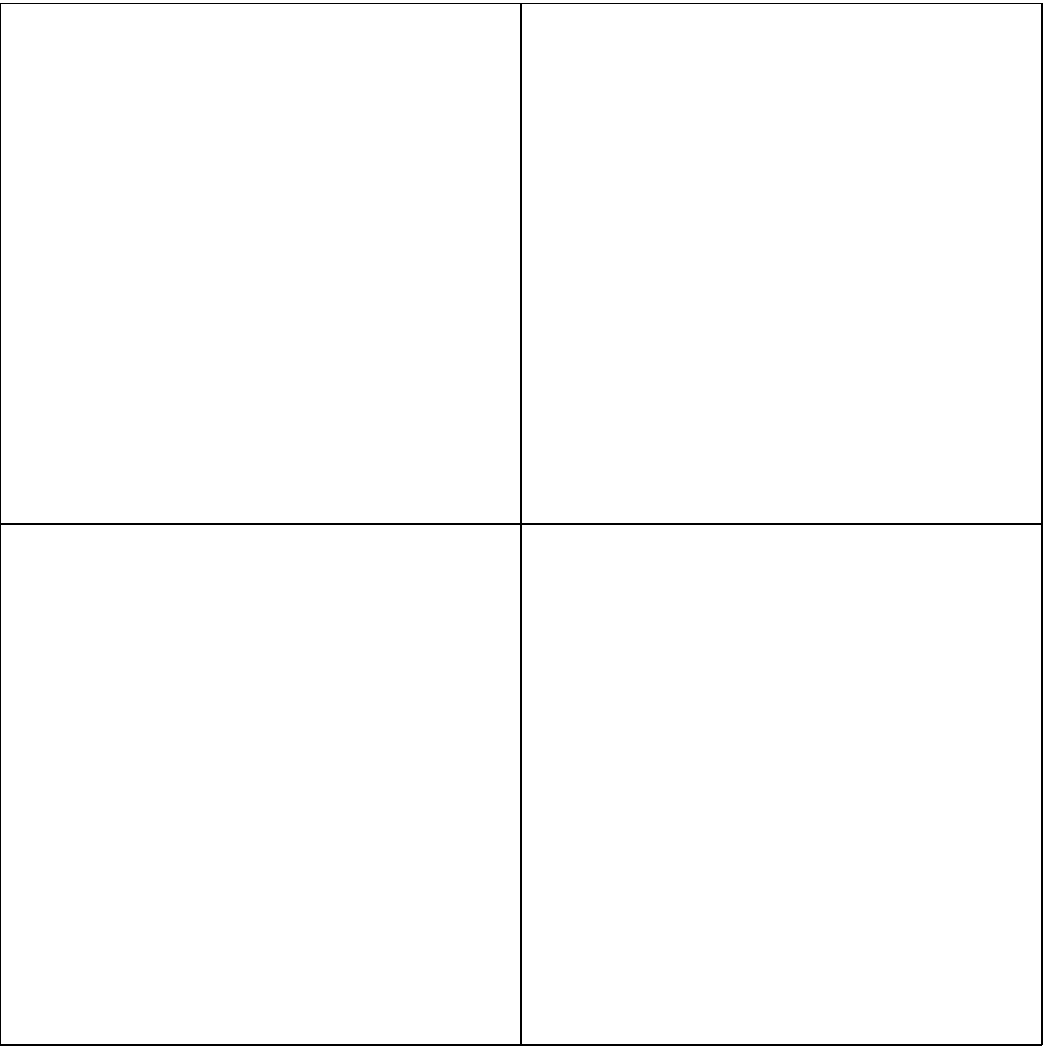}
&\includegraphics[width=2.5cm]{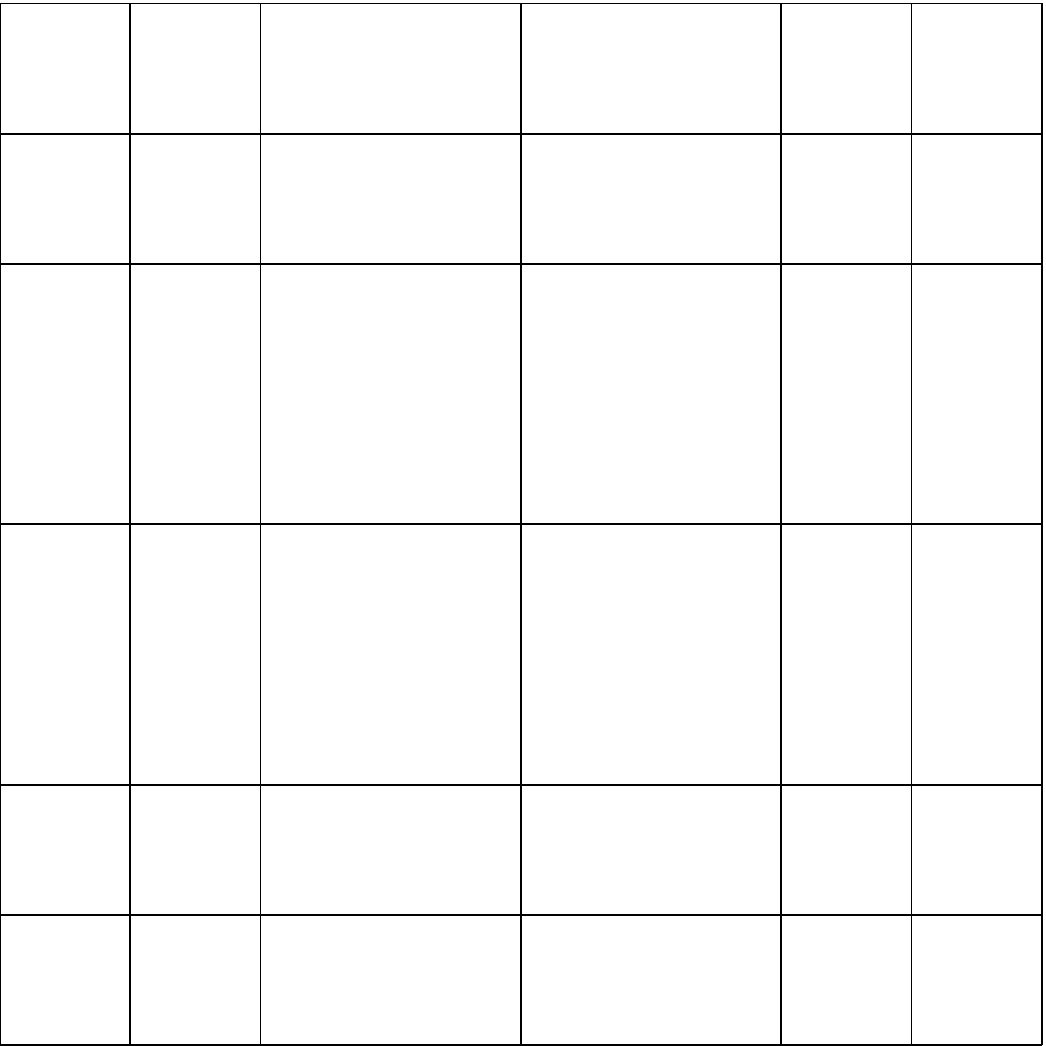}
&\includegraphics[width=2.5cm]{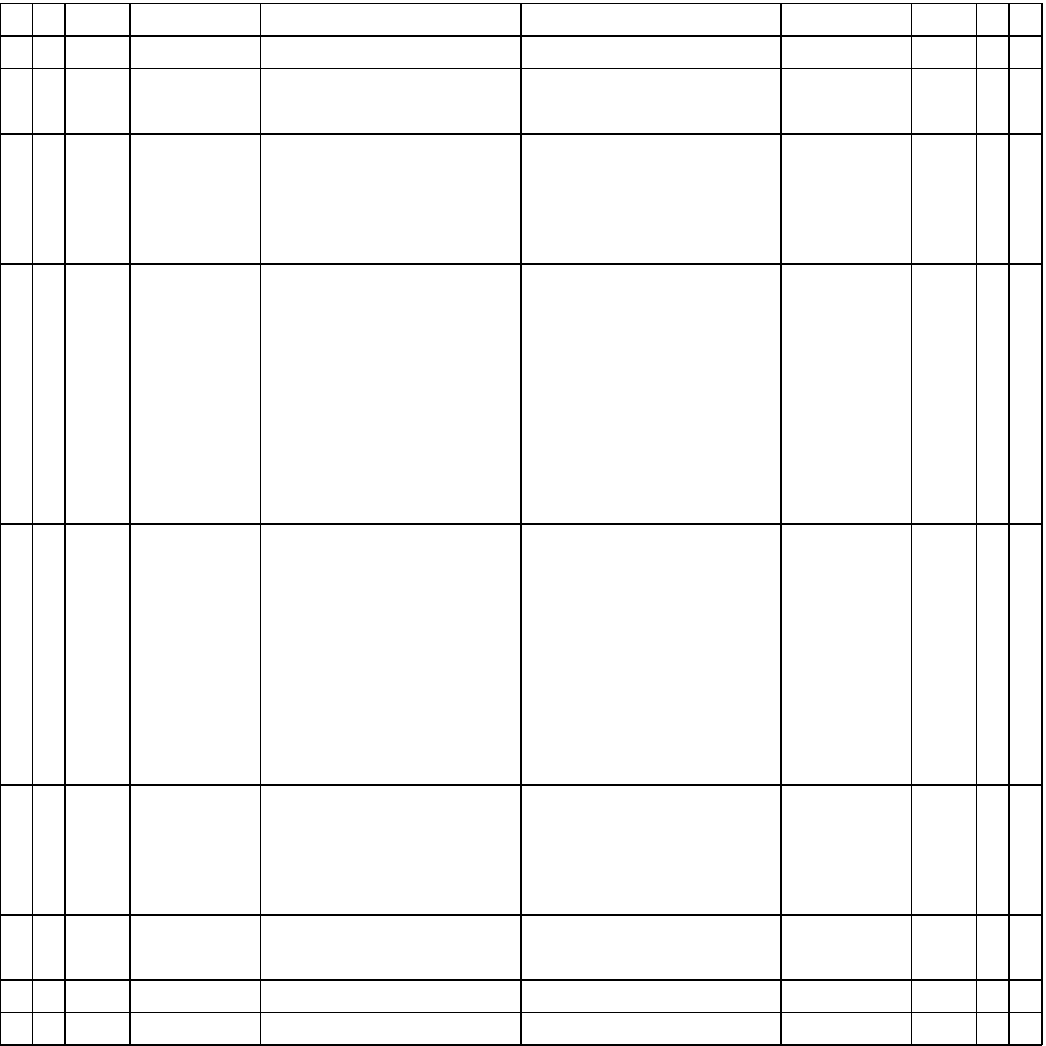}
\end{tabular}
\caption{Geometrically-graded spatial meshes used in conjunction with the geometrically-graded temporal meshes for the problem of section~\ref{sec:numexp_2}. From left to right, the meshes are those used for the first, third and fifth computations. The corresponding number of spatial degrees of freedom $\dofx$ are respectively $100$, $1128$, and $3980$.}
\label{fig:hp_meshes}
\end{center}
\setlength{\figureheight}{3.8cm}
\setlength{\figurewidth}{4.2cm}
\begin{center}
\begin{tabular}{c c}
\includegraphics{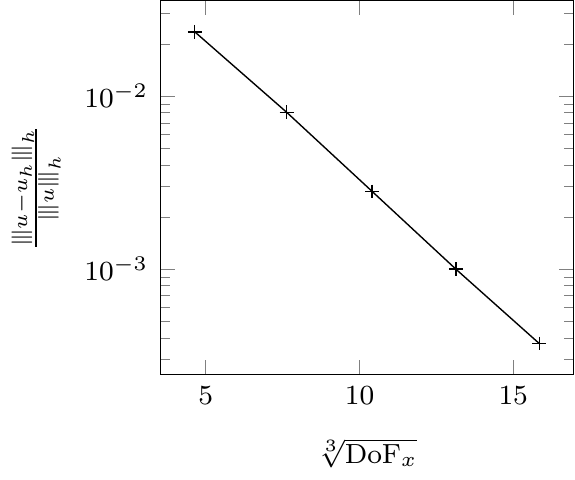} &
\includegraphics{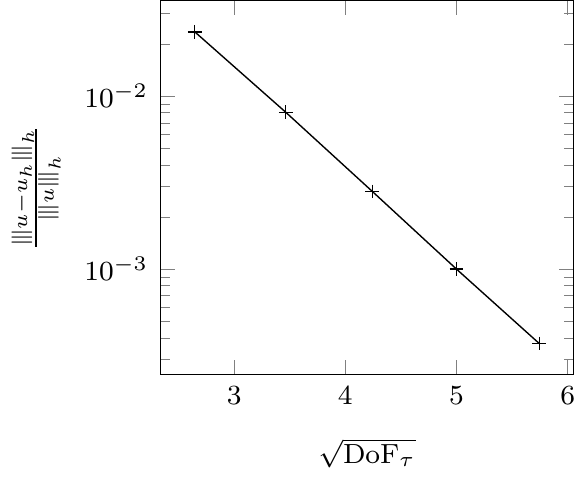}\\
\includegraphics{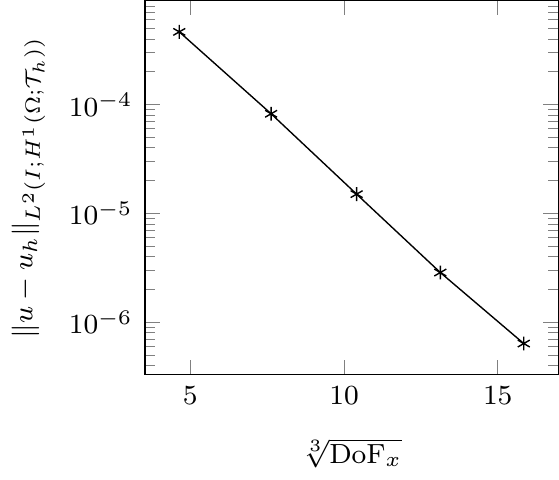} &
\includegraphics{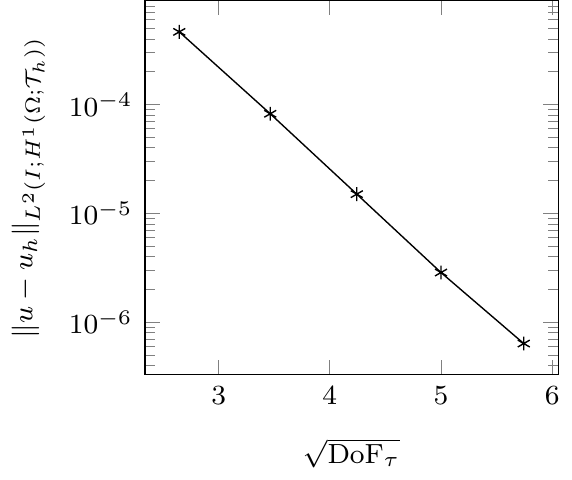}
\end{tabular}
\caption{Exponential convergence rates under $hp$-$\tau q$ refinement for the problem of section~\ref{sec:numexp_2}. The errors in the norms $\normx{\cdot}$ and $\norm{\cdot}_{L^2(I;H^1(\Om;\calT_h))}$ are plotted against $\sqrt[3]{\dofx}$ and $\sqrt{\doft}$, where $\dofx$ is the number of spatial degrees of freedom and $\doft$ is the number of temporal degrees of freedom. Exponential convergence rates of the form of \eqref{eq:exponential_convergence} are confirmed.}
\label{fig:numexp_2}
\end{center}
\end{figure}

Despite the limited regularity of the solution, accurate results can be obtained by using geometrically-graded time partitions with varying temporal polynomial degrees; see~\cite{Schotzau2000}.
Specifically, a combination of $\tau q$-refinement in time and $hp$-refinement in space can lead to a convergence rate
\begin{equation}\label{eq:exponential_convergence}
\normx{u-u_h} \lesssim \exp(- c_1 \sqrt[3]{\dofx} ) + \exp(- c_2 \sqrt{\doft} ),
\end{equation}
where $\dofx\coloneqq \Dim \Vh$, where $\doft = \sum_{n=1}^N (q_n+1)$ is the number of degrees of freedom of the temporal finite element space, and where $c_1$ and $c_2$ are positive constants.
We give here an experimental confirmation of these expectations.

The method is applied on a sequence of geometrically-graded partitions $\{\It\}_{\tau}$ constructed as follows. Let $T=0.05$, and let $t_n = \sigma^{N-n}\,T$ for $n=1,\dots,N$, for a chosen $\sigma \in (0,1)$, and $N=2,\dots,6$. As suggested in \cite{Schotzau2000}, we choose $\sigma =0.2$. The temporal polynomial degrees are linearly increasing with $n$, with $q_n \coloneqq n+1$. We choose $T$ to be small, because in practice it is natural to use $\tau q$-refinement on a small initial time segment, and then apply uniform or spectral refinement on the remaining time interval, see~\cite{Schotzau2000}. The spatial meshes are defined as follows: starting with a regular partition of $\Om$ into four quadrilateral elements, for each successive computation, we refine the meshes geometrically towards the boundary, thereby leading to the meshes given in Figure~\ref{fig:hp_meshes}. The polynomial degrees $p_K\geq 3$ are chosen to be linearly increasing away from the boundary.

Figure~\ref{fig:numexp_2} presents the resulting errors in the norms $\normx{\cdot}$ and $\norm{\cdot}_{L^2(I;H^1(\Om;\calT_h))}$, plotted against $\sqrt[3]{\dofx}$ and $\sqrt{\doft}$. It is found that the convergence rates of \eqref{eq:exponential_convergence} are attained, with higher accuracies being achieved in lower order norms. These results show the computational efficiency of the method for problems with limited regularity.

\section{Conclusion}

We have introduced and analysed a fully-discrete $hp$- and $\tau q$-version DGFEM for parabolic HJB equations with Cordes coefficients. The method is consistent and unconditionally stable, with proven convergence rates. The numerical experiments demonstrated the efficiency and accuracy of the method on problems with strongly anisotropic diffusion coefficients, and illustrated exponential convergence rates for solutions with limited regularity under $hp$- and $\tau q$-refinement.

\appendix
\normalsize
\section{Approximation theory}\label{sec:approximation}

\subsection{Trace theorem for Besov spaces}
We will show that, for a suitable domain $K\subset \R^\dim$, functions in the Besov space $B^{1/2}_{2,1}(K)$ have traces in $L^2(\p K)$.
Recall the discrete form of the J-method of interpolation of function spaces \cite{Adams2003}: a function $u\in L^2(K)$ belongs to $B^{1/2}_{2,1}(K)$ if and only if there exists a sequence $\{u_i\}_{i\in\Z}\subset H^1(K)$, such that $u=\sum_{i\in\Z} u_i$, where the series converges absolutely in $L^2(K)$, and such that the sequence $\{2^{-i/2} J(2^i,u_i)\}_{i\in\Z} \in \ell^1$, where
$J(t,v)\coloneqq\max\{\norm{v}_{L^2(K)},t\norm{v}_{H^1(K)}\}$. Moreover, we may define a norm on $B^{1/2}_{2,1}(K)$ by
\begin{equation}
\norm{u}_{B^{1/2}_{2,1}(K)}\coloneqq\inf\left\{ \norm{ \{2^{- i/2} J(2^i,u_i)\}_{i\in\Z}}_{\ell^1},\,u=\sum\nolimits_{i\in\Z} u_i,\;u_i\in H^1(K)\right\}.
\end{equation}
Also, for any such sequence, we have
\begin{equation}\label{eq:besov_convergence_sequence}
\lim_{m\tends \infty} \norm{u-\sum\nolimits_{\abs{i}\leq m} u_i}_{B^{1/2}_{2,1}(K)} \leq \lim_{m\tends\infty}\sum_{\abs{i}>m} 2^{-i/2}J(2^i,u_i) =0.
\end{equation}
Hence $H^1(K)$ is dense in $B^{1/2}_{2,1}(K)$.

It is sometimes problematic to work with the infinite series representation of a function in the Besov space $B^{1/2}_{2,1}(K)$, as a result of questions concerning convergence of the series in appropriate norms. The following lemma is a key ingredient of our proof of the Trace Theorem, and shows that it is possible to work with representations by finite sums of functions in the dense subspace $H^1(K)$.
\begin{lemma}\label{lem:besov_finite_sum}
Let $K\subset\R^\dim$ be a domain. Then, for each $u\in H^1(K)$, there exists a positive integer $m$ and a finite set $\{u_i\}_{\abs{i}\leq m}\subset H^1(K)$, with $u=\sum_{\abs{i}\leq m} u_i$, and 
\begin{equation}\label{eq:besov_finite_sum_bounded}
\sum_{\abs{i}\leq m} 2^{-i/2} J(2^i,u_i) \lesssim \norm{u}_{B^{1/2}_{2,1}(K)},
\end{equation}
where the constant is independent of all other quantities.
\end{lemma}
\begin{proof}
Since the case $u=0$ is trivial, we assume that $u\neq 0$. Since $H^1(K)$ is embedded in $B^{1/2}_{2,1}(K)$, there exists a sequence $\{v_i\}_{i\in\Z}\subset H^1(K)$ such that $u=\sum_{i\in\Z} v_i$, and such that $\norm{\{2^{-i/2} J(2^i,v_i)\}_i}_{\ell^1}\leq \sqrt{2} \norm{u}_{B^{1/2}_{2,1}(K)}$. Let $m\geq 1$ be the smallest integer such that $\norm{u}_{H^1(K)}\leq 2^{m/2} \norm{u}_{B^{1/2}_{2,1}(K)}$.
The series $\sum_{i\in\Z} v_i$ converges absolutely to $u$ in $L^2(K)$, since $\sum_{i\in\Z} \norm{v_i}_{L^2(K)} \leq \sum_{i\in\Z} 2^{-i/2} J(2^i,v_i) \leq \sqrt{2}\norm{u}_{B^{1/2}_{2,1}(K)}$.
Therefore,
\begin{align}
\label{eq:besov_l2_bound}
&\norm{u-\sum\nolimits_{\abs{i}< m} v_i}_{L^2(K)} \leq \sum\nolimits_{\abs{i}\geq m} \norm{ v_i}_{L^2(K)} 
\\ & \qquad\qquad\leq 2^{-m/2} \sum\nolimits_{\abs{i}\geq m} 2^{-i/2}J(2^i,v_i) \leq 2^{-(m-1)/2} \norm{u}_{B^{1/2}_{2,1}(K)}. \nonumber 
\\ \label{eq:besov_h1_bound}
& \sum_{\abs{i}< m} \norm{v_i}_{H^1(K)}\leq 2^{(m-1)/2} \sum_{\abs{i}< m} 2^{i/2} \norm{v_i}_{H^1(K)} \leq 2^{m/2}\norm{u}_{B^{1/2}_{2,1}(K)}.
\end{align}
Now, define $u_i\coloneqq v_i$ for $\abs{i} < m$, and $u_{-m}\coloneqq u-\sum_{\abs{i}<m}u_i$, whilst $u_i \coloneqq 0$ otherwise. By hypothesis, $u\in H^1(K)$, so $u_{-m}\in H^1(K)$, and we have $u=\sum_{\abs{i}\leq m} u_i$. It follows from \eqref{eq:besov_l2_bound} that $2^{m/2} \norm{u_{-m}}_{L^2(K)}\leq \sqrt{2}\norm{u}_{B^{1/2}_{2,1}(K)}$. 
The choice of the integer $m$ and the bound \eqref{eq:besov_h1_bound} show that
\[
2^{-m/2}\norm{u_{-m}}_{H^1(K)}
\leq 2^{-m/2}\left(\norm{u}_{H^1(K)} + \sum\nolimits_{\abs{i}< m}\norm{v_i}_{H^1(K)}\right)
\leq 2 \norm{u}_{B^{1/2}_{2,1}(K)}.
\]
Therefore, $2^{m/2}J(2^{-m},u_{-m})\lesssim \norm{u}_{B^{1/2}_{2,1}(K)}$, and we find that \eqref{eq:besov_finite_sum_bounded} holds with a constant that is independent of all other quantities, thereby showing that the set $\{u_i\}_{\abs{i}\leq m}$ fulfills all of the above claims.
\qed
\end{proof}

\begin{theorem}\label{thm:mesh_besov_trace_theorem}
Let $\Om\subset\R^\dim$ be a bounded Lipschitz polytopal domain, and let $\{\calT_h\}_h$ be a shape-regular sequence of simplicial or parallelepipedal meshes on $\Om$. Then, for each $\calT_h$ and each $K\in\calT_h$, the trace operator $\gamma\colon H^1(K)\tends L^2(\p K)$ has a unique extension to a bounded linear operator on $B^{1/2}_{2,1}(K)$, and there holds
\begin{equation}\label{eq:mesh_besov_trace}
\norm{\gamma u}_{L^2(\p K)}\lesssim  \norm{u}_{B^{1/2}_{2,1}(K)}+h_K^{-1/2} \norm{u}_{L^2(K)} \qquad\forall\, u\in B^{1/2}_{2,1}(K).
\end{equation}
\end{theorem}
\begin{proof}
For an element $K\in\calT_h$, let $\gamma\colon H^1(K)\tends L^2(\p K)$ denote the trace operator. First, we claim that
\begin{equation}\label{eq:mesh_trace_bound_1}
\norm{\gamma u}_{L^2(\p K)} \lesssim \norm{u}_{B^{1/2}_{2,1}(K)} + h_K^{-1/2}\norm{u}_{L^2(K)} \quad\forall\, u\in H^1(K).
\end{equation}
For a given $u\in H^1(K)$, Lemma~\ref{lem:besov_finite_sum} shows that there exists a finite set $\{u_i\}_{\abs{i}\leq m}\subset H^1(K)$ such that $u=\sum_{\abs{i}\leq m} u_i$, and such that \eqref{eq:besov_finite_sum_bounded} holds. Since $\{\calT_h\}_h$ is a shape-regular sequence of simplicial or parallelepipedal meshes, we have the multiplicative trace inequality (c.f.\ \cite{DiPietro2012,Monk1999})
\begin{equation}\label{eq:mesh_mult_trace_ineq}
\norm{\gamma u}_{L^2 (\p K)}\lesssim \left( \abs{u}_{H^1(K)}+h_K^{-1}\norm{u}_{L^2(K)}\right)^{1/2}\norm{u}_{L^2(K)}^{1/2}\quad\forall\,u\in H^1(K),
\end{equation}
where the constant depends only the dimension $\dim$ and the shape-regularity of $\{\calT_h\}_h$.
We remark that the multiplicative trace inequality was proven for the case of triangles in two dimensions in \cite{Monk1999}, and can be extended to simplices and parallelepipeds in $\R^\dim$, see \cite{DiPietro2012}.
Let $\overline{u}$ denote the mean-value of $u$ over $K$, and note that $\norm{u-\overline{u}}_{L^2(K)} \lesssim h_K\abs{u}_{H^1(K)}$, see \cite{Brenner2008}.
Then, $u-\overline{u}=\sum_{\abs{i}\leq m} (u_i-\overline{u}_i)$, and \eqref{eq:mesh_mult_trace_ineq} implies that
\begin{equation}
\begin{aligned}
\norm{\gamma(u-\overline{u})}_{L^2(\p K)} &\lesssim 
\sum_{\abs{i}
\leq m} \left(\abs{u_i}_{H^1(K)}+h_K^{-1}\norm{u_i-\overline{u}_i}_{L^2(K)}\right)^{1/2}\norm{u_i-\overline{u}_i}^{1/2}_{L^2(K)}
\\ &\lesssim \sum_{\abs{i}\leq m} \abs{u_i}_{H^1(K)}^{1/2}\norm{u_i}_{L^2(K)}^{1/2}
 \\ &\lesssim \sum_{\abs{i}\leq m} 2^{-i/2}\norm{u_i}_{L^2(K)}+2^{i/2}\norm{u_i}_{H^1(K)}
\\ & \lesssim \sum_{\abs{i}\leq m} 2^{-i/2} J(2^i,u_i)
\lesssim \norm{u}_{B^{1/2}_{2,1}(K)}.
\end{aligned}
\end{equation}
It is also easily found that $\norm{\gamma \overline{u}}_{L^2(\p K)} \lesssim h_K^{-1/2} \norm{u}_{L^2(K)}$. Therefore, the bound \eqref{eq:mesh_trace_bound_1} follows from the above bounds and the triangle inequality. Thus, the trace operator $\gamma$ is uniformly bounded in the norm of $B^{1/2}_{2,1}(K)$ over the space $H^1(K)$, which is densely embedded in $B^{1/2}_{2,1}(K)$. Hence, $\gamma$ has a unique extension to a bounded linear operator $\gamma \colon B^{1/2}_{2,1}(K)\tends L^2(\p K)$, and \eqref{eq:mesh_besov_trace} holds.
\qed
\end{proof}

In the following, we will often omit any explicit reference to the trace operator $\gamma$. For example, we shall write $\norm{u}_{L^2(\p K)}$ rather than $\norm{\gamma u}_{L^2(\p K)}$.


\subsection{Polynomial approximation in Sobolev spaces}%
We recall the results from \cite{Babuvska1987a}. For a positive integer $\dim$ and a nonnegative integer $p$, let $\mathcal{P}_p$ denote the space of real valued polynomials on $\R^\dim$ with either partial or total degree at most $p$.
\begin{lemma}\label{lem:polynomial_transformation}
For a nonnegative integer $p$ and $\rho\in\R_{>0}$, a function $u\colon (-\rho,\rho)\tends \R$ is an algebraic polynomial of degree at most $p$ if and only if the function $V\colon \xi \mapsto u(\rho\sin\xi)$ is a trigonometric polynomial of degree at most $p$.
\end{lemma}
\begin{proof}
Suppose that $u$ is an algebraic polynomial of degree at most $p$. Then it is easily found that $V$ is a trigonometric polynomial of degree at most $p$. To show the converse, suppose that $V$ is a trigonometric polynomial of degree at most $p$. Observe that $V$ is necessarily symmetric about $\pm \pi/2$, and thus we have, for any $k\geq 0$,
\begin{align}\label{eq:V_identities}
&\int_{-\pi}^{\pi} V(\xi) \sin (2k \xi)\, \d \xi  = 0, &\int_{-\pi}^{\pi} V(\xi) \cos( \left(2k+1\right) \xi) \, \d \xi=0.\end{align}
Indeed, the first identity in \eqref{eq:V_identities} is found by writing
\begin{equation}\label{eq:algebraic_trig_correspondence}
\begin{split}
\int_{-\pi}^{\pi} v(\xi) \sin (2k \xi)\,\d \xi
&=\int_{0}^{\pi} \left(v(\xi) - v(-\xi)\right) \sin(2k\xi)\, \d \xi \\
&= (-1)^k \int_{-\pi/2}^{\pi/2}\left(v(\tfrac{\pi}{2} + \delta ) - v(-\tfrac{\pi}{2} + \delta)\right) \sin(2k\delta) \,\d \delta,
\end{split}
\end{equation}
and by noting that the right-hand side of \eqref{eq:algebraic_trig_correspondence} is the integral of an odd function over an interval centred about $\delta=0$, as a result of the symmetry of $V$. The proof of the second identity in \eqref{eq:V_identities} is analogous.

Since $V$ is a trigonometric polynomial of degree at most $p$, it follows from \eqref{eq:V_identities} that
\[
V(\xi) = \sum_{1\leq 2k+1\leq p} a_k \sin(\left(2k +1\right)\xi) + \sum_{0\leq 2k\leq p} b_k \cos(2k \xi).
\]
For $x\in (-\rho,\rho)$ and $k\geq 0$, define $P_{2k+1}(x) \coloneqq \sin(\left(2k+1\right)\arcsin(x/\rho))$ and $ Q_{2k}(x) \coloneqq \cos(2k\arcsin(x/\rho))$.
So, for example, $Q_0(x)=1$, $P_1(x) = x$, and $Q_2(x) = 1-2 x^2$. Therefore, $u$ may be written as $u(x) = \sum_{1\leq 2k+1\leq p} a_k \,P_{2k+1}(x) + \sum_{0\leq 2k\leq p} b_k \,Q_{2k}(x)$.
The recurrence relations $P_{2k+1}(x)  = P_{2k-1}(x) + 2\, x\, Q_{2k}(x)$ and $Q_{2k+2}(x) = 2\, Q_2(x) \, Q_{2k}(x) - Q_{2k-2}(x)$, for all $k\geq 1$, allow us to deduce that $P_{2k+1} \in \calP_{2k+1}$ and that $Q_{2k}\in \calP_{2k}$ for each $k\geq 0$, where $\calP_p$ denotes here the space of univariate polynomials of degree at most $p$. It then follows that $u\in \calP_p$.
\qed
\end{proof}

\begin{theorem}\label{thm:babuska_suri}
Let $Q\subset[-1,1]^\dim$ be either the unit hypercube or the unit simplex in $\R^\dim$, $\dim\geq 1$. For each integer $p\geq0$, there exists a linear operator $\pj^p\colon L^2(Q)\tends \calP_p$, with the following properties. There is a constant $C$, independent of $p$, such that
\begin{equation}\label{eq:babuskasuri_l2_stability}
\norm{\pj^p u}_{L^2(Q)} \leq C \norm{u}_{L^2(Q)} \qquad\forall\,u\in L^2(Q).
\end{equation}
For nonnegative integers $j\leq s$, there is a constant $C$, independent of $p$ but dependent on $s$, such that
\begin{equation}\label{eq:babuskasuri_approximation}
\norm{u-\pj^p u }_{H^j(Q)} \leq C (p+1)^{-(s-j)} \norm{u}_{H^s(Q)}\qquad \forall \,u\in H^s(Q).
\end{equation}
\end{theorem}
\begin{proof}
Our proof is similar to the one given in \cite{Babuvska1987a}, except that we also show that generally $u\neq \pj^p u$, even if $u\in \calP_p$, contrary to what is claimed in \cite{Babuvska1987}.
First, we momentarily assume that $\calP_p$ denotes the space of polynomials of partial degree at most $p$.
Since $Q$ is a Lipschitz domain, the Stein Extension Theorem \cite{Adams2003} shows that there exists a linear total extension operator $E\colon L^2(Q)\tends L^2(\R^\dim)$, such that, for each nonnegative integer $s$, $\norm{E u}_{H^s(\R^\dim)}\lesssim \norm{u}_{H^s(Q)}$ for all $u\in H^s(Q)$. 
For $\rho\in \R_{>0}$, let $Q(\rho)\coloneqq [-\rho,\rho]^\dim$.
Without loss of generality, we may assume that $\supp Eu \subset Q(3/2)$ for every $u\in L^2(Q)$. Let $\Phi$ be the diffeomorphism from $Q(\pi/2)$ to $Q(2)$ defined by $\Phi(\xi)\coloneqq (2\sin \xi_1,\dots,2\sin \xi_\dim)$. 
For $u\in L^2(Q)$, let $V(\xi)\coloneqq Eu(\Phi(\xi))$ for $\xi \in \R^\dim$.
It follows that $V$ is a $2\pi$-periodic function that is symmetric about each hyperplane $\xi_i = \pm \pi/2$, i.e.\ for any $\xi\in\R^\dim$ such that $\xi_i=\pm \pi/2$ and any $\delta \in\R$, we have $V(\xi+\delta\,e_i ) = V(\xi-\delta\,e_i)$, where $e_i$ is the $i$-th unit vector.
Since $\supp Eu \subset Q(3/2)$, we may use the symmetry of $V$ to show that, for any integer $s\geq 0$ and any $u\in H^s(Q)$, we have $\norm{V}_{H^s(Q(\pi))}^2=2^\dim \norm{V}_{H^s(Q(\pi/2))}^2 = 2^\dim \norm{V}_{H^s(\Phi^{-1}(Q(3/2)))}^2$, and therefore we deduce that $\norm{V}_{H^s(Q(\pi))}\lesssim \norm{u}_{H^s(Q)}$ for all $u\in H^s(Q)$ and all integers $s\geq 0$.
The function $V$ admits the Fourier expansion $V = \sum_{k\in\Z^\dim} a_k \,\e^{\im \,k\cdot \xi}$, where the coefficients $a_k\in \C$ satisfy $ \overline{a_k} = a_{-k}$, for each $k\in\Z^\dim$, because $V$ is real-valued.
For an integer $p\geq 0$, define the trigonometric polynomial $V_p$ by $V_p(\xi) \coloneqq \sum_{\abs{k}_{\infty}\leq p} a_k\, \rm{e}^{\im \,k\cdot \xi}$. The relation $\overline{a_k}=a_{-k}$ shows that
\[
V_p(\xi) = a_0 + \sum_{\substack{k\in\N^\dim\setminus\{0\} \\ \abs{k}_\infty \leq p}} \half (a_k+\overline{a_k})(\e^{\im\, k\cdot \xi} + \e^{-\im \,k \cdot \xi}) + \half (a_k - \overline{a_k})(\e^{\im \,k\cdot \xi} - \e^{-\im\, k \cdot \xi}),
\]
thus implying that $V_p$ is real-valued. For any integers $j\leq s$, and any $u\in H^s(Q)$,
\begin{equation}\label{eq:vp_approximation}
\begin{split}
\abs{V-V_p}_{H^j(Q(\pi))}^2 &\lesssim \sum_{\abs{k}_{\infty}>p} \abs{k}_\infty^{2j} \abs{a_k}^2
\lesssim (p+1)^{-2(s-j)} \sum_{k\in\Z^\dim} \abs{k}_{\infty}^{2s} \abs{a_k}^2
\\ &\lesssim  (p+1)^{-2(s-j)}\abs{V}_{H^s(Q(\pi))}^2 \lesssim  (p+1)^{-2(s-j)}\norm{u}_{H^s(Q)}^2,
\end{split}
\end{equation}
where the constants are independent of $u$ and $p$.

Define the linear map $\pj^p\colon L^2(Q)\tends L^2_{\rm{loc}}(Q(2))$ by $\pj^p u  \coloneqq V_p\circ \Phi^{-1}$. Since the mapping $\Phi\colon Q(\pi/2) \tends Q(2)$ is a diffeomorphism, and since $Q$ is compactly contained in $Q(2)$, we find that $\norm{\pj^p u}_{L^2(Q)}^2 \lesssim \norm{V_p}_{L^2(Q(\pi/2))}^2 \leq \norm{V}_{L^2(Q(\pi))}^2 \lesssim \norm{u}_{L^2(Q)}^2$ for any $u\in L^2(Q)$, where the constants are independent of $u$ and $p$, thus giving \eqref{eq:babuskasuri_l2_stability}.
Likewise, \eqref{eq:babuskasuri_approximation} follows from \eqref{eq:vp_approximation} and from $\norm{u-\pj^p u}_{H^j(Q)} \lesssim \norm{V-V_p}_{H^j(Q(\pi))}$.

In order to show that $\pj^p u$ is a polynomial of partial degree at most $p$, it is enough to show that the univariate functions $x_i \mapsto \pj^p u(x_1,\dots,x_i,\dots,x_\dim)$ are polynomials of degree at most $p$, for each $x\in Q(2)$. However, this follows from Lemma~\ref{lem:polynomial_transformation} because the trigonometric polynomial $V_p = \pj^p u \circ \Phi$ has partial degree at most $p$.

We now show that $\pj^p$ is \emph{inexact} when applied to polynomials: in general, $u\neq \pj^p u$ is possible for $u\in \calP_p$. To show this, consider the special case where $\dim=1$ and $u\equiv 1$. Since $Eu$ is compactly supported on $Q(3/2)$ and is not identically zero, $Eu$ is necessarily not a polynomial of finite degree on $Q(2)$. Since $V(\xi)=Eu(2\sin \xi)$, Lemma~\ref{lem:polynomial_transformation} shows that $V$ is not a trigonometric polynomial of finite degree, and we also have $\norm{V-1}_{L^2(Q(\pi))}>0$. By convergence of Fourier series, there exists $p_0\geq 0$ such that for all $p\geq p_0$, we have $\norm{V-V_p}_{L^2(Q(\pi))}<\half \norm{V-1}_{L^2(Q(\pi))}$, so that
\begin{equation}\label{eq:V_p_dist_1}
\norm{V_p-1}_{L^2(Q(\pi))}>\half\norm{V-1}_{L^2(Q(\pi))}>0.
\end{equation}
Since nonzero trigonometric polynomials have at most finitely many roots, $V_p$ cannot be identically equal to $1$ on any open subset of $Q(\pi)$, because otherwise $V_p$ would have to be identically equal to $1$ on $Q(\pi)$, thereby contradicting \eqref{eq:V_p_dist_1}. Therefore, $V_p\not\equiv 1\equiv V$ on $\Phi^{-1}(Q)$, and thus $u\neq \pj^p u$ on $Q$.

Now, we consider the case where $\calP_p$ denotes the space of polynomials of total degree $p$. Since the space of polynomials of partial degree $k$ is contained in $\calP_p$ whenever $k\leq p/\dim$, we may choose $k\leq p/\dim \leq k+1$, and we find that the projector $\pj^k$ defined above has the required properties.\qed
\end{proof}

We note that the polynomial inexactness of the Babu\v{s}ka--Suri projector, as defined in \textrm{\cite{Babuvska1987,Babuvska1987a}}, is independent of the choice of the extension operator, since it results from the requirement that the extended functions have compact support. This requirement is not easily avoided, since it is used to obtain the bound $\norm{V}_{H^s(Q(\pi))}\lesssim \norm{u}_{H^s(Q)}$.

\begin{lemma}\label{lem:exact_projector_reference_element}
Let $Q\subset[-1,1]^\dim$ be either the unit hypercube or the unit simplex in $\R^\dim$, $\dim\geq 1$. For each pair of nonnegative integers $p$ and $m$, there exists a linear operator $\pj^{m,p}\colon L^2(Q)\tends \calP_{p}$, the space of polynomials with partial degree at most $p$, such that $\pj^{m,p}$ has the following properties. If $u$ is a polynomial of total degree at most $\min(m,p)$, then $\pj^{m,p} u = u$. There exists a constant $C$, independent of $p$ and $m$, such that
\begin{equation}\label{eq:exact_projector_l2stab}
\norm{\pj^{m,p}u}_{L^2(Q)} \leq C\norm{u}_{L^2(Q)} \qquad\forall\,u\in L^2(Q).
\end{equation}
For any nonnegative integer $s$, there is a constant $C$, independent of $p$ but dependent on $s$ and $m$, such that for each nonnegative integer $j\leq s$,
\begin{equation}\label{eq:exact_projector_bound}
\norm{u-\pj^{m,p}u}_{H^j(Q)}\leq C (p+1)^{-(s-j)}  \sum_{r=t}^s \abs{u}_{H^r(Q)} \quad\forall\,u\in H^s(Q),
\end{equation}
where $t\coloneqq \min(s,p+1,m+1)$.
\end{lemma}
\begin{proof}
For nonnegative integers $m$ and $p$, let $\pj^p$ be the Babu\v{s}ka--Suri projector as given by Theorem~\ref{thm:babuska_suri}, and let $\pj^{\min(m,p)}_{L^2}\colon L^2(Q)\tends \calP_{\min(m,p)}$ denote the $L^2$ projection into the space of polynomials of total degree at most $\min(m,p)$. Then, define
\begin{equation}
\pj^{m,p} u \coloneqq \pj^{\min(m,p)}_{L^2} u + \pj^{p} \left( u-\pj^{\min(m,p)}_{L^2} u\right), \quad u\in L^2(Q).
\end{equation}
It follows that $\pj^{m,p}$ is a well-defined linear operator mapping $L^2(Q)$ into $\calP_p$. Since $\pj^p$ is a linear operator, we see that $\pj^{m,p}$ is exact on the space of polynomials of total degree at most $\min(m,p)$. To show \eqref{eq:exact_projector_l2stab}, we use the triangle inequality
\begin{equation}
\norm{\pj^{m,p}u}_{L^2(Q)}\leq \norm{\pj^{\min(m,p)}_{L^2} u}_{L^2(Q)} + \norm{\pj^p}_{L^2(Q)\tends L^2(Q)}\norm{u-\pj^{\min(m,p)}_{L^2} u}_{L^2(Q)},
\end{equation}
and we note that, by \eqref{eq:babuskasuri_l2_stability}, $\norm{\pj^p}_{L^2(Q)\tends L^2(Q)}\leq C$, with $C$ independent of $p$, and that $\norm{\pj^{\min(m,p)}_{L^2}}_{L^2(Q)\tends L^2(Q)}\leq 1$. Now, let $j\leq s$ be nonnegative integers, and apply \eqref{eq:babuskasuri_approximation} to obtain
\begin{equation}\label{eq:exact_projector_bound_1}
\norm{u-\pj^{m,p}u}_{H^j(Q)}\leq C (p+1)^{-(s-j)}\norm{u-\pj^{\min(m,p)}_{L^2} u}_{H^s(Q)} \quad\forall\,u\in H^s(Q),
\end{equation}
where $C$ is independent of $p$ and $m$ but dependent on $s$. Since $Q$ is the unit simplex or unit hypercube, the Bramble--Hilbert Lemma \cite{Brenner2008} shows that
\begin{equation}\label{eq:exact_projector_bound_2}
\norm{u-\pj^{\min(m,p)}_{L^2}u}_{H^s(Q)}\leq C \sum_{r=t}^s \abs{u}_{H^r(Q)}\quad\forall\,u\in H^s(Q),
\end{equation}
where $t\coloneqq\min(s,\min(m,p)+1)$ and $C$ depends on $s$, $\min(m,p)$ and on $Q$. Moreover, by considering seperately the cases $p<m$ and $p\geq m$, it is seen that we may choose the constant in \eqref{eq:exact_projector_bound_2} to depend only on $m$, and not on $p$. We thus obtain \eqref{eq:exact_projector_bound} by combining \eqref{eq:exact_projector_bound_1} and \eqref{eq:exact_projector_bound_2}, and noting that the constant may be chosen to be independent of $p$.
\qed
\end{proof}

\paragraph{Definition of fractional order Sobolev spaces} For a domain $K$ and a real number $s>0$ such that $s \in (r,r+1)$ for a nonnegative integer $r$, we define
\begin{equation}
H^s(K)\coloneqq \left( H^r(K), H^{r+1}(K)\right)_{s-r,2;J}.
\end{equation}
Here, we use the standard norm on $H^r(K)$ when $r$ is an integer.
It follows from the Equivalence Theorem \cite{Adams2003} that $H^s(K) = \left(H^r(K), H^{r+1}(K)\right)_{s-r,2;K}$, where the constant in the equivalence of norms depends only on $s$. Also, in view of the Re-iteration Theorem, we note that 
\begin{equation}
\left(H^r(K),H^{r+1}(K)\right)_{s-r,1;J}\hookrightarrow H^s(K) \hookrightarrow \left(H^r(K),H^{r+1}(K)\right)_{s-r,\infty;K},
\end{equation}
where the embedding constants depend only on $s$, see \cite[Thm.~7.16, Cor.~7.20]{Adams2003}. We remark that it is important in the following that these constants are independent of the domain $K$.

In the following, $a \lesssim b$ for $a,b \in \R$ means that there exists a constant $C$ such that $a \leq C\,b$, where $C$ is independent of discretisation parameters, such as the element sizes of the meshes and the polynomial degrees of finite element spaces, but otherwise possibly dependent on other fixed quantities, such as the shape-regularity parameters of the mesh, for example.

\begin{theorem}\label{thm:polynomial_approximation}
Let $\Om\subset \R^\dim$ be a bounded Lipschitz polytopal domain, and let $\{\calT_h\}_h$ be a shape-regular sequence of simplicial or parallelepipedal meshes on $\Om$. For each mesh $\calT_h$, suppose that $h=\max_{K\in\calT_h}h_K$, where $h_K\coloneqq \diam K$ for all $K\in\calT_h$. For each mesh $\calT_h$, let $\bo m= (m_K;\; K\in\calT_h)$ and $\bo p=( p_K;\; K\in\calT_h)$ be vectors of nonnegative integers. Then, there exists a sequence of linear operators $\{\pmp\}_h$, such that $\pmp \colon L^2(\Om)\tends \Vh$, with $\eval{\pmp u}{K}=\eval{u}{K}$ if $\eval{u}{K}$ is a polynomial of total degree at most $\min(m_K,p_K)$, and such that, for each $K\in\calT_h$,
\begin{equation}\label{eq:pmp_l2_stab}
\norm{\pmp u}_{L^2(K)} \lesssim \norm{u}_{L^2(K)} \qquad\forall\,u\in L^2(K).
\end{equation}
Also, for each $K\in\calT_h$, $s_K \in \R_{\geq 0}$, each nonnegative integer $j\leq s_K$ and, if $s_K>1/2$, for each multi-index $\beta$, with $\abs{\beta}<s_K-1/2$, we have
\begin{align}
\norm{u-\pmp u}_{H^j(K)} &\lesssim \frac{h_K^{t_K-j}}{(p_K+1)^{s_K-j}}\norm{u}_{H^{s_K}(K)} & &\forall\, u\in H^{s_K}(K),
\label{eq:pmp_approximation_bound}
\\ \norm{D^\beta(u-\pmp u)}_{L^2(\p K)} &\lesssim \frac{h_K^{t_K-\abs{\beta}-1/2}}{(p_K+1)^{s_K-\abs{\beta}-1/2}} \norm{u}_{H^{s_K}(K)} & &\forall \,u\in H^{s_K}(K),\label{eq:pmp_trace_bound}
\end{align}
where $t_K\coloneqq \min(s_K,p_K+1,m_K+1)$.
\end{theorem}
\begin{proof}
Since the meshes $\{\calT_h\}$ consist of simplices or parallelepipeds, each element $K$ is affine-equivalent to the unit simplex or unit hypercube, with a corresponding affine mapping $F_K\colon K\tends Q$. For each $K\in\calT_h$, define $\hat{u}=u\circ F_K^{-1}$ and $\eval{\pmp u}{K}=\left(\pj^{m_K,p_K} \hat{u}\right)\circ F_K \in \calP_{p_K}$,
where $\pj^{m_K,p_K}$ is the operator given by Lemma~\ref{lem:exact_projector_reference_element}.
The stability bound \eqref{eq:pmp_l2_stab} then follows from the shape-regularity of the mesh and from the bound \eqref{eq:exact_projector_l2stab} of Lemma~\ref{lem:exact_projector_reference_element}.
Also, for any nonnegative integers $j\leq s$, we have
\begin{equation}\label{eq:pmp_bound_1}
\norm{u-\pmp u}_{H^j(K)} \lesssim \frac{h_K^{t_K-j}}{(p_K+1)^{s_K-j}} \norm{u}_{H^{s_K}(K)}\quad\forall\,u\in H^{s_K}(K),
\end{equation}
where $t_K=\min(s_K,p_K+1,m_K+1)$ and where the constant depends only on $s_K$, $m_K$, on $\max h$ the maximum mesh size over all meshes, on the reference element and on the shape-regularity of $\{\calT_h\}$.
We remark that the additional dependence on $\max h$ stems from the fact that we use the bound $h_K^{t_K-i}\leq  \max h^{j-i}\, h_K^{t_K-j}$, $i\leq j$, to obtain \eqref{eq:pmp_bound_1}.
The Exact Interpolation Theorem~\cite{Adams2003} shows that \eqref{eq:pmp_bound_1} extends to each nonnegative integer $j$ and each nonnegative real number $s_K$ such that $j\leq s_K$, thus giving \eqref{eq:pmp_approximation_bound}.

We now show \eqref{eq:pmp_trace_bound}. Let $s_K>1/2$ and $\beta$ be a multi-index with $\abs{\beta}<s_K-1/2$. First, consider the case where $\abs{\beta}\leq s_K-1$. Then, \eqref{eq:pmp_trace_bound} follows from \eqref{eq:pmp_approximation_bound} and from the multiplicative trace inequality \eqref{eq:mesh_mult_trace_ineq}.
Now, consider the case where $s_K-\abs{\beta}\in (\half,1)$. Theorem~\ref{thm:mesh_besov_trace_theorem} shows that, for any $u\in H^{s_K}(K)$,
\[
\norm{D^{\beta}(u-\pmp u)}_{L^2(\p K)}\lesssim \norm{D^{\beta}( u-\pmp u)}_{B^{1/2}_{2,1}(K)}+h_K^{-1/2} \norm{u-\pmp u}_{H^{\abs{\beta}}(K)}.
\]
Given \eqref{eq:pmp_approximation_bound} for the case $j=\abs{\beta}$, we can obtain \eqref{eq:pmp_trace_bound} provided that we can show that, for any $u\in H^{s_K}(K)$,
\begin{equation}\label{eq:pmp_bound_2}
\norm{D^{\beta}(u-\pmp u)}_{B^{1/2}_{2,1}(K)}\lesssim \frac{h_K^{t_K-\abs{\beta}-1/2}}{(p_K+1)^{s_K-\abs{\beta}-1/2}} \norm{u}_{H^{s_K}(K)}.
\end{equation}
The Exact Interpolation Theorem and \eqref{eq:pmp_bound_1} show that $\norm{u-\pmp u}_{H^{s_K}(K)}\lesssim \norm{u}_{H^{s_K}(K)}$ for any $u\in H^{s_K}(K)$.
The Re-iteration Theorem \cite{Adams2003} shows that
\[
B^{1/2}_{2,1}(K)=\left(L^2(K),H^{s_K-\abs{\beta}}(K)\right)_{\lambda,1;J},
\]
where $\lambda \coloneqq \tfrac{1}{2(s_K-\abs{\beta})}$, and where the constant in the equivalence of norms depends only on $s_K-\abs{\beta}$. Therefore, for any $u\in H^{s_K}(K)$, there holds
\[
\norm{D^\beta(u-\pmp u)}_{B^{1/2}_{2,1}(K)} \lesssim \left( \frac{h_K^{t_K-\abs{\beta}}}{(p_K+1)^{s_K-\abs{\beta}}}\right)^{1-\lambda} \norm{u}_{H^{s_K}(K)}.
\]
Since $t_K\leq s_K$, we have $(t_K-\abs{\beta})(1-\lambda)\geq t_K-\abs{\beta}-1/2$, and therefore we deduce \eqref{eq:pmp_bound_2} and \eqref{eq:pmp_trace_bound}.
\qed
\end{proof}

\subsection{Polynomial approximation in Bochner spaces}\label{sec:bochner_approximation}
To simplify the notation in the following approximation results, let the spaces $\{\Xell\}_{\ell=0}^2$ be defined by
\begin{equation*}
\begin{aligned}
X_0&\coloneqq L^2(\Om),
& X_1&\coloneqq H^1_0(\Om),
& X_2 &\coloneqq H=H^2(\Om)\cap H^1_0(\Om).
\end{aligned}
\end{equation*}
The approximation theory for Sobolev spaces can be extended to Bochner spaces as follows.
\begin{lemma}\label{lem:bochner_eigenfunction_decomp}
Let $I$ be an open interval and let $\Om\subset\R^\dim$ be a bounded convex domain. Let $\{\psi_k\}_{k=1}^\infty \subset H\coloneqq H^2(\Om)\cap H^1_0(\Om)$ be an orthonormal basis of $L^2(\Om)$, such that $\{\psi_k\}_{k=1}^\infty$ is also an orthogonal basis of $H^1_0(\Om)$ and of $H$, which satisfies
\[
\begin{aligned}
\int_\Om \psi_k\,\psi_j \,\d x &= \delta_{kj},
&\int_\Om \nabla \psi_k \cdot \nabla \psi_j\,\d x & = \lambda_k \,\delta_{kj},
&\int_\Om \Delta \psi_k \, \Delta \psi_j \, \d x & = \lambda_k^2\, \delta_{kj},
\end{aligned}
\]
where $\lambda_k > 0$ for eack $k\in \N$. Then, for any $\ell\in\{0,1,2\}$, and any $u\in L^2(I;\Xell)$, we have $u=\sum_{k=1}^{\infty} u_k \,\psi_k$, where $u_k(t)\coloneqq  \pair{u(t)}{\psi_k}_{L^2(\Om)}$, and where the series converges in $L^2(I;\Xell)$.
For any integer $s\geq 0$, any $u\in H^s(I;\Xell)$, we have the generalised Parseval Identity
\begin{equation}\label{eq:bochner_eigenfunction_norms}
\abs{u}_{H^s(I;\Xell)}^2 = \sum_{k=1}^\infty \lambda_{k}^\ell \,\abs{u_k}_{H^s(I)}^2.
\end{equation}
\end{lemma}

\begin{proof}
Let $\ell\in\{0,1,2\}$ and let the function $u\in L^2(I;\Xell)$. Then, $u_k$ defined above is a measurable real-valued function, and $\norm{u_k(t)}_{L^2(I)}\leq \norm{u}_{L^2(I;X_0)}$ for each $k\in \N$. For each $m\in\N$, define the function $v_m\in L^2(I;X_2)$ by
$v_m\coloneqq \sum_{k=1}^m u_k\,\psi_k$. Then, orthogonality of the $\{\psi_k\}_{k=1}^{\infty}$ in $\Xell$ implies the Bessel Inequality $\sum_{k=1}^m \lambda_{k}^\ell\, \norm{u_k}_{L^2(I)}^2= \norm{v_m}_{L^2(I;\Xell)}^2 \leq \norm{u}_{L^2(I;\Xell)}^2$.
It can then be shown that $\{v_m\}_{m=1}^{\infty}$ is a Cauchy sequence in $L^2(I;\Xell)$, with limit denoted by $v$. Moreover, there exists a subsequence of $\{v_m\}_{m=1}^{\infty}$ which converges to $v$ in $\Xell$ pointwise almost everywhere on $I$. Thus, it follows from the definition of the functions $v_m$ that $\pair{v(t)}{\psi_k}_{L^2(\Om)}=u_k(t)=\pair{u(t)}{\psi_k}_{L^2(\Om)}$ for each $k\in \N$, for a.e.\ $t\in I$, which shows that $v=u$, since $\{\psi_k\}_{k=1}^\infty$ is an orthonormal basis of $L^2(\Om)$. This proves that $u=\sum_{k=1}^{\infty} u_k \,\psi_k$ and shows Parseval's Identity \eqref{eq:bochner_eigenfunction_norms} for the case $s=0$.
Now, let $s\geq 1$ be an integer, and suppose $u\in H^s(I;\Xell)$ for some $\ell\in\{0,1,2\}$. Let $\phi\in C^{\infty}_0(I)$, and compute $\int_I u_k \,\p_t^{s} \phi \, \d t =  \int_{I} \pair{u}{\p_t^{s}(\phi\,\psi_k)}_{L^2(\Om)}\,\d t= (-1)^s\int_{I}\pair{\p_t^{s} u}{\psi_k}_{L^2(\Om)} \phi\,\d t$.
Therefore, the weak derivative $\p_t^s u_k $ exists in $L^2(I)$ and $\p_t^{s} u_k = \pair{ \p_t^{s} u}{\psi_k}_{L^2(\Om)}$. So, the generalised Parseval Identity \eqref{eq:bochner_eigenfunction_norms} for integer $s\geq 1$ is found by applying \eqref{eq:bochner_eigenfunction_norms} for $s=0$ to the function $\p_t^s u$.\qed
\end{proof}

Recall that for a Banach space $X$ and a nonnegative integer $q$, the space of univariate $X$-valued polynomials of degree at most $q$ is denoted by $\calQ_q(X)$.

\begin{lemma}\label{lem:general_bochner_approximation}
Let $\Om\subset \R^\dim$ be a bounded convex domain, let $I$ be an open interval of length $\tau_0$, and let $r$ and $q$ be nonnegative integers. Then, for each open interval $J\subset I$ of length $\tau\leq \tau_0$, there exists a linear operator $\ptau$ defined on $L^2(J;L^2(\Om))$ with the following properties. The operator $\ptau \colon L^2(J;\Xell)\tends \calQ_q(\Xell)$ for each $\ell\in\{0,1,2\}$, with $\ptau u=u$ if $u\in \calQ_{\min(r,q)}(\Xell)$. Furthermore,
\begin{equation}\label{eq:bochner_approx_l2_stab}
\norm{\ptau u}_{L^2(J;\Xell)} \lesssim \norm{u}_{L^2(J;\Xell)} \quad\forall\,u\in L^2(J;\Xell),
\end{equation}
where the constant is independent of all quantities.
For any real $\sigma\geq 0$ and any nonnegative integer $j\leq \sigma$,
\begin{equation}\label{eq:bochner_approx_bound}
\abs{u-\ptau u}_{H^j(J;\Xell)}\lesssim \frac{\tau^{\varrho - j}}{(q+1)^{\sigma-j}} \norm{u}_{H^{\sigma}(J;\Xell)}\quad\forall\,u\in H^{\sigma}(J;\Xell),
\end{equation}
where $\varrho\coloneqq \min(\sigma,r+1,q+1)$, and where the constant depends only on $\tau_0$, $\sigma$ and $r$.
\end{lemma}
\begin{proof}
Let $u\in L^2(J;L^2(\Om))$ and define $u_k$, $k\in\N$, as in Lemma~\ref{lem:bochner_eigenfunction_decomp}. Let $F$ denote the affine mapping from the reference element $(-1,1)$ to $J$. Then, for each $k\in\N$, define the univariate real-valued polynomial $\ptau u_k\coloneqq \left(\pj^{r,q} \hat{u}_{k}\right)\circ F^{-1}$, where $\hat{u}_{k}\coloneqq u_k\circ F$, and where $\pj^{r,q}$ is the approximation operator on the reference element given by Lemma~\ref{lem:exact_projector_reference_element} for $\dim=1$.
For each $k\in\N$, $\ptau u_k$ has degree at most $q$.
It follows from Lemma~\ref{lem:exact_projector_reference_element} that $\norm{\ptau u_k}_{L^2(J)} \lesssim \norm{u_k}_{L^2(J)}$, where the constant is independent of all other quantities.
Therefore, Lemma~\ref{lem:bochner_eigenfunction_decomp} implies that $\ptau u\coloneqq \sum_{k=1}^\infty \ptau u_k\,\psi_k $ is well-defined in $L^2(J,L^2(\Om))$.
Furthermore,  if $u\in L^2(J;\Xell)$ for some $\ell\in \{0,1,2\}$, then Lemma~\ref{lem:bochner_eigenfunction_decomp} shows that
\[
\norm{\ptau u}_{L^2(J;X)}^2=\sum_{k=1}^\infty \lambda_{k}^\ell \norm{ \ptau u_k}_{L^2(J)}^2 \lesssim \norm{u}_{L^2(J;\Xell)}^2,
\]
where the constant is independent of all quantities, thereby showing \eqref{eq:bochner_approx_l2_stab}. This also implies that $\prq \colon L^2(J;\Xell)\tends \calQ_q(\Xell)$ for each $\ell\in \{0,1,2\}$.
Moreover, if $u\in \calQ_{\min(r,q)}(\Xell)$, then $\ptau u_k = u_k$ for each $k\in\N$ by Lemma~\ref{lem:exact_projector_reference_element}, which implies that $\ptau u=u $ by Lemma~\ref{lem:bochner_eigenfunction_decomp}.

Let $j\leq \sigma$ be nonnegative integers and let $u\in H^\sigma(J;\Xell)$ for some $\ell\in\{0,1,2\}$. Then, Lemmas~\ref{lem:exact_projector_reference_element}~and~\ref{lem:bochner_eigenfunction_decomp} imply that \begin{multline}\label{eq:bochner_approx_bound_1}
\abs{u-\ptau u}_{H^{j}(J;\Xell)}^2 = \sum_{k=1}^{\infty} \lambda_{k}^\ell\,\abs{u_k-\ptau u_k}_{H^{j}(J)}^2
\\  \lesssim \sum_{\nu=\varrho}^{\sigma} \frac{\tau^{2(\nu-j)}}{(q+1)^{2(\sigma-j)}}  \sum_{k=1}^{\infty} \lambda_{k}^\ell\, \abs{u_k}_{H^\nu(J)}^2
 \lesssim \frac{\tau^{2(\varrho-j)} \max\left(1,\tau_0^{2(\sigma-\varrho)}\right)}{(q+1)^{2(\sigma-j)}} \sum_{\nu=\varrho}^{\sigma}  \abs{u}_{H^\nu(J;\Xell)}^2,
\end{multline}
where the constant depends only on $\sigma$ and $r$, thereby giving the bound \eqref{eq:bochner_approx_bound} for the case where $\sigma$ is an integer. Therefore, the bound \eqref{eq:bochner_approx_bound} for general $\sigma\in\R_{\geq 0}$ follows from \eqref{eq:bochner_approx_bound_1} and the theory of interpolation of function spaces.
\qed
\end{proof}

\begin{footnotesize}

\end{footnotesize}

\begin{thebibliography}{10}

\bibitem{Adams2003}
Adams, R.A., Fournier, J.F.: Sobolev spaces, \emph{Pure and Applied Mathematics}, vol. 140, second edition.
\newblock Elsevier (2003)

\bibitem{Akrivis2004}
Akrivis, G., Makridakis, C.: Galerkin time-stepping methods for nonlinear parabolic equations.
\newblock M2AN Math. Model. Numer. Anal. \textbf{38}(2), 261--289 (2004).

\bibitem{Babuvska1987}
Babu{\v{s}}ka, I., Suri, M.: The {$h$}-{$p$} version of the finite element
  method with quasi-uniform meshes.
\newblock RAIRO Mod\'el. Math. Anal. Num\'er. \textbf{21}(2), 199--238 (1987).

\bibitem{Babuvska1987a}
Babu{\v{s}}ka, I., Suri, M.: The optimal convergence rate of the {$p$}-version
  of the finite element method.
\newblock SIAM J. Numer. Anal. \textbf{24}(4), 750--776 (1987).

\bibitem{Barles1991}
Barles, G., Souganidis, P.: Convergence of approximation schemes for fully
  nonlinear second-order equations.
\newblock Asymptotic Anal. \textbf{4}(3), 271--283 (1991).

\bibitem{Bonnans2003}
Bonnans, J.F., Zidani, H.: Consistency of generalized finite difference schemes for the stochastic {HJB} equation.
\newblock SIAM J. Numer. Anal. \textbf{41}(3), 1008--1021 (2003).

\bibitem{Brenner2008}
Brenner, S.C., Scott, L.R.: The mathematical theory of finite element methods, \emph{Texts in Applied Mathematics}, vol.~15, third edition.
\newblock Springer, New York (2008).

\bibitem{Camilli1995}
Camilli, F., Falcone, M.: An approximation scheme for the optimal control of diffusion processes.
\newblock RAIRO Mod\'el. Math. Anal. Num\'er. \textbf{29}(1), 97--122 (1995).

\bibitem{Cordes1956}
Cordes, H.O.: \"{U}ber die erste {R}andwertaufgabe bei quasilinearen differentialgleichungen zweiter ordnung in mehr als zwei variablen.
\newblock Math. Ann. \textbf{131}, 278--312 (1956).

\bibitem{Crandall1996}
Crandall, M.G., Lions, P.L.: Convergent difference schemes for nonlinear parabolic equations and mean curvature motion.
\newblock Numer. Math. \textbf{75}(1), 17--41 (1996).

\bibitem{Debrabant2013}
Debrabant, K., Jakobsen, E.R.: Semi-{L}agrangian schemes for linear and fully
  nonlinear diffusion equations.
\newblock Math. Comp. \textbf{82}(283), 1433--1462 (2013).

\bibitem{DiPietro2012}
Di~Pietro, D.A., Ern, A.: Mathematical aspects of discontinuous {G}alerkin
  methods, \emph{Math\'ematiques \& Applications (Berlin)}, vol.~69.
\newblock Springer, Heidelberg (2012).

\bibitem{Fleming2006}
Fleming, W.H., Soner, H.M.: Controlled {M}arkov processes and viscosity solutions, \emph{Stochastic Modelling and Applied Probability}, vol.~25, second edition.
\newblock Springer, New York (2006).

\bibitem{Gilbarg2001}
Gilbarg, D., Trudinger, N.S.: Elliptic partial differential equations of second order.
\newblock Classics in Mathematics. Springer-Verlag, Berlin (2001).

\bibitem{Grisvard2011}
Grisvard, P.: Elliptic problems in nonsmooth domains, \emph{Classics in Applied
  Mathematics}, vol.~69.
\newblock SIAM, Philadelphia (2011).

\bibitem{Jensen2013}
Jensen, M., Smears, I.: On the convergence of finite element methods for
  {H}amilton--{J}acobi--{B}ellman equations.
\newblock SIAM J. Numer. Anal. \textbf{51}(1), 137--162 (2013).

\bibitem{Kushner1990}
Kushner, H.J.: Numerical methods for stochastic control problems in continuous time.
\newblock SIAM J. Control Optim. \textbf{28}(5), 999--1048 (1990).

\bibitem{Maugeri2000}
Maugeri, A., Palagachev, D.K., Softova, L.G.: Elliptic and parabolic equations with discontinuous coefficients, \emph{Mathematical Research}, vol. 109.
\newblock Wiley-VCH Verlag Berlin GmbH, Berlin (2000).


\bibitem{Monk1999}
Monk, P., S{\"u}li, E.: The adaptive computation of far-field patterns by a posteriori error estimation of linear functionals.
\newblock SIAM J. Numer. Anal. \textbf{36}(1), 251--274 (1999).

\bibitem{Mozolevski2007}
Mozolevski, I., S{\"u}li, E., B{\"o}sing, P.R.: {$hp$}-version a priori error analysis of interior penalty discontinuous {G}alerkin finite element approximations to the biharmonic equation.
\newblock J. Sci. Comput. \textbf{30}(3), 465--491 (2007).

\bibitem{Renardy2004}
Renardy, M., Rogers, R.C.: An introduction to partial differential equations,
  \emph{Texts in Applied Mathematics}, vol.~13, second edition.
\newblock Springer-Verlag, New York (2004).

\bibitem{Schotzau2000}
Sch{\"o}tzau, D., Schwab, C.: Time discretization of parabolic problems by the {$hp$}-version of the discontinuous {G}alerkin finite element method.
\newblock SIAM J.~Numer.~Anal.\ \textbf{38}(3), 837--875 (2000).

\bibitem{Smears2013}
Smears, I., S{\"u}li, E.: Discontinuous {G}alerkin finite element approximation of nondivergence form elliptic equations with {C}ordes coefficients.
\newblock SIAM J.~Numer.~Anal.\ \textbf{51}, 2088--2106 (2013).

\bibitem{Smears2014}
Smears, I., S{\"u}li, E.: Discontinuous {G}alerkin finite element approximation
  of {H}amilton--{J}acobi--{B}ellman equations with {C}ordes coefficients.
\newblock SIAM J. Numer. Anal. \textbf{52}(2), 993--1016 (2014).

\bibitem{Thomee2006}
Thom{\'e}e, V.: Galerkin finite element methods for parabolic problems, \emph{Springer Series in Computational Mathematics}, vol.~25, second edition.
\newblock Springer-Verlag, Berlin (2006).

\bibitem{Wloka1987}
Wloka, J.: Partial differential equations.
\newblock Cambridge University Press, Cambridge (1987).

\end{thebibliography}
\end{document}